\documentclass[12pt,a4paper]{article}
\usepackage[utf8]{inputenc}
\usepackage[english]{babel}
\usepackage{amsmath}
\usepackage{amsfonts}
\usepackage{amssymb}
\usepackage{color}

\usepackage{comment}

\usepackage[left=2cm,right=2cm,top=2cm,bottom=2cm]{geometry}
\title{Analysis of { Boundary-Domain Integral Equations to the Mixed BVP for a Compressible Stokes System with Variable Viscosity} 
}
\author{ }
\newcommand{\bs}[1]{\boldsymbol{#1}}
\newcommand{\supp}{\text{supp}}
\renewcommand{\div}{{\rm div\,}}
\newenvironment{proof}{\paragraph{Proof:}}{\hfill$\square$}

\newtheorem{theorem}{Theorem}[section]

\newtheorem{rem}[theorem]{Remark}
\newtheorem{lemma}[theorem]{Lemma}
\newtheorem{corollary}[theorem]{Corollary}
\newtheorem{cor}[theorem]{Corollary}

\numberwithin{equation}{section}

\allowdisplaybreaks
\begin{document}
\maketitle

\begin{center}
S.E. MIKHAILOV, C. F. PORTILLO\footnote{Corresponding author}
\end{center}
\begin{abstract}
 The mixed boundary value problem for a compressible Stokes system of partial differential equations in a bounded domain is reduced to two different systems of segregated direct Boundary Integral Equations (BDIEs)
expressed in terms of surface and volume parametrix-based potential type operators. 
Equivalence of the BDIE systems to the mixed BVP and invertibility of the matrix operators associated with the BDIE systems are proved in appropriate Sobolev spaces.
\end{abstract}
\section{Introduction}

Boundary integral equations and the hydrodynamic potential theory for the Stokes PDE system with constant viscosity have been extensively studied { in numerous publications, cf.}, e.g., \cite{ladynes,lions,hsiao,reidinger,steinbach,kohr1, wenlandzhu}. 
{ The reduction of different boundary value problems for the Stokes system to boundary integral equations in the case of constant viscosity was possible since the fundamental solutions for both, velocity and pressure, are readily available in an explicit form.
Such reduction was used not only to analyse the properties of the Stokes system and BVP solutions, but also to solve BVPs by solving numerically the corresponding boundary integral equations.}

{ In this paper we consider the stationary Stokes PDE system with variable viscosity and compressibility, in a bounded domain that models the motion of a laminar compressible viscous fluid, e.g., through a variable temperature field that makes both, viscosity and compressibility depending on coordinates.
Reduction of the BVPs for the Stokes system with arbitrarily variable viscosity to explicit
boundary integral equations is usually not possible, since the
fundamental solution needed for such reduction is generally not
available in an analytical form (except for some special
dependence of the viscosity on coordinates).
Using a {\it parametrix} ({\it Levi function}) as a substitute of a fundamental solution, in the spirit of \cite{Levi1909}, \cite{Hilbert1912}, it is possible however to reduce such a BVPs to some systems of Boundary-Domain Integral Equations, BDIEs, (cf. e.g.  \cite[Sect. 18]{miranda}, \cite{Pomp1998b, Pomp1998a}, where the
Dirichlet, Neumann and Robin problems for some PDEs were reduced
to {\em indirect} BDIEs)

We will extend here the approach developed in \cite{CMN-1, MikEABEM2002} for a scalar variable-coefficient PDE, and will reduce the mixed boundary value problem for a {\it compressible} Stokes system of partial differential equations to two different systems of segregated direct BDIEs expressed in terms of surface and volume parametrix-based potential type operators.
A parametrix for a { given} PDE (or PDE system) is not unique and a special care will be taken to chose a parametrix that leads to te BDIE systems simple enough to be analysed. 
The mapping properties of the parametrix-based hydrodynamic surface and volume potentials will be obtained and the equivalence and invertibility theorems for the operators associated with the BDIE systems will be proved. 

Some preliminary results in this direction were obtained in \cite{carlos1}, where we derived BDIE systems for the mixed {\it incompressible} Stokes problem in a bounded domain and equivalence between the BVP and BDIE systems was shown, however, invertibility results were not given there.

Note that the paper is mainly aimed not at the mixed boundary-value problem for the Stokes system, which properties are well-known nowadays, but rather at analysis of the BDIE systems per se. The analysis is interesting not only in its own rights but is also to pave the way for studying the corresponding localised BDIEs and analysing convergence and stability of BDIE-based numerical methods for PDEs,
 cf., e.g., \cite{MikEABEM2002, CMN-LocJIEA2009, GMR2013, MikNakJEM, MikMoh2012, SSA2000, SSZ2005, Taigbenu1999, ZZA1998, ZZA1999}.
}

\section{Preliminaries}
Let $\Omega=\Omega^{+}\subset\mathbb R^3$ be a \textit{bounded} and simply{-}connected domain and let $\Omega^{-}:=\mathbb{R}^{3}\smallsetminus \overline{\Omega}^{+}$. 
We will assume that the boundary $\partial\Omega$ is simply{-}connected, closed and infinitely differentiable{.}
Furthermore, ${\partial\Omega} :=\overline{{\partial\Omega}}_{N}\cup \overline{{\partial\Omega}}_{D}$ where both ${\partial\Omega}_{N}$ and ${\partial\Omega}_{D}$ are non-empty, connected disjoint { sub}manifolds of ${\partial\Omega}${, and the interface between} these two submanifolds is also infinitely differentiable{.} 
 
Let $\boldsymbol{v}$ be the velocity vector field;  $p$ the pressure scalar field and $\mu\in \mathcal{C}^{\infty}(\Omega)$ be the variable kinematic viscosity of the fluid such that  $\mu(\boldsymbol{x})>c>0$. 
{ For an arbitrary couple $( p, \boldsymbol{v})$ the stress tensor operator, $\sigma_{ij}$, and the Stokes operator, $\mathcal{A}_{j}$, for a compressible fluid are defined as}
\begin{align}
\sigma_{ji}( p, \boldsymbol{v})(\boldsymbol{x}):&= -\delta_{i}^{j}p(\boldsymbol{x}) + \mu(\boldsymbol{x})\left(\dfrac{\partial v_{i}(\boldsymbol{x})}{\partial x_{j}} + \dfrac{\partial v_{j}(\boldsymbol{x})}{\partial x_{i}}
-\frac{2}{3}\delta_{i}^{j}\div\boldsymbol{v}(\boldsymbol{x})\right),
\label{sigmadef}\\
\mathcal{A}_{j}( p, \boldsymbol{v})(\boldsymbol{x} ):&= 
\frac{\partial}{\partial x_{i}}\sigma_{ji}( p, \boldsymbol{v})(\boldsymbol{x})=\nonumber\\
\frac{\partial}{\partial x_{i}}&\left(\mu(\boldsymbol{x})
\left(\frac{\partial v_{j}(\boldsymbol{x})}{\partial x_{i}} + \frac{\partial v_{i}(\boldsymbol{x})}{\partial x_{j}}
-\frac{2}{3}\delta_{i}^{j} \div\boldsymbol{v}(\boldsymbol{x})\right)\right)
-\frac{\partial p(\boldsymbol{x})}{\partial x_{j}},\quad
j,i\in \lbrace 1,2,3\rbrace,\label{ch2operatorA}
\end{align}
where $\delta_{i}^j$ is the Kronecker symbol. 
{ Henceforth} we assume the Einstein summation in repeated indices from 1 to 3.
{ We denote the Stokes operator as $\boldsymbol{\mathcal{A}}=\{\mathcal{A}_{j}\}_{j=1}^3$ and $\boldsymbol{\mathring{\mathcal{A}}}:=\boldsymbol{\mathcal{A}}|_{\mu=1}$. 
 We will also} use the following notation for derivative operators: $\partial_{j}=\partial_{x_{j}}:=\dfrac{\partial}{\partial x_{j}}$ with $j=1,2,3$; $\nabla:= (\partial_{1}, \partial_{2}, \partial_{3})$. 

In what follows $ H^s(\Omega)$, $H^{s}({\partial\Omega})$ are the
Bessel potential spaces, where $s\in \mathbb R$ is an arbitrary real
number (see, e.g., \cite{lions}, \cite{mclean}). We recall that $H^s$
coincide with the Sobolev--Slobodetski spaces $W^s_2$ for any
non-negative $s$. Let $H^{s}_{K}:= \lbrace g \in H^{1}(\mathbb{R}^{3}): \supp(g)\subseteq K\rbrace$ where $K$ is a compact subset of $\mathbb{R}^{3}$. In what follows we use the bold notation: $\boldsymbol{H}^{s}(\Omega) = [H^{s}(\Omega)]^{3}$ for 3-dimensional vector spaces.
We { denote} 
$
\widetilde{\boldsymbol{H}}^s (\Omega):=\{\bs{g}:\;\bs{g}\in \bs{H}^s  (\mathbb R^3),\; {\rm supp} \,\bs{g}
\subset\overline{\Omega}\}
$ 
{ and,} similarly, 
$\widetilde{\boldsymbol{H}}^{s}(S_{1}){:}=\lbrace \bs{g}\in \bs{H}^{s}({\partial\Omega}),\ {\rm supp}\,\bs{g}\subset\overline{S}_{1}\rbrace${.}

We will also make use of the following space (cf. e.g. \cite{costabel} \cite{CMN-1}){,}
\begin{align}
\label{Hs0def}
\bs{H}^{ s,0}(\Omega;\boldsymbol{\mathcal{A}})&:= 
\lbrace ( p, \boldsymbol{v})\in { H^{s-1}}(\Omega)\times \boldsymbol{H}^{ s}(\Omega):  \boldsymbol{\mathcal{A}}(p,\bs{v})\in \boldsymbol{L}_{2}(\Omega)\rbrace,
\end{align} 
endowed with the norm \begin{align*}
\parallel ( p, \boldsymbol{v}) \parallel_{\bs{H}^{1,0}(\Omega; L)}&:=
\left(\parallel p \parallel^{2}_{{ H^{s-1}}(\Omega)}
+\parallel \boldsymbol{v} \parallel^{2}_{\boldsymbol{H}^{ s}(\Omega)}+\parallel  \boldsymbol{\mathcal{A}}(p,\bs{v}) \parallel^{2}_{\boldsymbol{L}_{2}(\Omega)}\right)^{1/2}.
\end{align*}
{ \begin{rem}\label{R1}
Note that $\bs{H}^{ s,0}(\Omega;\boldsymbol{\mathcal{A}})=\bs{H}^{ s,0}(\Omega;\boldsymbol{\mathring{\mathcal{A}}})$ if $s\ge 1$.

Indeed, ${\mathcal A}_j(p,\bs{v})=\mu{\mathring{\mathcal A}_j}(p,\bs{v})+ B_j(p,\bs{v})$, where
$$
B_j(p,\bs{v}):=
\frac{\partial \mu(\boldsymbol{x})}{\partial x_{i}}
\left(\frac{\partial v_{j}(\boldsymbol{x})}{\partial x_{i}} + \frac{\partial v_{i}(\boldsymbol{x})}{\partial x_{j}}
-\frac{2}{3}\delta_{i}^{j} \div\boldsymbol{v}(\boldsymbol{x})\right)\in \boldsymbol{L}_{2}(\Omega)
$$
if $\bs{v}\in \boldsymbol{H}^{s}(\Omega)$ and $s\ge 1$.
\end{rem}
Similar to \cite[Theorem 3.12]{MikJMAA2011} one can prove the following assertion.
\begin{theorem}\label{densT} The space
$\mathcal D(\overline \Omega)\times\boldsymbol{\mathcal D}(\overline \Omega)$ is dense in 
$\bs{H}^{1,0}(\Omega;  \boldsymbol{\mathcal{A}})$. 
\end{theorem}}

The operator $\bs{\mathcal{A}}$ acting on $(p, \bs{v})$ is well defined in the weak sense provided $\mu(\bs{x})\in L^{\infty}(\Omega)$ as
  \[\left\langle \bs{\mathcal{A}}(p, \bs{v}),\bs{u}\right\rangle_{\Omega} := -\mathcal{E}((p, \bs{v}),\bs{u}), \quad \quad \forall \bs{u}\in \widetilde{\bs{H}}^{1}(\Omega),\]  
where the form $ \mathcal{E}: \left[ L^{2}(\Omega)\times\bs{H}^{1}(\Omega)\right] \times \widetilde{\bs{H}}^{1}(\Omega)\to \mathbb{R}$ is defined as 
\begin{equation}\label{ch2mathcalE}
\mathcal{E}\left((p,\bs{v}),\boldsymbol{u}\right) :=\int_{\Omega} \, E\left((p,\bs{v}),\boldsymbol{u}\right)(\boldsymbol{x})\, dx,
\end{equation}
and the function $E\left((p,\bs{v}),\boldsymbol{u}\right)$ is defined as 
\begin{align} 
E\left((p,\bs{v}),\boldsymbol{u}\right)(\boldsymbol{x}):&=\
\dfrac{1}{2}\mu(\bs{x})\left(\frac{\partial u_{i}(\boldsymbol{x})}{\partial x_{j}} + \frac{\partial u_{j}(\boldsymbol{x})}{\partial x_{i}}\right)\left(\frac{\partial v_{i}(\boldsymbol{x})}{\partial x_{j}} + \frac{\partial v_{j}(\boldsymbol{x})}{\partial x_{i}}\right)\nonumber\\
&\quad -\frac{2}{3}\mu(\bs{x}){\div}\boldsymbol{v}(\boldsymbol{x})\,\div \boldsymbol{u}(\boldsymbol{x})
-p(\boldsymbol{x})\div\boldsymbol{u}(\boldsymbol{x})\label{ch2exE}.
\end{align}

For sufficiently smooth functions $(p,\bs{v})\in \bs{H}^{s-1}(\Omega^\pm)\times H^{s}(\Omega^\pm)$ with $s>3/2$,  we can define the classical traction operators, $\bs{T}^{c\pm}=\{T^{c\pm}_{i}\}_{i=1}^3$, on the boundary ${\partial\Omega}$ as
\begin{align}
{ T^{c\pm}_{i}}( p, &\boldsymbol{v})(\boldsymbol{x})
:=[\gamma^{\pm}\sigma_{ij}( p, \boldsymbol{v})(\boldsymbol{x})]n_{j}(\boldsymbol{x})\nonumber\\
&= -n_{i}(\boldsymbol{x})\gamma^{\pm}p(\boldsymbol{x}) + n_j(\boldsymbol{x})\mu(\boldsymbol{x})\gamma^{\pm}\left(\dfrac{\partial v_{i}(\boldsymbol{x})}{\partial x_{j}} + \dfrac{\partial v_{j}(\boldsymbol{x})}{\partial x_{i}}
-\frac{2}{3}\delta_{i}^{j}\div\boldsymbol{v}(\boldsymbol{x})\right),\
\boldsymbol{x}\in\partial\Omega,
\label{ch2Tcl}
\end{align}
where $n_{j}(\boldsymbol{x})$ denote components of the unit outward normal vector $\boldsymbol{n}(\boldsymbol{x})$ to the boundary ${\partial\Omega}$ of the domain $\Omega$ and $\gamma^{\pm}(\,\cdot \,)$ denote the trace operators from inside and outside $\Omega$. { We will sometimes write $\gamma u$ if $\gamma u^+=\gamma u^-$, and similarly for $\bs{T}^c$, etc.}

Traction operators \eqref{ch2Tcl} can be continuously extended to the {\em canonical} traction operators $\boldsymbol{T}^{\pm}: \bs{H}^{1,0}(\Omega^\pm; \boldsymbol{\mathcal{A}}) \to \boldsymbol{H}^{-1/2}({\partial\Omega})$
defined in the weak form \cite[Section 34.1]{carlos1}  similar to {\cite[Lemma 3.2]{costabel}, \cite[Definition 3.8]{MikJMAA2011}, \cite[Definition 2.10]{GKMW2017}} as
\begin{align*}
\langle \boldsymbol{T}^{\pm}( p, \boldsymbol{v}) , \boldsymbol{w}\rangle_{{\partial\Omega}}:=
\pm \int_{\Omega^{\pm}}\left[ \boldsymbol{\mathcal{A}}( p, \boldsymbol{v}) \boldsymbol{\gamma}^{-1}\boldsymbol{w} + E\left((p,\bs{v}),\boldsymbol{\gamma}^{-1}\boldsymbol{w}\right)\right]\, dx,\\ 
\forall\,  ( p, \boldsymbol{v})\in \bs{H}^{1,0}(\Omega^\pm; \boldsymbol{\mathcal{A}}),\ 
 \forall\, \boldsymbol{w}\in \boldsymbol{H}^{1/2}({\partial\Omega}).
\end{align*}
Here the operator 
$\boldsymbol{\gamma}^{-1}:\boldsymbol{H}^{1/2}({\partial\Omega})\to 
\boldsymbol{H}^{1}(\mathbb{R}^{3})$ 
denotes a continuous right inverse of the trace operator 
$\boldsymbol{\gamma}: \boldsymbol{H}^{1}(\mathbb{R}^{3})\to\boldsymbol{H}^{1/2}({\partial\Omega})$.

Furthermore, if $(p,\bs{v})\in \bs{H}^{1,0}(\Omega ; \boldsymbol{\mathcal{A}})$ and $\boldsymbol{u}\in \boldsymbol{H}^{1}(\Omega)$, the following first Green identity holds, \cite[Eq. (34.2)]{carlos1}, cf. also { \cite[Lemma 3.4(i)]{costabel}, \cite[Theorem 3.9]{MikJMAA2011},    \cite[Lemma 2.11]{GKMW2017}}
\begin{equation}\label{ch2GF1}
\langle \boldsymbol{T}^{+}(p,\bs{v}) ,\boldsymbol{\gamma}^{+}\boldsymbol{u}\rangle_{{\partial\Omega}}=\displaystyle\int_{\Omega}[ \boldsymbol{\mathcal{A}}(p,\bs{v})\boldsymbol{u} + E\left((p,\bs{v}),\boldsymbol{u}\right)(\boldsymbol{x})] dx.
\end{equation}

Applying { identity} \eqref{ch2GF1} to the pairs $( p, \boldsymbol{v}), (q,\boldsymbol{u}) \in \bs{H}^{1,0}(\Omega ;\boldsymbol{\mathcal{A}})$ with exchanged roles and subtracting the one from the other, we arrive at the second Green identity, cf. { \cite[Lemma 3.4(ii)]{costabel}, \cite[Eq. 4.8]{MikJMAA2011},    \cite[Lemma 2.11]{GKMW2017}},
\begin{align}\label{ch2secondgreen}
&\int_{\Omega}\left[ \mathcal{A}_{j}( p, \boldsymbol{v})u_{j} - \mathcal{A}_{j}(q, \boldsymbol{u})v_{j} + q\, \div\boldsymbol{v} - p\, \div\boldsymbol{u}\right]\,\, dx =\nonumber\\
& \langle \boldsymbol{T}^{+}(p,\bs{v}) ,\boldsymbol{\gamma}^{+}\boldsymbol{u}\rangle_{{\partial\Omega}} - \langle \boldsymbol{T}^{+}(q,\bs{u}) ,\boldsymbol{\gamma}^{+}\boldsymbol{v}\rangle_{{\partial\Omega}}. 
\end{align}
Now we are ready to define the { following} mixed BVP for which we aim to derive equivalent { BDIE systems} and investigate the existence and uniqueness of their solutions.

{\em For $\boldsymbol{f}\in \boldsymbol{L}_{2}(\Omega)$, $g\in L^{2}(\Omega)$, $\boldsymbol{\varphi}_{0}\in \boldsymbol{H}^{1/2}({\partial\Omega}_{D})$ and $\boldsymbol{\psi}_{0}\in \boldsymbol{H}^{-1/2}({\partial\Omega}_{N})$, find $( p, \boldsymbol{v})\in \bs{H}^{1,0}(\Omega,\boldsymbol{\mathcal{A}})$  such that:}
\begin{subequations}
\label{ch2BVPM}
\begin{align}
\label{ch2BVP1}
             \boldsymbol{\mathcal{A}}(p,\bs{v})(\boldsymbol{x})&=\boldsymbol{f}(\boldsymbol{x}),\hspace{0.5em} \boldsymbol{x}\in\Omega,\\
             \div\boldsymbol{v}(\boldsymbol{x})&=g(\boldsymbol{x}),\hspace{0.5em} \boldsymbol{x}\in\Omega,\label{ch2BVPdiv}\\
 \label{ch2BVPD}     r_{{\partial\Omega}_{D}}\boldsymbol{\gamma}^{+}\boldsymbol{v}(\boldsymbol{x})&=\boldsymbol{\boldsymbol{\boldsymbol{\varphi}}}_{0}(\boldsymbol{x}),
 \hspace{0.3em} \boldsymbol{x}\in {\partial\Omega}_{D},\\
\label{ch2BVPN}      r_{{\partial\Omega}_{N}}\boldsymbol{T}^{+}( p, \boldsymbol{v})(\boldsymbol{x})&=\boldsymbol{\boldsymbol{\boldsymbol{\psi}}}_{0}(\boldsymbol{x}), \hspace{0.3em} \boldsymbol{x}\in {\partial\Omega}_{N}.           
\end{align}
\end{subequations}
Applying the first Green identity it is easy to prove the following uniqueness result.
\begin{theorem}\label{ch2BVPUS}
Mixed BVP \eqref{ch2BVPM} has at most one solution in the space $\bs{H}^{1,0}(\Omega ;\boldsymbol{\mathcal{A}})$.
\end{theorem} 
  
\begin{proof}
Let us suppose that there are two possible solutions: $(p_{1}, \bs{v}_{1})$ and $(p_{2}, \bs{v}_{2})$ belonging to the space $\bs{H}^{1,0}(\Omega ;\boldsymbol{\mathcal{A}})$, that satisfy the BVP \eqref{ch2BVPM}. Then, the pair $(p, \bs{v}):=(p_{2}, \bs{v}_{2})-(p_{1}, \bs{v}_{1})$ also belongs to the space $\bs{H}^{1,0}(\Omega ;\boldsymbol{\mathcal{A}})$ and satisfies the following homogeneous mixed BVP
\begin{subequations}
\label{ch2BVPMh}
\begin{align}
\label{ch2BVP1h}
             \boldsymbol{\mathcal{A}}(p,\bs{v})(\boldsymbol{x})&=\boldsymbol{0},\hspace{0.5em} \boldsymbol{x}\in\Omega,\\
             \text{div}(\boldsymbol{v})(\boldsymbol{x})&=0,\hspace{0.5em} \boldsymbol{x}\in\Omega,\label{ch2BVPdivh}\\
 \label{ch2BVPDh}     r_{{\partial\Omega}_{D}}\boldsymbol{\gamma}^{+}\boldsymbol{v}(\boldsymbol{x})&=\boldsymbol{0},
 \hspace{0.3em} \boldsymbol{x}\in {\partial\Omega}_{D},\\
\label{ch2BVPNh}      r_{{\partial\Omega}_{N}}\boldsymbol{T}^{+}( p, \boldsymbol{v})(\boldsymbol{x})&=\boldsymbol{0}, \hspace{0.3em} \boldsymbol{x}\in {\partial\Omega}_{N}.           
\end{align}
\end{subequations}

{ Applying the first Green identity \eqref{ch2GF1} to $(p, \bs{v})$ and $\bs{u}=\bs{v}$ and taking into account \eqref{ch2BVPMh}, we obtain,}
 \[  \int_{\Omega} \
\dfrac{1}{2}\mu(\bs{x})\left(\frac{\partial v_{i}(\boldsymbol{x})}{\partial x_{j}}+\frac{\partial v_{j}(\boldsymbol{x})}{\partial x_{i}}\right)^{2} dx\,=\, 0.\] 
As $\mu(\bs{x})>0$, the only possibility is that $\bs{v}(\bs{x})= \bs{a}+\bs{b}\times \bs{x}$, i.e., $\bs{v}$ is a rigid movement, \cite[Lemma 10.5]{mclean}. Nevertheless, taking into account the Dirichlet condition \eqref{ch2BVPDh}, we deduce that $\bs{v}\equiv \bs{0}$. Hence, $\bs{v}_{1}=\bs{v}_{2}$. 

Considering now $\bs{v}\equiv \bs{0}$ and keeping in mind the Neumann-traction condition \eqref{ch2BVPNh}, it is easy to conclude that $p_{1}=p_{2}$. 
\end{proof} 

\section{Parametrix and Remainder}
When $\mu(\boldsymbol{x})=1$, the operator $\boldsymbol{\mathcal{A}}$ becomes the constant-coefficient Stokes operator $\boldsymbol{\mathring{\mathcal{A}}}$, for  which we know an explicit fundamental solution defined by the pair of { functions} $( \mathring{q}^{k} , \mathring{\boldsymbol{u}}^{k} ),$ where { summation in $k$ is not assumed,} $\mathring{u}_{j}^{k}$ represent components of the incompressible velocity fundamental solution  and $\mathring{q}^{k}$ represent the components of the pressure fundamental solution (see e.g. \cite{ladynes}, \cite{kohr1},  \cite{hsiao}).
\begin{align}
\label{F-q}
\mathring{q}^{k}(\boldsymbol{x},\boldsymbol{y})&
={-}\frac{(x_{k}-y_{k})}{4\pi\vert \boldsymbol{x}-\boldsymbol{y}\vert^{3}}
=\frac{\partial}{\partial x_k}\left(\frac{1}{4\pi\vert \boldsymbol{x} -\boldsymbol{y}\vert}\right),\\
\label{F-u}
\mathring{u}_{j}^{k}(\boldsymbol{x},\boldsymbol{y})&=-\frac{1}{8\pi}\left\lbrace \dfrac{\delta_{j}^{k}}{\vert \boldsymbol{x} - \boldsymbol{y}\vert}+\dfrac{(x_{j}-y_{j})(x_{k}-y_{k})}{\vert \boldsymbol{x} - \boldsymbol{y}\vert^{3}}\right\rbrace, \hspace{0.5em}j,k\in \lbrace 1,2,3\rbrace.
\end{align}
Therefore, { the couple} $(\mathring{q}^{k}, \mathring{\boldsymbol{u}}^{k})$ { satisfies
\begin{align}
\label{dF-q}
\frac{\partial}{\partial x_k}\mathring{q}^{k}(\boldsymbol{x},\boldsymbol{y})
&=\sum_{i=1}^{3}\frac{\partial^2}{\partial x_k^2}\left(\frac{1}{4\pi\vert \boldsymbol{x} -\boldsymbol{y}\vert}\right)
=-\delta(\boldsymbol{x}-\boldsymbol{y}),\\
\mathring{\mathcal{A}}_{j}(\boldsymbol{x})(\mathring{q}^{k}(\boldsymbol{x},\boldsymbol{y}), \mathring{\boldsymbol{u}}^{k} (\boldsymbol{x},\boldsymbol{y}))
&= \sum_{i=1}^{3}\dfrac{\partial^{2} \mathring{u}_{j}^{k}(\boldsymbol{x},\boldsymbol{y})}{\partial x_{i}^{2}}  
- \dfrac{\partial \mathring{q}^{k}(\boldsymbol{x},\boldsymbol{y})}{\partial x_{j} }  = \delta_{j}^{k}\delta(\boldsymbol{x}-\boldsymbol{y}),\quad
\div_{\boldsymbol{x}}\mathring{\boldsymbol{u}}^{k}(\boldsymbol{x},\boldsymbol{y})=0.
\end{align}
}

Let us denote $\mathring{\sigma}_{ij}( p, \boldsymbol{v}):=\sigma_{ij}( p, \boldsymbol{v})\vert_{\mu=1}$,
$\mathring{T}^c_{i}( p, \boldsymbol{v}):=T^c_{i}( p, \boldsymbol{v})\vert_{\mu=1}$.
{ Then by \eqref{sigmadef}
the stress tensor of the fundamental solution}
reads as  
\[\mathring{\sigma}_{ij}{(\boldsymbol{x})(\mathring{q}^{k}(\boldsymbol{x},\boldsymbol{y}),\mathring{\boldsymbol{u}}^{k} (\boldsymbol{x},\boldsymbol{y}))}
= \frac{3}{4\pi}\frac{(x_{i}-y_{i})(x_{j}-y_{j})(x_{k}-y_{k})}{\vert \boldsymbol{x}-\boldsymbol{y}\vert^{5}},\]
and the { classical} boundary traction { of the fundamental solution} becomes
\begin{align*}
{\mathring{T}^c_{i}(\boldsymbol{x}) (\mathring{q}^{k}(\boldsymbol{x},\boldsymbol{y}) ,\mathring{\boldsymbol{u}}^{k} (\boldsymbol{x}, \boldsymbol{y}))}:
&= \mathring{\sigma}_{ij}{(\boldsymbol{x}) (\mathring{q}^{k}(\boldsymbol{x},\boldsymbol{y}) , \mathring{\boldsymbol{u}}^{k}(\boldsymbol{x},\boldsymbol{y}))}\,n_{j}(\boldsymbol{x})\\
&= \frac{3}{4\pi}\frac{(x_{i}-y_{i})(x_{j}-y_{j})(x_{k}-y_{k})}{\vert \boldsymbol{x}-\boldsymbol{y}\vert^{5}}\,n_{j}(\boldsymbol{x}).
\end{align*}

Let us define a pair of functions $(q^{k}, \boldsymbol{u}^{k})_{k=1}^3${,} 
\begin{align}
q^{k}(\boldsymbol{x},\boldsymbol{y})&
=\frac{\mu(\boldsymbol{x})}{\mu(\boldsymbol{y})}\mathring{q}^{k}(\boldsymbol{x},\boldsymbol{y})
={-}\frac{\mu(\boldsymbol{x})}{\mu(\boldsymbol{y})}\dfrac{x_{k}-y_{k}}{4\pi\vert \boldsymbol{x} - \boldsymbol{y}\vert^{3}}, \hspace{0.5em}
j,k\in \lbrace 1,2,3\rbrace.
\label{ch2PRq}\\
u_{j}^{k}(\boldsymbol{x},\boldsymbol{y})&=
\frac{1}{\mu(\boldsymbol{y})}\mathring{u}_{j}^{k}(\boldsymbol{x},\boldsymbol{y}) =-\frac{1}{8\pi\mu(\boldsymbol{y})}\left\lbrace \frac{\delta_{j}^{k}}{\vert \boldsymbol{x} - \boldsymbol{y}\vert}+\frac{(x_{j}-y_{j})(x_{k}-y_{k})}{\vert \boldsymbol{x} - \boldsymbol{y}\vert^{3}}\right\rbrace, \label{ch2PRu}
\end{align}
{ Then by \eqref{sigmadef},}
\begin{align}
&\sigma_{ij}{(\boldsymbol{x}) ({q}^{k}(\boldsymbol{x},\boldsymbol{y}) , {\boldsymbol{u}}^{k}(\boldsymbol{x},\boldsymbol{y}))}=
\frac{\mu(\boldsymbol{x})}{\mu(\boldsymbol{y})}\mathring{\sigma}_{ij}{(\boldsymbol{x}) (\mathring{q}^{k}(\boldsymbol{x},\boldsymbol{y}) , \mathring{\boldsymbol{u}}^{k}(\boldsymbol{x},\boldsymbol{y}))},
\label{squ}\\
&{T}_{i}{(\boldsymbol{x})({q}^{k}(\boldsymbol{x},\boldsymbol{y}), {\boldsymbol{u}}^{k}(\boldsymbol{x},\boldsymbol{y}))}:= 
\sigma_{ij}{(\boldsymbol{x})({q}^{k}(\boldsymbol{x},\boldsymbol{y}) , {\boldsymbol{u}}^{k}(\boldsymbol{x},\boldsymbol{y}))}\,n_{j}(\boldsymbol{x})=
\frac{\mu(\boldsymbol{x})}{\mu(\boldsymbol{y})}\mathring{T}_{i}{(\boldsymbol{x}) (\mathring{q}^{k}(\boldsymbol{x},\boldsymbol{y}) , \mathring{\boldsymbol{u}}^{k}(\boldsymbol{x},\boldsymbol{y}))}.
\label{Tqu}
\end{align}

Substituting \eqref{ch2PRq}-\eqref{ch2PRu} in the Stokes system with variable coefficient, \eqref{ch2operatorA} gives
\begin{equation}\label{ch2param}
\mathcal{A}_{j}{(\boldsymbol{x})({q}^{k}(\boldsymbol{x},\boldsymbol{y}) , {\boldsymbol{u}}^{k}(\boldsymbol{x},\boldsymbol{y}))}=
\delta_{j}^{k}\delta(\boldsymbol{x}-\boldsymbol{y})+R_{kj}(\boldsymbol{x},\boldsymbol{y}),
\end{equation}
where
\begin{align}\label{R}
R_{kj}(\boldsymbol{x},\boldsymbol{y})&=
\frac{1}{\mu(\boldsymbol{y})}\frac{\partial \mu(\boldsymbol{x})}{\partial x_{i}}
\mathring{\sigma}_{ij}{(\boldsymbol{x}) (\mathring{q}^{k}(\boldsymbol{x},\boldsymbol{y}) , \mathring{\boldsymbol{u}}^{k}(\boldsymbol{x},\boldsymbol{y}))}\nonumber\\
&{=\frac{3}{4\pi\mu(\boldsymbol{y})}\frac{\partial \mu(\boldsymbol{x})}{\partial x_{i}}\frac{(x_{i}-y_{i})(x_{j}-y_{j})(x_{k}-y_{k})}{\vert \boldsymbol{x}-\boldsymbol{y}\vert^{5}}}
=\mathcal{O}(\vert \boldsymbol{x}-\boldsymbol{y}\vert)^{-2})
\end{align}
is a  weakly singular remainder.
This implies that
 $(q^{k}, \boldsymbol{u}^{k})$ is a parametrix of the operator  $\boldsymbol{\mathcal{A}}$.
 
{ Note that the parametrix is generally not unique (cf. \cite{carlos2} for BDIEs based on an alternative parametrix for a scalar PDE). The possibility to factor out $\frac{\mu(\boldsymbol{x})}{\mu(\boldsymbol{y})}$ in \eqref{squ}-\eqref{Tqu} and $\frac{\nabla\mu(\boldsymbol{x})}{\mu(\boldsymbol{y})}$ in \eqref{R} is due to the careful choice of the parametrix in form \eqref{ch2PRq}-\eqref{ch2PRu} and this essentially simplifies the analysis of obtained parametrix-based potentials and BDIE systems further on.} 

\section{ Parametrix-based volume and surface potentials}\label{SectionHP}

Let $\rho$ and $\boldsymbol{\rho}$ be sufficiently smooth scalar and vector function on $\overline{\Omega}$, e.g., ${\rho}\in{\mathcal D}(\overline{\Omega})$,
$\boldsymbol{\rho}\in\boldsymbol{\mathcal D}(\overline{\Omega})$. 
Let us define the parametrix-based Newton-type  and  remainder vector potentials { for the velocity,}
\begin{align*}
{[\boldsymbol{\mathcal{U}}\boldsymbol{\rho}]_k}(\boldsymbol{y})=\mathcal{U}_{kj}{\rho}_j(\boldsymbol{y})&:=\displaystyle\int_{\Omega} u_{j}^{k}( \boldsymbol{x},\boldsymbol{y})\rho_{j}(\boldsymbol{x})\hspace{0.05em}dx,\\
{[\boldsymbol{\mathcal{R}}\boldsymbol{\rho}]_k}(\boldsymbol{y})=\mathcal{R}_{kj}{\rho}_j(\boldsymbol{y})&:=
\int_{\Omega} R_{kj}(\boldsymbol{x},\boldsymbol{y})\rho_{j}(\boldsymbol{x})\hspace{0.05em}dx, 
\end{align*}
and the  scalar Newton-type { and remainder potentials for the pressure,}
\begin{align}
\label{Qrho1}
&{[\bs{\mathcal{Q}}\rho]_j}(\boldsymbol{y})=\mathcal{Q}_j{\rho}(\boldsymbol{y})
:=\int_{\Omega} q^{j}( \boldsymbol{y},\boldsymbol{x})\rho(\boldsymbol{x})dx
=-\int_{\Omega} q^{j}( \boldsymbol{x},\boldsymbol{y})\rho(\boldsymbol{x})dx, \quad \\
\label{Qrho2}
&\mathcal{Q}\bs{\rho}(\boldsymbol{y}):={\bs{\mathcal{Q}\cdot}\boldsymbol{\rho}(\boldsymbol{y})=}
\mathcal{Q}_j{\rho}_j(\boldsymbol{y})
=\int_{\Omega} q^{j}( \boldsymbol{y},\boldsymbol{x})\rho_{j}(\boldsymbol{x})dx
=-\int_{\Omega} q^{j}( \boldsymbol{x},\boldsymbol{y})\rho_{j}(\boldsymbol{x})dx, \quad \\
&\mathcal{R}^{\bullet}\boldsymbol{\rho}(\boldsymbol{y})=\mathcal{R}^{\bullet}_j{\rho}_j(\boldsymbol{y}):=
{-}2\,{\rm v.p.}\int_{\Omega} \frac{\partial\mathring{q}^{j}(\boldsymbol{x},\boldsymbol{y})}{\partial x_{i}}
\frac{\partial \mu(\boldsymbol{x})}{\partial x_{i}}\rho_{j}(\boldsymbol{x})dx - \dfrac{4}{3}\rho_{j}\dfrac{\partial \mu}{\partial y_{j}}\label{ch2Rb}\\
&\hspace{8.5em}= -2\left< \partial_i \mathring{q}^j(\cdot,\boldsymbol{y})\, , \, \rho_i\partial_j\mu \right>_{\Omega} 
-2\rho_i(\boldsymbol{y})\partial_i\mu(\boldsymbol{y}), 
\label{ch2Rb-bl}
\end{align}
{ for $\boldsymbol{y}\in\mathbb{R}^3$. 
The integral in \eqref{ch2Rb} is understood as a 3D strongly singular integral in the Cauchy sense. 
The bilinear form in \eqref{ch2Rb-bl} should be understood in the sense of distribution, and the equality between \eqref{ch2Rb} and \eqref{ch2Rb-bl} holds since 
\begin{multline}
\left< \partial_i \mathring{q}^j(\cdot,\boldsymbol{y}), \rho_i\partial_j\mu \right>_{\Omega}  
  =-\left< \mathring{q}^j(\cdot,\boldsymbol{y}) , \partial_i (\rho_i\partial_j\mu) \right>_{\Omega} 
  +\left< n_i \mathring{q}^j(\cdot,\boldsymbol{y}), \rho_i\partial_j\mu \right>_{\partial\Omega}   \\
=-\lim_{\epsilon\to 0}\left< \mathring{q}^j(\cdot,\boldsymbol{y}),
\partial_i (\rho_i\partial_j\mu)\right>_{\Omega_\epsilon} 
      +2\left< n_i \mathring{q}^j(\cdot,\boldsymbol{y}), \rho_i\partial_j\mu \right>_{\partial\Omega}\\ 
 =\lim_{\epsilon\to 0}\left<  \partial_i \mathring{q}^j(\cdot,\boldsymbol{y}),
\rho_i\partial_j\mu\right>_{\Omega_\epsilon} 
- \lim_{\epsilon\to 0}\left< n_i \mathring{q}^j(\cdot,\boldsymbol{y}), 
\rho_i\partial_j\mu \right>_{\partial\Omega_\epsilon\setminus\partial\Omega}\\
={\rm v.p.}\int_{\Omega} \frac{\partial\mathring{q}^{j}(\boldsymbol{x},\boldsymbol{y})}{\partial x_{i}}
\frac{\partial \mu(\boldsymbol{x})}{\partial x_{i}}\rho_{j}(\boldsymbol{x})dx 
- \dfrac{1}{3}\rho_{j}\dfrac{\partial \mu}{\partial y_{j}}           
\end{multline}
where $\Omega_\epsilon=\Omega\setminus\bar B_\epsilon(\boldsymbol{y})$ and $B_\epsilon(\boldsymbol{y})$ is the ball of radius $\epsilon$ centred in $\boldsymbol{y}$, which implies that 
\begin{multline*}
-2\left< \partial_i \mathring{q}^j(\cdot,\boldsymbol{y})\, , \, \rho_i\partial_j\mu \right>_{\Omega} 
-2v_i(\boldsymbol{y})\partial_i\mu(\boldsymbol{y})=\\
-{\rm v.p.}\int_{\Omega} \frac{\partial\mathring{q}^{j}(\boldsymbol{x},\boldsymbol{y})}{\partial x_{i}}
\frac{\partial \mu(\boldsymbol{x})}{\partial x_{i}}\rho_{j}(\boldsymbol{x})dx 
- \dfrac{4}{3}\rho_{j}(\boldsymbol{y})\dfrac{\partial \mu(\boldsymbol{y})}{\partial y_{j}}
=\mathcal{R}^{\bullet}\boldsymbol{\rho}(\boldsymbol{y}). 
\end{multline*}
}

{ Let us now define the parametrix-based velocity} single layer potential, double layer potential and their respective direct values on the boundary, as follows:
\begin{align*}
{[\boldsymbol{V\rho}]_k}(\boldsymbol{y})=V_{kj}{\rho}_j(\boldsymbol{y})&:=
-\int_{{\partial\Omega}}u_{j}^{k}(\boldsymbol{x},\boldsymbol{y})\rho_{j}(\boldsymbol{x})
\, dS(\boldsymbol{x}),\hspace{0.4em}\boldsymbol{y}\notin {\partial\Omega},\\
{[\boldsymbol{W\rho}]_k}(\boldsymbol{y})=W_{kj}{\rho}_j(\boldsymbol{y})&:=
-\int_{{\partial\Omega}}{ T^c_{j}}(\boldsymbol{x};q^{k},\boldsymbol{u}^{k})(\boldsymbol{x},\boldsymbol{y})
\rho_{j}(\boldsymbol{x})
\, dS(\boldsymbol{x}),\hspace{0.4em}\boldsymbol{y}\notin {\partial\Omega}.
\end{align*}

For { the pressure 
we will} need the following single-layer and double layer potentials:
\begin{align*}
{\Pi^s}\boldsymbol{\rho}(\boldsymbol{y})={\Pi^s_j}{\rho}_j(\boldsymbol{y})&:=
{\int_{{\partial\Omega}} \mathring{q}^{j}( \boldsymbol{x},\boldsymbol{y})\rho_{j}(\boldsymbol{x})\hspace{0.05em}dS(\boldsymbol{x})},
\hspace{0.2em}\boldsymbol{y}\notin {\partial\Omega}\\
{\Pi^d}\boldsymbol{\rho}(\boldsymbol{y})={\Pi^d_j}{\rho}_j(\boldsymbol{y})&:=
{ 2\int_{{\partial\Omega}}\frac{\partial \mathring{q}^{j}(\boldsymbol{x},\boldsymbol{y})}{\partial n(\boldsymbol{x})}
\mu(\boldsymbol{x})\rho_{j}(\boldsymbol{x})\hspace{0.05em}dS(\boldsymbol{x})},\hspace{0.2em}\boldsymbol{y}\notin {\partial\Omega}.
\end{align*}


{ It is easy to observe that the parametrix-based integral operators, with the variable coefficient $\mu$,}
can be expressed in terms of the corresponding integral operators for the constant{-}coefficient case, $\mu=1$, { marked by $\mathring{}$,}

\begin{align} \label{ch2relationU}
&{\boldsymbol{\mathcal U}}\boldsymbol{\rho}=
\frac{1}{\mu}\mathring{\boldsymbol{\mathcal U}}\boldsymbol{\rho},\\
&{[\bs{\mathcal{R}\rho}]_k}=
\frac{-1}{\mu}\Big[
\frac{\partial}{\partial y_j}\mathring{\mathcal{U}}_{ki}(\rho_j\partial_i\mu) 
+\frac{\partial}{\partial y_i}\mathring{\mathcal{U}}_{kj}(\rho_j\partial_i\mu)
-\mathring{\mathcal{Q}}_k(\rho_j\partial_j\mu)\Big],\label{ch2relationR}\\
&\bs{\mathcal{Q}}\rho=\dfrac{1}{\mu(\bs{y})}\bs{\mathring{\mathcal Q}}(\mu\rho),\label{ch2relationQ2}\\ 
&{\mathcal{R}^{\bullet}\bs{\rho}}=
-2{\partial_i}\mathring{\mathcal{Q}}_j(\rho_j\partial_i\mu)-2\rho_{j}\partial_j\mu, \label{ch2relationRdot}\\
&{\bs{V}}\boldsymbol{\rho}=
\frac{1}{\mu}{\mathring{\bs{V}}}\boldsymbol{\rho},\qquad
{\bs{W}}\boldsymbol{\rho}=
\frac{1}{\mu}{\mathring{\bs{W}}}(\mu\boldsymbol{\rho}),\label{ch2relationVW}\\
&{\Pi^s\bs{\rho}}=
{\mathring{\Pi}^s\bs{\rho}}, \hspace{3.5em}
{\Pi^d\bs{\rho}}={\mathring{\Pi}^d(\mu\bs{\rho})}.
\label{ch2relationP}
\end{align}

Note that { although the constant-coefficient velocity potentials $\mathring{\boldsymbol{\mathcal U}}\boldsymbol{\rho}$, $\mathring{\bs{V}}\boldsymbol{\rho}$ and $\mathring{\bs{W}}\boldsymbol{\rho}$ are divergence-free in $\Omega^\pm$, the corresponding potentials $\boldsymbol{\mathcal U}\boldsymbol{\rho}$, $\bs{V}\boldsymbol{\rho}$ and $\bs{W}\boldsymbol{\rho}$ are {\em not divergence-free for the variable coefficient $\mu(\boldsymbol{y})$}.
Note also that by \eqref{F-q} and \eqref{Qrho1},
\begin{align}\label{QP}
\mathring{\mathcal{Q}}_j\rho={-}\partial_j\mathcal P_\Delta\rho,
\end{align}
where
\begin{align} 
\mathcal{P}_\Delta{\rho}(\boldsymbol{y})
=-\frac{1}{4\pi}\int_{\Omega} \frac{1}{\vert \boldsymbol{x} -\boldsymbol{y}\vert}\rho(\boldsymbol{x})dx
\end{align}
is the harmonic Newton potential. Hence 
\begin{align}\label{dQ}
\div\mathring{\bs{\mathcal Q}}\rho=\partial_j\mathring{\mathcal{Q}}_j\rho={-}\Delta\mathcal P_\Delta\rho={-}\rho.
\end{align}
Moreover, for the constant-coefficient potentials potentials we have the following well-known relations,
\begin{align}
&\mathring{\mathcal{A}}(\mathring{\Pi}^s\bs{\rho},\bs{\mathring{V}\rho})=\bs{0},\quad
\mathring{\mathcal{A}}(\mathring{\Pi}^d\bs{\rho},\bs{\mathring{W}}\bs{\rho})=\bs{0}\quad \mbox{in } \Omega^\pm,
\label{A0V0}\\
&\mathring{\mathcal{A}}(\mathring{\mathcal{Q}}\bs{\rho},\mathring{\boldsymbol{\mathcal{U}}}\bs{\rho})=\bs{\rho}.
\label{A0U0}
\end{align}
In addition, by \eqref{QP} and \eqref{dQ},
\begin{align}\label{4.15}
\mathring{\mathcal{A}_{j}}(\frac{4}{3}\rho, {-}\mathring{\boldsymbol{\mathcal{Q}}}\rho) 
&= {-}\partial_{i}\left( \partial_i\mathring{\mathcal{Q}_{j}}\rho
+\partial_j\mathring{\mathcal{Q}_{i}}\rho
-\dfrac{2}{3}\delta_{i}^{j}\div\boldsymbol{\mathring{\mathcal{Q}}}\rho\right) 
-\dfrac{4}{3} \partial_j\rho\nonumber\\
&=  {-}(\Delta\mathring{\mathcal{Q}_{j}}\rho
+\partial_j\div\bs{\mathring{\mathcal{Q}}}\rho
-\dfrac{2}{3}\partial_j\div\boldsymbol{\mathring{\mathcal{Q}}}\rho) 
-\dfrac{4}{3} \partial_j\rho=0
\end{align}
}

The following assertions of this section are well-known for the constant coefficient case, see e.g. \cite{kohr1, hsiao}. Then, by relations \eqref{ch2relationU}-\eqref{ch2relationTV} we obtain their counterparts for the variable-coefficient case.
\begin{theorem}\label{ch2thmUR:theo} The following operators are continuous: 
\begin{align}
\bs{\mathcal{U}}&:\widetilde{\bs{H}}^{s}(\Omega) \to \bs{H}^{s+2}(\Omega),\hspace{0.5em} s\in \mathbb{R},\label{ch2OpC1}\\
\bs{\mathcal{U}}&: \bs{H}^{s}(\Omega) \to \bs{H}^{s+2}(\Omega),\hspace{0.5em} s>-1/2,\label{ch2OpC2}\\
\bs{\mathcal Q}&:\widetilde{{H}}^{s}(\Omega)  \to \bs{ H}^{s+1}(\Omega),\hspace{0.5em}s\in \mathbb{R},\label{ch2OpC7}\\
\bs{\mathcal Q}&:{{H}}^{s}(\Omega) \to \bs{ H}^{s+1}(\Omega),\hspace{0.5em}s>-1/2,\label{ch2OpC7b}\\
{\mathcal Q}&:\widetilde{\bs{ H}}^{s}(\Omega)  \to { H}^{s+1}(\Omega),\hspace{0.5em}s\in \mathbb{R},\label{ch2OpC7s}\\
{\mathcal Q}&:{\bs{ H}}^{s}(\Omega) \to { H}^{s+1}(\Omega),\hspace{0.5em}s>-1/2,\label{ch2OpC7bs}\\
\bs{\mathcal{R}}&:\widetilde{\bs{H}}^{s}(\Omega) \to \bs{H}^{s+1}(\Omega),\hspace{0.5em} s\in \mathbb{R},\label{ch2OpC3}\\
\bs{\mathcal{R}}&: \bs{H}^{s}(\Omega) \to \bs{H}^{s+1}(\Omega),\hspace{0.5em} s>-1/2,\label{ch2OpC4}\\ 
\mathcal R^{\bullet}&: \widetilde{\bs{H}}^{s}(\Omega) \to H^{s}(\Omega),\hspace{0.5em}s\in \mathbb{R}.\label{ch2OpC8}\\
\mathcal R^{\bullet}&: \bs{H}^{s}(\Omega) \to H^{s}(\Omega),\hspace{0.5em}s>-1/2,\label{ch2OpC8b}\\ 
(\mathring{\mathcal Q},\bs{\mathcal{U}})&:\bs{H}^{s}(\Omega) \to \bs{H}^{s+2,0}(\Omega;\boldsymbol{\mathcal{A}}),\hspace{0.5em} s\ge 0,\label{ch2H10QU}\\
(\dfrac{4\mu}{3}I,{-}\bs{Q})&:H^{s-1}(\Omega)\to\bs{H}^{s,0}(\Omega;\bs{\mathcal{A}}),\hspace{0.5em} s\ge 1,\label{ch2H10IQ}\\
(\mathcal R^{\bullet},\bs{\mathcal R})&:\bs{H}^{s}(\Omega) \to \bs{H}^{s+1,0}(\Omega;\boldsymbol{\mathcal{A}}),\hspace{0.5em} s\ge 1.\label{ch2H10RR}
\end{align}
\end{theorem} 
\begin{proof}
 Since the surface ${\partial\Omega}$ is infinitely differentiable, the operators $\bs{\mathcal{U}}$ and $\bs{\mathcal{Q}}$ are respectively pseudodifferential operators of order $-2$ and $-1$, see{, e.g.,} \cite[Lemma 5.6.6. and Section 9.1.3]{hsiao}. Then, the continuity of the operators $\bs{\mathcal{U}}$ and $\bs{\mathcal{Q}}$ from the `tilde spaces'{, i.e., \eqref{ch2OpC1} and \eqref{ch2OpC7},} immediately follows by virtue of the mapping properties of the pseudodifferential operators (see, e.g. \cite{eskin,michlin}). Alternatively, these mapping properties are well studied for the constant coefficient case, i.e. operators $\mathring{\bs{\mathcal{U}}}$ and $\mathring{\bs{\mathcal{Q}}}$, see \cite[Lemma 5.6.6]{hsiao}. 
{ Then continuity of operator \eqref{ch2OpC7s}  immediately follows from representation \eqref{Qrho2} and continuity of operator \eqref{ch2OpC7}.
Continuity of the remainder operators \eqref{ch2OpC3} and \eqref{ch2OpC8} is also implied by continuity of operators \eqref{ch2OpC7}, \eqref{ch2OpC1} and relations \eqref{ch2relationR}, \eqref{ch2relationRdot}.}

For the remaining part of the proof, we shall { first} assume that $s\in \left(-1/2, 1/2 \right)$. In this case, $H^{s}(\Omega)$ { can be identified with} $\widetilde{H}^{s}(\Omega)$. Hence, the continuity of the operator \eqref{ch2OpC2} immediately follows from the continuity of \eqref{ch2OpC1}.

Let now $s\in(1/2, 3/2)$ and $\bs{g}=(g_{1},g_{2},g_{3})\in \bs{H}^{s}(\Omega)$. Then $\partial_{j} g_{i}\in H^{s-1}(\Omega)$ and $\gamma^{+}\bs{g}\in \bs{H}^{s-1/2}({\partial\Omega})$. 
Consequently, 
integrating by parts (cf. the proof of Theorem 3.8 in \cite{CMN-1}),
\begin{equation}\label{ch2th38a}
\partial_{j}\mathring{\mathcal{U}}_{ik}g_{k} = \mathring{\mathcal{U}}_{ik}(\partial_{j}g_{k}) + \mathring{V}_{ik}({ n_{j}\gamma^{+}g_{k}}),\quad\quad i,j,k\in\lbrace 1,2,3\rbrace. 
\end{equation}
Keeping in mind { continuity of the operators 
$\mathring{\bs{\mathcal U}}: \bs{H}^{s-1}(\Omega) \to\bs{H}^{s+1}(\Omega)$ 
proved in the previous paragraph and the well-known continuity of the operators  
$\mathring {\bs{ V}}:\bs{H}^{s-1/2}({\partial\Omega})\to \bs{H}^{s+1}(\Omega)$, 
for $s\in \left(1/2,{3}/{2}\right)$, we deduce that the operators 
$\partial_{j}\mathring{\bs{\mathcal U}}: \bs{H}^{s}(\Omega)\to \bs{H}^{s+1}(\Omega)$ are continuous for $s\in \left(1/2,{3}/{2}\right)$. 
Since the Bessel potential spaces and the Sobolev-Slobodetsky spaces are equivalent for non-negative smoothness index, the continuity of operator \eqref{ch2OpC2} for $s\in(1/2, 3/2)$ immediately follows from  the Sobolev-Slobodetsky spaces definition and relation \eqref{ch2relationU}.  
Furthermore,  using representation \eqref{ch2th38a} one can prove by induction that the operator \eqref{ch2OpC2} is continuous also for $s\in(k-1/2,k+1/2)$,  $k\in \mathbb{N}$. 
Continuity} of the operator \eqref{ch2OpC2} for the cases $s=k+1/2$ is proved by applying the theory of interpolation of Bessel potential spaces (see, e.g. \cite[Chapter 4]{triebel}).

Continuity of the operator \eqref{ch2OpC7b} and hence \eqref{ch2OpC7bs} can be proved following a similar argument. 
{ Continuity of} the remainder operators \eqref{ch2OpC4} and \eqref{ch2OpC8b} immediately follows from the continuity of { operators \eqref{ch2OpC2} and \eqref{ch2OpC7b} by relations \eqref{ch2relationR} and \eqref{ch2relationRdot}. 
ontinuity of operator \eqref{ch2H10QU} and \eqref{ch2H10RR} is implied by continuity of operators \eqref{ch2OpC2}, \eqref{ch2OpC7bs} and \eqref{ch2OpC4}, \eqref{ch2OpC8b}, respectively, along with the space definition \eqref{Hs0def}.

Let us prove continuity of operator \eqref{ch2H10IQ}. Let $\rho\in H^{s-1}(\Omega)$ and $g=\mu \rho$. Then by relation \eqref{ch2relationQ2} and mapping property  \eqref{ch2OpC7b}, 
$(\frac{4}{3}\mu \rho,\bs{\mathcal{Q}}\rho)=(\frac{4}{3}g,\frac{1}{\mu}\bs{\mathring{\mathcal{Q}}}g)\in H^{s-1}(\Omega)\times\bs{H}^{s}(\Omega)$  and
\begin{align}
{\mathcal{A}_{j}}\left(\frac{4}{3}\mu \rho, {-}{\boldsymbol{\mathcal{Q}}}\rho\right)
&= {\mathcal{A}_{j}}\left(\frac{4}{3}g, {-}\frac{1}{\mu}\bs{\mathring{\mathcal{Q}}}g\right)
= {\mathring{\mathcal{A}_{j}}}\left(\frac{4}{3}g, {-}\bs{\mathring{\mathcal Q}}g\right)
{+}\partial_i\left[\frac{\partial_i\mu}{\mu}\mathring{\mathcal Q}_jg
+\frac{\partial_j\mu}{\mu}\mathring{\mathcal Q}_ig
-\dfrac{2}{3}\delta_{i}^{j}\frac{\partial_j\mu}{\mu}\mathring{\mathcal Q}_lg\right].
\end{align}
By \eqref{4.15}, 
$
\mathring{\mathcal{A}_{j}}(\frac{4}{3}g, {-}\mathring{\boldsymbol{\mathcal{Q}}}g)=0, 
$
while
$$
\partial_i\left[\frac{\partial_i\mu}{\mu}\mathring{\mathcal Q}_jg
+\frac{\partial_j\mu}{\mu}\mathring{\mathcal Q}_ig
-\dfrac{2}{3}\delta_{i}^{j}\frac{\partial_j\mu}{\mu}\mathring{\mathcal Q}_lg\right]: H^{s-1}(\Omega)\to H^{s-1}(\Omega)
$$
is a continuous operator due to \eqref{ch2OpC7b}, which implies continuity of \eqref{ch2H10IQ}.}
 \end{proof}

\begin{theorem}\label{ch2thRcomp}
 { Let $s>1/2$. The following operators are compact,}
\begin{align*}
&\bs{\mathcal R}: \bs{H}^{s}(\Omega) \to \bs{H}^{s}(\Omega),
\qquad &&{\mathcal R^\bullet}: \bs{H}^{s}(\Omega) \to { H}^{s-1}(\Omega),\\
&\gamma^{+}\bs{\mathcal{R}}: \bs{H}^{s}(\Omega) \to \bs{H}^{s-1/2}({\partial\Omega}),\qquad
&&  \bs{T}^{\pm}(\mathcal{R}^{\bullet},\bs{\mathcal{R}}): \bs{H}(\Omega;\bs{\mathcal{A}}) \to \bs{H}^{-1/2}({\partial\Omega}). 
\end{align*}
\end{theorem}
\begin{proof}
The proof of the compactness for the operators $\bs{\mathcal{R}}$, $\gamma^{+}\bs{\mathcal{R}}$ and $\mathcal{R}^{\bullet}$  immediately follows from Theorem \ref{ch2thmUR:theo} and the trace theorem along with the Rellich compact embedding theorem. 

To prove { compactness of the operators $\bs{T}^{\pm}(\mathcal{R}^{\bullet},\bs{\mathcal{R}})$, let us} consider a function $\bs{g}\in\bs{H}^{1}(\Omega)$. 
Then, 
$(\mathcal{R}^{\bullet}\bs{g},\bs{\mathcal{R}g})\in 
H^{1}(\Omega)\times \bs{H}^{2}(\Omega){\subset} \bs{H}^{1,0}(\Omega;\bs{\mathcal{A}})$, 
{ which implies that both canonical and classical conormal derivatives of $(\mathcal{R}^{\bullet}\bs{g},\bs{\mathcal{R}g})$ are well defined and moreover, similar to \cite[Corollary 3.14]{MikJMAA2011} and \cite[Theorem 2.13]{GKMW2017}, one can prove that they coincide, $\bs{T}^{\pm}(\mathcal{R}^{\bullet}\bs{g},\bs{\mathcal{R}}\bs{g})
=\bs{T}^{c\pm}(\mathcal{R}^{\bullet}\bs{g},\bs{\mathcal{R}}\bs{g})\in \bs{H}^{1/2}(\partial\Omega)$.
By definitions \eqref{sigmadef}, \eqref{ch2Tcl}, the operators 
$\bs{T}^{c\pm}: H^{1}(\Omega)\times \bs{H}^{2}(\Omega)\to \bs{H}^{1/2}(\partial\Omega)$
are continuous, which implies that the operators
$\bs{T}^{\pm}(\mathcal{R}^{\bullet},\bs{\mathcal{R}})=\bs{T}^{c\pm}(\mathcal{R}^{\bullet},\bs{\mathcal{R}}): \bs{H}^{1}(\Omega)\to \bs{H}^{1/2}(\partial\Omega)$
are continuous as well.}
Then, compactness {of the operators $\bs{T}^{\pm}(\mathcal{R}^{\bullet},\bs{\mathcal{R}})=\bs{T}^{c\pm}(\mathcal{R}^{\bullet},\bs{\mathcal{R}}): \bs{H}^{1}(\Omega)\to \bs{H}^{-1/2}(\partial\Omega)$} follows from the Rellich compact embedding $\bs{H}^{1/2}({\partial\Omega}){\hookrightarrow}\bs{H}^{-1/2}({\partial\Omega})$.
\end{proof}


\begin{theorem}\label{ch2T3} 
The following operators are continuous,
\begin{align}
\label{ch2OpC5-}
&\bs{V}: \bs{H}^{s}({\partial\Omega}) \to \bs{H}^{s+\frac{3}{2}}(\Omega),
\quad\bs {W}: \bs{H}^{s}({\partial\Omega}) \to \bs{H}^{s+1/2}(\Omega),\quad s\in \mathbb{R},\\
%
&{\Pi^s}:\bs{H}^{s-\frac{3}{2}}({\partial\Omega}) \to H^{s-1}(\Omega), 
\quad
{\Pi^d}: \bs{H}^{s-1/2}({\partial\Omega}) \to H^{s-1}(\Omega),
\quad s\in \mathbb{R},
\label{ch2OpC6}\\
%
&({\Pi^s},\bs{V}):\bs{H}^{-1/2}({\partial\Omega})\to \bs{H}^{1,0}(\Omega;\bs{\mathcal{A}}), 
\quad
({\Pi^d}, \bs{W}):\bs{H}^{1/2}({\partial\Omega})\to \bs{H}^{1,0}(\Omega;\bs{\mathcal{A}}). \label{ch2H10PiW}
\end{align}
\end{theorem}
\begin{proof} 
The continuity of the operators in \eqref{ch2OpC5-}, \eqref{ch2OpC6} follows from relations \eqref{ch2relationVW}, \eqref{ch2relationP} and the continuity of the counterpart operators for the constant coefficient case, see e.g. \cite{kohr1, hsiao}.

Let us prove continuity of the operators in \eqref{ch2H10PiW}. We first remark that an arbitrary pair $(p,\bs{v})$ belongs to $\bs{H}^{1,0}(\Omega;\bs{\mathcal{A}})$ if $(p,\bs{v})\in L^{2}(\Omega)\times \bs{H}^{1}(\Omega)$ and $\bs{\mathcal{A}}(p,\bs{v})\in \bs{L}^{2}(\Omega)$. 
By expanding the operator  $\mathcal{A}_{j}(\boldsymbol{y}; p,\bs{v})$
\begin{align}
\mathcal{A}_{j}(\boldsymbol{y}; p,\bs{v})&=\mathring{\mathcal{A}}_{j}(\boldsymbol{y}; p,\mu\bs{v})
-\partial_{i}\left[v_{j}(\bs{y})\partial_i\mu(\bs{y}) +v_{i}(\bs{y})\partial_j\mu(\bs{y})-\dfrac{2}{3}\delta_{i}^{j}v_{l}(\bs{y})\partial_l\mu(\bs{y})\right]\label{ch2H10-2},
 \end{align}
we can see that if $\bs{v}\in \bs{H}^{1}(\Omega)$, then the second term in \eqref{ch2H10-2} belongs to $\bs{L}^{2}(\Omega)$. Therefore, we only need to check that $\mathring{\mathcal{A}}_{j}(\boldsymbol{y}; p,\mu\bs{v})\in L^{2}(\Omega)$. 

First, let us prove the corresponding mapping property for the pair the pair $({\Pi^s},\bs{V})$. 
Let $\bs{\Psi}\in \bs{H}^{-1/2}({\partial\Omega})$. 
Then, $({\Pi^s}\bs{\Psi}, \bs{{V}}\bs{\Psi})\in L^{2}(\Omega)\times \bs{H}^{1}(\Omega)$ by virtue of \eqref{ch2OpC5-}, \eqref{ch2OpC6}. 
Now, applying relations \eqref{ch2relationVW} and \eqref{ch2relationP}, $\mathring{\mathcal{A}}_{j}({\Pi^s}\bs{\Psi},\mu\bs{V\Psi}) =\mathring{\mathcal{A}}_{j}(\mathring{\Pi}^s\bs{\Psi},\bs{\mathring{V}\Psi})=0$ in $\Omega$, which completes the proof for the pair $({\Pi^s},\bs{V})$.

Let $\bs{\Phi}\in \bs{H}^{1/2}({\partial\Omega})$. 
By virtue of \eqref{ch2OpC5-}, \eqref{ch2OpC6}, $\left({\Pi^d}\bs{\Phi},\bs{W}\bs{\Phi}\right)\in L^{2}(\Omega)\times \bs{H}^{1}(\Omega)$. 
Moreover, by applying relations \eqref{ch2relationVW} and \eqref{ch2relationP} we deduce $\mathring{\mathcal{A}}_{j}({\Pi^d}\bs{\Phi},\mu\bs{W}\bs{\Phi}) =\mathring{\mathcal{A}}_{j}(\mathring{{\Pi^d}}(\mu\bs{\Phi}),\bs{\mathring{W}}(\mu\bs{\Phi}))=0$ in $\Omega$,  which completes the proof for $({\Pi^d}, \bs{W})$. 
\end{proof}

{ Let us now define direct values on the boundary of the parametrix-based velocity single layer and double layer potentials and introduce the notations for the conormal derivative of the latter,
\begin{align*}
{[\boldsymbol{\mathcal V\rho}]_k}(\boldsymbol{y})=\mathcal{V}_{kj}{\rho}_j(\boldsymbol{y})&:=
-\int_{{\partial\Omega}} u_{j}^{k}(\boldsymbol{x},\boldsymbol{y})\rho_{j}(\boldsymbol{x})
\, dS(\boldsymbol{x}),\hspace{0.4em}\boldsymbol{y}\in {\partial\Omega},\\
{[\boldsymbol{\mathcal W\rho}]_k}(\boldsymbol{y})=\mathcal{W}_{kj}{\rho}_j(\boldsymbol{y})&:=
-\int_{{\partial\Omega}}{ T^c_{j}}(\boldsymbol{x};q^{k},\boldsymbol{u}^{k})(\boldsymbol{x},\boldsymbol{y}) \rho_{j}(\boldsymbol{x})
\, dS(\boldsymbol{x}),\hspace{0.4em}\boldsymbol{y}\in {\partial\Omega},\\
{[\bs{\mathcal{W}}'\boldsymbol{\rho}]_k}(\boldsymbol{y}) =\mathcal{W}'_{kj}{\rho}_j(\boldsymbol{y})&:=
-\int_{{\partial\Omega}}{ T^c_{j}}(\boldsymbol{y};q^{k},\boldsymbol{u}^{k})(\boldsymbol{x},\boldsymbol{y}) \rho_{j}(\boldsymbol{x})
\, dS(\boldsymbol{x}),\hspace{0.4em}\boldsymbol{y}\in {\partial\Omega},\\
{\bs{\mathcal{L}}^{\pm}\boldsymbol{\rho}}(\boldsymbol{y})  
&:= \bs{T}^{\pm}({\Pi^d}\boldsymbol{\rho}, \boldsymbol{W\rho})(\boldsymbol{y}), \quad \boldsymbol{y}\in {\partial\Omega}.
\end{align*}
Here $\bs{T}^{\pm}$ are the { canonical derivative} (traction) operators for the {\em compressible} fluid that are well defined due to continuity of the second operator in \eqref{ch2H10PiW}.

Similar to the potentials in the domain, we can also express the boundary operators in terms of their counterparts with the constant coefficient $\mu=1$,
\begin{align} 
&{\bs{\mathcal V}}\boldsymbol{\rho}
=\frac{1}{\mu}{\mathring{\bs{\mathcal V}}}\boldsymbol{\rho},\qquad
{\bs{\mathcal W}}\boldsymbol{\rho}=
\frac{1}{\mu}{\mathring{\bs{\mathcal W}}}(\mu\boldsymbol{\rho}),\label{ch2relationVWcal}\\
&[\bs{\mathcal{W}}'\boldsymbol{\rho}]_k
= [\mathring{\bs{\mathcal W}}{}'\boldsymbol{\rho}]_k - 
\left(\frac{\partial_i\mu}{\mu}\,[\mathring{\bs{\mathcal V}}\boldsymbol{\rho}]_k+
\frac{\partial_k\mu}{\mu}\,[\mathring{\bs{\mathcal V}}\boldsymbol{\rho}]_i
-\frac{2}{3}\delta_i^k\frac{\partial_j\mu}{\mu}\,[\mathring{\bs{\mathcal V}}\boldsymbol{\rho}]_j
\right)n_i.\label{ch2relationTV}
\end{align}
\begin{theorem}\label{ch2T3cal} 
Let $s\in \mathbb{R}$. Let $S_1$ and $S_2$ be two non empty manifolds on ${\partial\Omega}$ with smooth boundaries $\partial S_1$ and $\partial S_2$, respectively.  Then the following operators are continuous,
\begin{align}
\bs{\mathcal V}&: \bs{H}^{s}({\partial\Omega}) \to \bs{H}^{s+1}({\partial\Omega}),
&\bs{\mathcal W}&: \bs{H}^{s}({\partial\Omega}) \to \bs{H}^{s+1}({\partial\Omega}),\label{cVcWcont}\\
r_{S_{2}}\bs{\mathcal V}&: \widetilde{\bs{H}}^{s}(S_{1}) \to \bs{H}^{s+1}(S_{2}),&
r_{S_{2}}\bs{\mathcal W}&: \widetilde{\bs{H}}^{s}(S_{1}) \to \bs{H}^{s+1}(S_{2}),\\
\bs{\mathcal L}^{\pm}&:\bs{H}^{s}({\partial\Omega}) \to \bs{H}^{s-1}({\partial\Omega}),&
\bs{\mathcal W}'&: \bs{H}^{s}({\partial\Omega}) \to \bs{H}^{s+1}({\partial\Omega}).
\label{cLcWpcont}
\end{align}
 Moreover, the following operators are compact,
\begin{align}
r_{S_{2}}\bs{\mathcal V}: \widetilde{\bs{H}}^{s}(S_{1}) \to \bs{H}^{s}(S_{2}),\label{cVcomp}\\
r_{S_{2}}\bs{\mathcal W}: \widetilde{\bs{H}}^{s}(S_{1}) \to \bs{H}^{s}(S_{2}),\label{cWcomp}\\
r_{S_{2}}\bs{\mathcal W}':\widetilde{\bs{H}}^{s}(S_{1}) \to \bs{H}^{s}(S_{2}).\label{cWpcomp}
\end{align}
\end{theorem}
\begin{proof}
Continuity of operators in \eqref{cVcWcont}-\eqref{cLcWpcont} follows from relations \eqref{ch2relationVWcal}-\eqref{ch2relationTV} and continuity of the counterpart operators for the constant coefficient case, see e.g. \cite{kohr1, hsiao}. 
Then compactness of operators \eqref{cVcomp}-\eqref{cWpcomp} is implied by the Rellich compactness embedding theorem.
\end{proof}
}

\begin{theorem}\label{ch2jumps} If $\boldsymbol{\tau}\in \boldsymbol{H}^{1/2}({\partial\Omega})$, $\boldsymbol{\rho}\in \boldsymbol{H}^{-1/2}({\partial\Omega})$, then the following jump relations hold on $\partial\Omega$:
\begin{align*}
&\gamma^{\pm}\bs{V}\boldsymbol{\rho}=\bs{\mathcal V}\boldsymbol{\rho},\qquad
\gamma^{\pm}\bs{W}\boldsymbol{\tau}=\mp\dfrac{1}{2}\bs{\tau}+\bs{\mathcal W}\boldsymbol{\tau}\\
&\bs{T}^{\pm}({\Pi^s}\boldsymbol{\rho},\boldsymbol{V\rho})=
\pm\frac{1}{2}\bs{\rho}+\bs{\mathcal W'}\boldsymbol{\rho}.
\end{align*}\end{theorem}
\begin{proof}
The proof of the theorem directly follows from relations \eqref{ch2relationVW},  \eqref{ch2relationVWcal}-\eqref{ch2relationTV} and the analogous jump properties for the counterparts of the operators for the constant coefficient case of $\mu=1$, see \cite[Lemma 5.6.5]{hsiao}.
\end{proof}

Let denote
\begin{align}
&
\mathring{\bs{\mathcal L}}\boldsymbol{\tau}(\boldsymbol{y})  
=\mathring{\bs{\mathcal L}}^\pm\boldsymbol{\tau}(\boldsymbol{y})  
:= \mathring{\bs{T}}^{\pm}(\mathring{\Pi}^d\boldsymbol{\tau}, \mathring{\bs{W}}\bs{\tau})(\boldsymbol{y}), \quad
\widehat{\bs{\mathcal L}}\boldsymbol{\tau}(\boldsymbol{y})
:= \mathring{\bs{\mathcal L}}(\mu\boldsymbol{\tau})(\boldsymbol{y}),\quad \boldsymbol{y}\in {\partial\Omega},\label{ch2relationL}
\end{align}
where the first equality is implied by Lyapunov-Tauber theorem for the constant-coefficient Stokes potentials.

\begin{theorem}\label{ch2jumps2} Let $\boldsymbol{\tau}\in \boldsymbol{H}^{1/2}({\partial\Omega})$. Then, the following jump relation holds:
\begin{align}
&(\mathcal{L}_{k}^{\pm}-\widehat{\mathcal{L}}_k)\boldsymbol{\tau}=\nonumber\\
 &\gamma^{\pm}\left(\mu\left[\partial_i\left(\dfrac{1}{\mu}\right)\,\mathring{W}_k(\mu\boldsymbol{\tau})+
\partial_k\left(\dfrac{1}{\mu}\right)\,\mathring{W}_i(\mu\boldsymbol{\tau})
-\frac{2}{3}\delta_i^k\partial_j\left(\dfrac{1}{\mu}\right)\,\mathring{W}_j(\mu\boldsymbol{\tau})\right]\right)n_{i}
\label{ch2jumpL}.
\end{align}
\begin{proof}
By Theorem \ref{ch2T3}, the operator 
$({\Pi^d}, \bs{W}):\bs{H}^{1/2}({\partial\Omega})\to \bs{H}^{1,0}(\Omega;\bs{\mathcal{A}})$
is continuous.
By Theorem~\ref{densT}, there exists a sequence
$(\pi^m,\bs{w}^m)_{m=1}^\infty\subset\mathcal D(\overline \Omega)\times\boldsymbol{\mathcal D}(\overline \Omega)$ converging to $(\mathring{\Pi}^d(\mu\bs{\tau}), \mathring{\bs{W}}(\mu\bs{\tau}))$ in $\bs{H}^{1,0}(\Omega;  \boldsymbol{\mathcal{A}})$. 
Then, due to \eqref{ch2relationVW}-\eqref{ch2relationP}, the sequence $(\pi^m,\dfrac{1}{\mu}\bs{w}^m)_{m=1}^\infty\subset
\mathcal D(\overline \Omega)\times\boldsymbol{\mathcal D}(\overline \Omega)$ 
converges to 
$(\mathring{\Pi^d}(\mu\bs{\tau}), \dfrac{1}{\mu}\mathring{\bs{W}}(\mu\bs{\tau}))
=({\Pi^d}\bs{\tau}, {\bs{W}}\bs{\tau})$ in $\bs{H}^{1,0}(\Omega;  \boldsymbol{\mathcal{A}})$ and 
by continuity of the canonical traction operators $\boldsymbol{T}^{\pm}: \bs{H}^{1,0}(\Omega^\pm; \boldsymbol{\mathcal{A}}) \to \boldsymbol{H}^{-1/2}({\partial\Omega})$
\begin{align}
\mathcal{L}_{k}^{\pm}\bs{\tau}:=T_{k}^{\pm}({\Pi^d}\bs{\tau}, \bs{W\tau})
=T_{k}^{\pm}({\Pi^d}\bs{\tau}, \bs{W\tau})
=\lim_{m\to\infty}T_{k}^{\pm}(\pi^m,\dfrac{1}{\mu}\bs{w}^m).
\end{align}
On the other hand,
\begin{align*}
&T_{k}^{\pm}(\pi^m,\dfrac{1}{\mu}\bs{w}^m)=T_{k}^{c\pm}(\pi^m,\dfrac{1}{\mu}\bs{w}^m)
=\gamma^{\pm}\sigma_{ik}(\pi^m,\dfrac{1}{\mu}\bs{w}^m)n_{i}\\
&=\gamma^{\pm}\mathring{\sigma}_{ik}(\pi^m,\bs{w}^m)n_{i} 
+ \gamma^{\pm}\left(\mu\left[\partial_i\left(\dfrac{1}{\mu}\right)\,w^m_k+
\partial_k\left(\dfrac{1}{\mu}\right)\,w^m_i
-\frac{2}{3}\delta_i^k\partial_j\left(\dfrac{1}{\mu}\right)\,w^m_j\right]\right)n_{i}\\
&\to\mathring{\mathcal{L}}_{k}^{\pm}(\mu\bs{\tau}) 
+ \gamma^{\pm}\left(\mu\left[\partial_i\left(\dfrac{1}{\mu}\right)\,\mathring{W}_k(\mu\boldsymbol{\tau})+
\partial_k\left(\dfrac{1}{\mu}\right)\,\mathring{W}_i(\mu\boldsymbol{\tau})
-\frac{2}{3}\delta_i^k\partial_j\left(\dfrac{1}{\mu}\right)\,\mathring{W}_j(\mu\boldsymbol{\tau})\right]\right)n_{i}\\ 
\end{align*}
since 
$$\gamma^{\pm}\mathring{\sigma}_{ik}(\pi^m,\bs{w}^m)n_{i}=\mathring{T}_{k}^{c\pm}(\pi^m,\bs{w}^m)
=\mathring{T}_{k}^{\pm}(\pi^m,\bs{w}^m)\to 
\mathring{T}_{k}^{\pm}(\mathring{\Pi}^d(\mu\bs{\tau}), \mathring{\bs{W}}(\mu\bs{\tau}))
=\mathring{\mathcal{L}}_{k}^{\pm}(\mu\bs{\tau})$$
as $m\to \infty$. 
This implies \eqref{ch2jumpL}.
\end{proof}

\begin{corollary}\label{ch2Lcompact}Let $S_{1}$ be a non empty submanifold of ${\partial\Omega}$ with smooth boundary. Then, the operators
\begin{align}\label{4.48}
\bs{\widehat{\mathcal{L}}}&: \bs{\widetilde{H}}^{1/2}(S_{1}) \to \bs{H}^{-1/2}({\partial\Omega}),\quad
(\bs{\mathcal{L}}^{\pm}-\bs{\widehat{\mathcal{L}}}):\bs{\widetilde{H}}^{1/2}(S_{1}) \to \bs{H}^{1/2}({\partial\Omega}),
\end{align}
are continuous and the operators
\begin{align}\label{4.49}
&(\bs{\mathcal{L}}^{\pm}-\bs{\widehat{\mathcal{L}}}): \bs{\widetilde{H}}^{1/2}(S_{1}) \to \bs{H}^{-1/2}({\partial\Omega}),
\end{align}
are compact.
\end{corollary}

\begin{proof}
The continuity of operators in \eqref{4.48} follows from Theorems \ref{ch2jumps2} and \ref{ch2T3}. The compactness of the operators \eqref{4.48} follows from the continuity of the second operators in \eqref{4.48} ands the compact embedding $\bs{H}^{1/2}(S_{1}){\hookrightarrow} \bs{H}^{-1/2}(S_{1})$.
\end{proof}

\end{theorem}

\section{The Third Green Identities}
{ 
\begin{theorem}\label{ch2th3gp}
For any $(p,\bs{v})\in \bs{H}^{1,0}(\Omega;  \boldsymbol{\mathcal{A}})$ the following third Green identities hold
\begin{align}\label{ch2representationp}
p+\mathcal{R}^{\bullet} \boldsymbol{v} - {\Pi^s}\boldsymbol{T}(p,\bs{v}) +{\Pi^d} \boldsymbol{\gamma}^{+}\boldsymbol{v}&=\mathring{\mathcal{Q}}\boldsymbol{\mathcal{A}}(p,\bs{v})
+\frac{4\mu}{3}\div\boldsymbol{v}\quad \text{in}\,\,\,\Omega,\\
\label{ch2vrepresentationA}
\boldsymbol{v}+\boldsymbol{\mathcal{R}}\boldsymbol{v}-\boldsymbol{V}\boldsymbol{T}^{+}(p,\bs{v})
+\boldsymbol{W}\boldsymbol{\gamma}^{+}\boldsymbol{v}
&=\boldsymbol{\mathcal{U}\mathcal{A}}(p,\bs{v}){-}\boldsymbol{\mathcal{Q}}\,\div\boldsymbol{v} \quad \text{in } \Omega.
\end{align}
\end{theorem}
} %
\begin{proof}
{ For an arbitrary fixed $\boldsymbol{y}\in\Omega$, let $B_\epsilon(\boldsymbol{y})\subset \Omega$ be a ball with a small enough radius $ \epsilon$ and centre $\boldsymbol{y}\in \Omega$, and let $\Omega_\epsilon(\boldsymbol{y})=\Omega\setminus \bar B_\epsilon(\boldsymbol{y})$. 
Let first 
$(p,\bs{v})\in\mathcal D(\overline{\Omega})\times \boldsymbol{\mathcal D}(\overline{\Omega})\subset  \bs{H}^{1,0}(\Omega;  \boldsymbol{\mathcal{A}})$. 

(i) Let us start from the velocity identity \eqref{ch2vrepresentationA}. 
For the parametrix, evidently, we have the inclusion $(q^{k},\bs{u}^{k})\in \bs{H}^{1,0}(\Omega_\epsilon(\boldsymbol{y});\bs{\mathcal A})$. 
Therefore, we can apply the second Green identity \eqref{ch2secondgreen} in the domain $\Omega_\epsilon(\boldsymbol{y})$ to  $(p,\bs{v})$ and to  $(q^{k},\boldsymbol{u}^{k})$ to obtain
\begin{multline}\label{4.G3tile}
\int_{\partial B_\epsilon(\boldsymbol{y})}\gamma^+\boldsymbol{u}^k(\boldsymbol{x},\boldsymbol{u})\cdot
\boldsymbol{T}^{+}(p(\boldsymbol{x}),\boldsymbol{v}(\boldsymbol{x}))  dS(\boldsymbol{x}) 
- \int_{\partial B_\epsilon(\boldsymbol{y})} \boldsymbol{T}^{+}_x (q^k(\boldsymbol{x},\boldsymbol{y}),\boldsymbol{u}^k(\boldsymbol{x},\boldsymbol{y}))\cdot
 \gamma^+\boldsymbol{v}(\boldsymbol{x}) dS(\boldsymbol{x}) 
\\
+\int_{\partial \Omega} \gamma^+\boldsymbol{u}^k(\boldsymbol{x},\boldsymbol{u})\cdot
\boldsymbol{T}^{+}(p(\boldsymbol{x}),\boldsymbol{v}(\boldsymbol{x}))  dS(\boldsymbol{x})  
- \int_{\partial \Omega} \boldsymbol{T}^{+}_x (q^k(\boldsymbol{x},\boldsymbol{y}),\boldsymbol{u}^k(\boldsymbol{x},\boldsymbol{y}))\cdot
 \gamma^+\boldsymbol{v}(\boldsymbol{x}) dS(\boldsymbol{x}) 
\\
+\int_{\Omega_\epsilon(\boldsymbol{y})}\,R_{kj}(\boldsymbol{x},\boldsymbol{y}) {v}_j(\boldsymbol{x})d\boldsymbol{x}
- \int_{\Omega_\epsilon(\boldsymbol{y})}\,q^k(\boldsymbol{x},\boldsymbol{y})\div\boldsymbol{v}(\boldsymbol{x})d\boldsymbol{x} 
 =  
 \int_{\Omega_\epsilon(\boldsymbol{y})}\bs{\mathcal A}( p, \boldsymbol{v})\cdot \boldsymbol{u}^k(\boldsymbol{x},\boldsymbol{y}) \,d\boldsymbol{x}.
 \end{multline}
Since all the functions in \eqref{4.G3tile} are smooth, the canonical conormal derivatives coincide with the classical ones, given by \eqref{ch2Tcl}, and it is easy to show that when $\epsilon\to 0$, the first integral in \eqref{4.G3tile} tends to 0, the second tends to $-v_k(\boldsymbol{y})$, while integrands in the remaining domain integrals are weakly singular and these integrals tend to the corresponding improper integrals, which leads us to \eqref{ch2vrepresentationA} for  
$(p,\bs{v})\in\mathcal D(\overline{\Omega})\times \boldsymbol{\mathcal D}(\overline{\Omega})\subset  \bs{H}^{1,0}(\Omega;  \boldsymbol{\mathcal{A}})$. 

(ii) Let us now prove the pressure identity \eqref{ch2representationp} 
for $(p,\bs{v})\in\mathcal D(\overline{\Omega})\times \boldsymbol{\mathcal D}(\overline{\Omega})$.
One can do this using the second Green identity similar to  \eqref{4.G3tile} but we will employ a slightly different approach.
Multiplying equation \eqref{ch2operatorA} by the fundamental pressure vector $\mathring{q}^j(\boldsymbol{x},\boldsymbol{y})$, integrating over the domain $\Omega$ and writing it as the bilinear form, which will be then treated in the sense of distributions, we obtain 
\begin{equation}\label{ch2gip1}
\left\langle\mathring{q}^j(\cdot,\boldsymbol{y}),\partial_{i}\left[\mu
(\partial_i v_{j} + \partial_j v_{i}
-\frac{2}{3}\delta_{i}^{j}\div\boldsymbol{v})\right]\right\rangle_{\Omega} 
-\left\langle \mathring{q}^j(\cdot,\boldsymbol{y}),\partial_j p\right\rangle_{\Omega}  = \left\langle\mathring{q}^j(\cdot,\boldsymbol{y}),\mathcal{A}_{j}(p, \bs{v})\right\rangle_{\Omega}.
\end{equation}
Applying the first Green identity to the first term{, we have,} 
\begin{align}\label{ch2gip2}
&\left< \mathring{q}^j(\cdot,\boldsymbol{y})\, , 
\, \partial_i\left(\mu\left(\partial_i v_{j} + \partial_j v_{i}
-\frac{2}{3}\delta_{i}^{j}\div\boldsymbol{v}\right)\right)\right>_{\Omega}  = \nonumber \\
&- \left< \partial_i \mathring{q}^j(\cdot,\boldsymbol{y})\, , 
\,\mu\left(\partial_i v_{j} + \partial_j v_{i}
-\frac{2}{3}\delta_{i}^{j}\div\boldsymbol{v}\right) \right>_{\Omega} \nonumber \\
& + \left<\mathring{q}^j(\cdot,\boldsymbol{y}) \,,
\,\mu\left(\partial_i v_{j} + \partial_j v_{i}
-\frac{2}{3}\delta_{i}^{j}\div\boldsymbol{v}\right) n_{j}\right>_{\partial\Omega},
\end{align}
and also in the second term 
\begin{align}\label{ch2gip3}
\left\langle \mathring{q}^j(\cdot,\boldsymbol{y}),\partial_j p\right\rangle_{\Omega} 
= - \left\langle \partial_j \mathring{q}^j(\cdot,\boldsymbol{y}),p\right\rangle_{\Omega}
+ \left<\mathring{q}^j(\cdot,\boldsymbol{y})\, , \, p\, n_{j}\right>_{\partial\Omega}
= { p(\boldsymbol{y})}
+ \left<\mathring{q}^j(\cdot,\boldsymbol{y})\, , \, p\, n_{j}\right>_{\partial\Omega},
\end{align}
where we took into account that by \eqref{dF-q} we have
\begin{equation}\label{ch2gip4bis}
\left<
\partial_j \mathring{q}^j(\cdot,\boldsymbol{y})\,,\, p\right>_{\Omega} =  -p(\boldsymbol{y}). 
\end{equation}
Substituting \eqref{ch2gip2} and \eqref{ch2gip3} into \eqref{ch2gip1} and rearranging terms we get
\begin{align}\label{ch2gip4}
\left<\mathring{q}^j(\cdot,\boldsymbol{y})\, , \,\mathcal{A}_{j}(p, \bs{v})\right>_{\Omega}
=&- \left< \partial_i \mathring{q}^j(\cdot,\boldsymbol{y})\, ,\,\mu\left(\partial_i v_{j} + \partial_j v_{i}
-\frac{2}{3}\delta_{i}^{j}\div\boldsymbol{v}\right) \right>_{\Omega} \nonumber \\
& + \left<\mathring{q}^j(\cdot,\boldsymbol{y}) \,, \,\mu\left(\partial_i v_{j} + \partial_j v_{i}
-\frac{2}{3}\delta_{i}^{j}\div\boldsymbol{v}\right) n_{j}\right>_{\partial\Omega}\nonumber\\
&+ \left\langle \partial_j \mathring{q}^j(\cdot,\boldsymbol{y}),p\right\rangle_{\Omega}
- \left<\mathring{q}^j(\cdot,\boldsymbol{y})\, , \, p\, n_{j}\right>_{\partial\Omega}.
\end{align}
By \eqref{ch2Tcl} we obtain
\begin{equation}
\left<\mathring{q}^j(\cdot,\boldsymbol{y}) \,, \,\mu\left(\partial_i v_{j} + \partial_j v_{i}
-\frac{2}{3}\delta_{i}^{j}\div\boldsymbol{v}\right) n_{j}\right>_{\partial\Omega} - \left<\mathring{q}^j(\cdot,\boldsymbol{y})\, , \, p\, n_{j}\right>_{\partial\Omega}=
 \left< \mathring{q}^j(\cdot,\boldsymbol{y})\, , \,T^{c+}_{j}(p, \bs{v})\right>_{\partial\Omega}. \label{ch2gip4uni}
\end{equation}
Let us now simplify the first term in the right hand side of \eqref{ch2gip4}
using the symmetry 
$\partial_{ x_{i}} \mathring{q}^j(\boldsymbol{x},\boldsymbol{y})
= \partial_{ x_{j}} \mathring{q}^{i}(\boldsymbol{x},\boldsymbol{y})$ 
and \eqref{dF-q}.
Then,  
\begin{align}
&\left< \partial_i \mathring{q}^j(\cdot,\boldsymbol{y})\, ,\,\mu\left(\partial_i v_{j} + \partial_j v_{i}
-\frac{2}{3}\delta_{i}^{j}\div\boldsymbol{v}\right) \right>_{\Omega} = 
2 \left< \partial_i \mathring{q}^j(\cdot,\boldsymbol{y})\, , \, \mu\partial_j v_{i}\right>_{\Omega} +\frac{2}{3}\mu(\boldsymbol{y})\div\boldsymbol{v}(\boldsymbol{y}).\label{ch2gip6}
\end{align}
Applying again the first Green identity to the first term in the right hand side of \eqref{ch2gip6}, we obtain 
\begin{align}
&\left< \partial_i \mathring{q}^j(\cdot,\boldsymbol{y})\, , \, \mu\partial_j v_i\right>_\Omega
=  \left< \partial_i \mathring{q}^j(\cdot,\boldsymbol{y})\, , \, \mu n_j\gamma^+v_i\right>_{\partial\Omega}  
-\left< \partial_i \mathring{q}^j(\cdot,\boldsymbol{y})\, , \, v_i\partial_j\mu \right>_{\Omega}   
 - \left< \partial_i\partial_j\mathring{q}^j(\cdot,\boldsymbol{y})\, , \, v_i\mu \right>_{\Omega}\nonumber \\
 &\qquad=  \left< \partial_i \mathring{q}^j(\cdot,\boldsymbol{y})\, , \, \mu n_j\gamma^+v_i\right>_{\partial\Omega}  
 -\left< \partial_i \mathring{q}^j(\cdot,\boldsymbol{y})\, , \, v_i\partial_j\mu \right>_{\Omega}   
  -v_i(\boldsymbol{y})\partial_i\mu(\boldsymbol{y}) 
  -\mu(\boldsymbol{y})\div\boldsymbol{v}(\boldsymbol{y}).\label{ch2gip9}
\end{align}
Now, plug \eqref{ch2gip9} into \eqref{ch2gip6}, 
\begin{multline}
\left< \partial_i \mathring{q}^j(\cdot,\boldsymbol{y})\, ,\,\mu\left(\partial_i v_{j} + \partial_j v_{i}
-\frac{2}{3}\delta_{i}^{j}\div\boldsymbol{v}\right) \right>_{\Omega} = 
2 \left< \partial_i \mathring{q}^j(\cdot,\boldsymbol{y})\, , \, \mu n_j\gamma^+v_i\right>_{\partial\Omega} \\ 
 -2\left< \partial_i \mathring{q}^j(\cdot,\boldsymbol{y})\, , \, v_i\partial_j\mu \right>_{\Omega}   
  -2v_i(\boldsymbol{y})\partial_i\mu(\boldsymbol{y}) 
  -\frac{4}{3}\mu(\boldsymbol{y})\div\boldsymbol{v}(\boldsymbol{y}).\label{ch2gip10}
\end{multline}
Now, substitute \eqref{ch2gip10}, \eqref{ch2gip4bis} and \eqref{ch2gip4uni} into \eqref{ch2gip4}. As a result, we obtain 
\begin{multline}
\left<\mathring{q}^j(\cdot,\boldsymbol{y})\, , \, \mathcal{A}_{j}(p, \bs{v})\,\right>_{\Omega}
=-2 \left< n_j\partial_j \mathring{q}^i(\cdot,\boldsymbol{y})\, , \, \mu \gamma^+v_i\right>_{\partial\Omega}  
 +2\left< \partial_i \mathring{q}^j(\cdot,\boldsymbol{y})\, , \, v_i\partial_j\mu \right>_{\Omega} \\  
  {+2v_i(\boldsymbol{y})\partial_i\mu(\boldsymbol{y}) 
  +\frac{4}{3}\mu(\boldsymbol{y})\div\boldsymbol{v}(\boldsymbol{y})
-p(\boldsymbol{y}) }
+ \langle \mathring{q}^j(\cdot,\boldsymbol{y})\, , \, T^{c+}_{j}(p, \bs{v})\rangle_{\partial\Omega}.
\end{multline}}
Rearranging the terms, taking into account that $T^{c+}_{j}(p, \bs{v})=T_{j}(p, \bs{v})$, and using the potential operator notations, we obtain \eqref{ch2representationp} for $(p,\bs{v})\in\mathcal D(\overline{\Omega})\times \boldsymbol{\mathcal D}(\overline{\Omega})$. 

Finally, if $(p,\bs{v})\in \bs{H}^{1,0}(\Omega;  \boldsymbol{\mathcal{A}})$ then by Theorem~\ref{densT}, the density argument and the mapping properties of the operators involved in \eqref{ch2representationp} and \eqref{ch2vrepresentationA}  extend these relations to any  $(p,\bs{v})\in \bs{H}^{1,0}(\Omega;  \boldsymbol{\mathcal{A}})$.
\end{proof}

If the couple $(  p,\bs{v})\in \bs{H}^{1,0}(\Omega;  \boldsymbol{\mathcal{A}})$ is a solution of the Stokes PDEs \eqref{ch2BVP1}-\eqref{ch2BVPdiv} with variable { viscosity} coefficient, then { the third Green identities \eqref{ch2representationp} and \eqref{ch2vrepresentationA}  reduce to}
\begin{align}\label{ch2GP}
&p+\mathcal{R}^{\bullet} \boldsymbol{v} - {\Pi^s}\boldsymbol{T}(p,\bs{v}) +{\Pi^d}\boldsymbol{\gamma}^{+}\boldsymbol{v}=\mathring{\mathcal{Q}}\boldsymbol{f}+\frac{4\mu}{3}g \quad \text{in}\quad \Omega,\\
\label{ch2GV}
&\boldsymbol{v}+\boldsymbol{\mathcal{R}}\boldsymbol{v}-\boldsymbol{V}\boldsymbol{T}^{+}(p,\bs{v})+\boldsymbol{W}\boldsymbol{\gamma}^{+}\boldsymbol{v}=\boldsymbol{\boldsymbol{\mathcal{U}f}} {-}\boldsymbol{\mathcal{Q}}g \quad\ \text{in}\quad \Omega.
\end{align}
We will also need the following trace and traction of the third Green identities for 
$(p,\bs{v})\in \bs{H}^{1,0}(\Omega;  \boldsymbol{\mathcal{A}})$  { on $\partial\Omega$},  
\begin{align}\label{ch2GG}
\frac{1}{2}\boldsymbol{\gamma}^{+}\boldsymbol{v}+\gamma^{+}\boldsymbol{\mathcal{R}}\boldsymbol{v}-
\boldsymbol{\mathcal{V}}\boldsymbol{T}^{+}(p,\bs{v})+\boldsymbol{\mathcal{W}}\gamma^{+}\boldsymbol{v}&=
\boldsymbol{\gamma}^{+}\boldsymbol{\mathcal{U}}\boldsymbol{f} {-}\boldsymbol{\gamma}^{+}\boldsymbol{\mathcal{Q}}g,
\\
\label{ch2GT}
\frac{1}{2}\boldsymbol{T}^{+}(p,\bs{v})+
\boldsymbol{T}^{+}(\mathcal{R}^{\bullet}, \boldsymbol{\mathcal{R}})\boldsymbol{v}
-\boldsymbol{\mathcal{W'}}\boldsymbol{T}^{+}(p,\bs{v})+\boldsymbol{\mathcal{L}}^{+}\gamma^{+}\boldsymbol{v}
&=\boldsymbol{\boldsymbol{T}^{+}}(\mathring{\mathcal{Q}}\boldsymbol{f}+\frac{4\mu}{3}g,\,\,\boldsymbol{\boldsymbol{\mathcal{U}f}}{-}\boldsymbol{\mathcal{Q}}g).
\end{align}
{ Note that the traction operators in \eqref{ch2GT} are well defined by virtue of the continuity of operators \eqref{ch2H10QU}-\eqref{ch2H10RR} in Theorem~\ref{ch2thmUR:theo} and operators \eqref{ch2H10PiW} in Theorem~\ref{ch2T3}. }

{ Let us now} prove the following three assertions that are instrumental for proving the equivalence of the { BDIE systems to the mixed BVP}.

\begin{theorem}\label{ch2L1}
Let $\boldsymbol{v}\in \boldsymbol{H}^{1}(\Omega)$, $p\in L_{2}(\Omega)$, $g\in L_{2}(\Omega)$, $\boldsymbol{f}\in \boldsymbol{L}_{2}(\Omega),$ $\boldsymbol{\boldsymbol{\Psi}}\in \boldsymbol{H}^{-1/2}({\partial\Omega})$ and $\boldsymbol{\boldsymbol{\Phi}}\in \boldsymbol{H}^{1/2}({\partial\Omega})$ satisfy the equations
\begin{align}
p+\mathcal{R}^{\bullet} \boldsymbol{v} - {\Pi^s}\boldsymbol{\Psi} +{\Pi^d}\boldsymbol{\Phi}&=\mathring{\mathcal{Q}}\boldsymbol{f}+\frac{4\mu}{3}g \hspace{0.5em}in\,\, \Omega,\label{ch2lemap}\\
\boldsymbol{v}+\boldsymbol{\mathcal{R}v}-\boldsymbol{V\Psi}+\boldsymbol{W\Phi}&=\boldsymbol{\boldsymbol{\mathcal{U}f}}{-}\boldsymbol{\mathcal{Q}}g \hspace{0.5em}in\,\, \Omega.\label{ch2lemav}
\end{align}
Then $(  p,\bs{v})\in \bs{H}^{1,0}(\Omega;\boldsymbol{\mathcal{A}})$ and solve the equations 
\begin{align}\label{PDEsol}
\boldsymbol{\mathcal{A}}( p, \boldsymbol{v})=\boldsymbol{f},\quad \div\bs{v}=g.
\end{align} 
Moreover, the following relations hold true:
\begin{align}
{\Pi^s}(\boldsymbol{\boldsymbol{\Psi}} - \boldsymbol{T}^{+}(p,\bs{v})) - {\Pi^d}(\boldsymbol{\boldsymbol{\Phi}} - \gamma^{+}\boldsymbol{v})&= 0\hspace{1em}\text{in}\,\,\Omega,\label{ch2lemarel2}\\
\boldsymbol{V}(\boldsymbol{\Psi} - \boldsymbol{T}^{+}(p,\bs{v})) - \boldsymbol{W}(\boldsymbol{\boldsymbol{\Phi}} - \gamma^{+}\boldsymbol{v}) &= \boldsymbol{0}\hspace{1em}\text{in}\,\,\Omega\label{ch2lemare1}.
\end{align}
\end{theorem}
\begin{proof}
 First, the embedding $(p,\bs{v})\in \bs{H}^{1,0}(\Omega;\boldsymbol{\mathcal{A}})$ is implied by continuity of operators \eqref{ch2OpC8b}-\eqref{ch2H10RR} in Theorem \ref{ch2thmUR:theo} and operators in \eqref{ch2H10PiW} in Theorem \ref{ch2T3}. 
Hence the third Green identities \eqref{ch2representationp} and \eqref{ch2vrepresentationA}
hold true. 
Subtracting from them equations \eqref{ch2lemap} and \eqref{ch2lemav} respectively we obtain
\begin{align}
{\Pi^d}\boldsymbol{\Phi}^{*}-{\Pi^s}\boldsymbol{\Psi}^{*}
&=\mathring{\mathcal{Q}}(\boldsymbol{\mathcal{A}}(p,\bs{v})-\boldsymbol{f})
+\frac{4\mu}{3}(\div\boldsymbol{v}-g),\label{ch2c118}\\
\boldsymbol{W}\boldsymbol{\boldsymbol{\Phi}}^{*}-\boldsymbol{V}\boldsymbol{\boldsymbol{\Psi}}^{*}
&=\boldsymbol{\mathcal{U}}(\boldsymbol{\mathcal{A}}(p,\bs{v})-\boldsymbol{f})
{-}\boldsymbol{\mathcal{Q}}(\div\boldsymbol{v}-g).\label{ch2c117}
\end{align}
where $\boldsymbol{\Psi}^{*}:=\boldsymbol{T}^{+}(p,\bs{v})-\boldsymbol{\Psi}$, and $\boldsymbol{\boldsymbol{\Phi}}^{*}=\boldsymbol{\gamma}^{+}\boldsymbol{v}-\boldsymbol{\boldsymbol{\Phi}}$. 


After multiplying \eqref{ch2c117} by $\mu$ and applying  relations 
\eqref{ch2relationU} and \eqref{ch2relationVW}, 
we arrive at
\begin{align}
\mathring{W}(\mu\boldsymbol{\boldsymbol{\Phi}}^{*})-\mathring{V}\boldsymbol{\boldsymbol{\Psi}}^{*}
&=\mathring{\boldsymbol{\mathcal{U}}}(\boldsymbol{\mathcal{A}}(p,\bs{v})-\boldsymbol{f})
{-}\mathring{\boldsymbol{\mathcal{Q}}}(\mu(\div\boldsymbol{v}-g)).
\label{ch2c117m}
\end{align}
Applying the divergence operator to both sides of \eqref{ch2c117m} and taking into account that the potentials $\mathring{\boldsymbol{\mathcal{U}}}, \mathring{\boldsymbol{V}}$, 
  and $\mathring{\boldsymbol{W}}$ are divergence free, while for $\mathring{\boldsymbol{\mathcal{Q}}}$ we have  equation \eqref{dQ}, we obtain
\begin{align}
0  &= {-}\div\mathring{\boldsymbol{\mathcal{Q}}}(\mu(\div\boldsymbol{v}-g))=\mu(\div\boldsymbol{v}-g),
 \end{align}
which implies the second equation in \eqref{PDEsol}.
Then equations \eqref{ch2c118} and \eqref{ch2c117m} reduce to
\begin{align*}
\mathring{{\Pi^d}}(\mu\boldsymbol{\boldsymbol{\Phi}}^{*})- \mathring{\Pi}^s\boldsymbol{\boldsymbol{\Psi}}^{*}&= \mathring{\mathcal{Q}}(\boldsymbol{\mathcal{A}}(p,\bs{v})-\boldsymbol{f}),\\
\mathring{W}(\mu\boldsymbol{\boldsymbol{\Phi}}^{*})-\mathring{V}\boldsymbol{\boldsymbol{\Psi}}^{*}&=\mathring{\boldsymbol{\mathcal{U}}}(\boldsymbol{\mathcal{A}}(p,\bs{v})-\boldsymbol{f}).
\end{align*}
Applying the Stokes operator with $\mu=1$ to these two equations, by \eqref{A0V0} and \eqref{A0U0} we obtain $\boldsymbol{\mathcal{A}}(p,\bs{v})-\boldsymbol{f}=\bs{0}$
and hence, the first equation in \eqref{PDEsol}. 

Finally, relations \eqref{ch2lemare1} and \eqref{ch2lemarel2} follow from the substitution of \eqref{PDEsol}
in \eqref{ch2c118} and \eqref{ch2c117}.  
\end{proof}

\begin{lemma}\label{ch2lemma2}
Let ${\partial\Omega}=\overline{S}_{1}\cup \overline{S}_{2}$, where $S_{1}$ and $S_{2}$ are open  non-empty non-intersecting simply connected submanifolds of ${\partial\Omega}$ with infinitely smooth boundaries. Let $\boldsymbol{\boldsymbol{\Psi}}^{*}\in \widetilde{\boldsymbol{H}}^{-1/2}(S_{1})$, $\boldsymbol{\Phi}^{*}\in \widetilde{\boldsymbol{H}}^{1/2}(S_{2})$. If 
\begin{equation}\label{ch2lema2iii}
{\Pi^s}(\boldsymbol{\Psi}^{*}) -{\Pi^d}(\boldsymbol{\Phi}^{*}) = {0},\quad \quad\boldsymbol{V}\boldsymbol{\Psi}^{*}(\boldsymbol{x})-\boldsymbol{W}\boldsymbol{\Phi}^{*}(\boldsymbol{x})=\boldsymbol{0}
\quad \text{in } \Omega,
\end{equation}
then $\boldsymbol{\boldsymbol{\Psi}}^{*} = \boldsymbol{0}$, and $\boldsymbol{\boldsymbol{\Phi}}^{*} = \boldsymbol{0}$, on ${\partial\Omega}$. 
\end{lemma}

\begin{proof}

Multiplying the second equation in \eqref{ch2lema2iii} by $\mu$ and applying relations \eqref{ch2relationVW}, \eqref{ch2relationP}, we obtain 
\begin{equation}\label{ch2lema2.3.1}
\mathring{\Pi}^s(\boldsymbol{\Psi}^{*}) -\mathring{\Pi}^d(\mu\boldsymbol{\Phi}^{*}) = {0},\quad
\mathring{\boldsymbol{V}}\boldsymbol{\Psi}^{*}(\boldsymbol{x})-\mathring{\boldsymbol{W}}(\mu\boldsymbol{\Phi}^{*}(\boldsymbol{x}))=\boldsymbol{0}.
\end{equation} 
Let us take the trace of the second equation in \eqref{ch2lema2.3.1} restricting it to $S_{1}$ and  take the traction with the constant coefficient $\mu=1$ of both equations in \eqref{ch2lema2.3.1} restricting it to $S_{2}$. Keeping in mind the jump relations given in  Theorem \ref{ch2jumps} and notation \eqref{ch2relationL}, we arrive at the system of equations
\begin{align*}
r_{S_{1}}\mathring{\boldsymbol{\mathcal{V}}}{\boldsymbol{\Psi}}(\boldsymbol{x})-r_{S_{1}}\mathring{\boldsymbol{\mathcal{W}}}\boldsymbol{\Phi}^*(\boldsymbol{x})=\boldsymbol{0},\hspace{0.5em}&\text{on}\hspace{0.5em} S_{1},\\ 
r_{S_{2}}\mathring{\boldsymbol{\mathcal{W'}}}\boldsymbol{\Psi}^*(\boldsymbol{x})-r_{S_{2}}\mathring{\boldsymbol{\mathcal{L}}}\widehat{\boldsymbol{\Phi}}(\boldsymbol{x})=\boldsymbol{0},\hspace{0.5em}&\text{on}\hspace{0.5em} S_{2},
\end{align*}
where 
$\widehat{\boldsymbol{\Phi}}:=\mu\boldsymbol{\Phi}^{*}$.

This BIE system has been studied in \cite[Theorem 3.10]{kohr1} (see Theorem~\ref{ch2thinvM11ring} below) which implies that it has  only the trivial solution, $\boldsymbol{\Psi}^* = \boldsymbol{0}$, $\widehat{\boldsymbol{\Phi}}=\boldsymbol{0}$.  
\end{proof}

 \section{ BDIE systems}

We aim to obtain two different { BDIE systems} for mixed BVP \eqref{ch2BVPM} { following the procedure similar to the one employed for a scalar PDE in  \cite{CMN-1}, \cite{carlos1} and \cite{MikEABEM2002} and} references therein. 

 To this end, let the functions $\boldsymbol{\Phi}_{0}\in \boldsymbol{H}^{1/2}({\partial\Omega})$ and $\boldsymbol{\Psi}_{0}\in \boldsymbol{H}^{-1/2}({\partial\Omega})$ be some continuations of the boundary functions 
$\boldsymbol{\varphi}_{0}\in \boldsymbol{H}^{1/2}({\partial\Omega}_{D})$ and 
$\boldsymbol{\psi}_{0}\in \boldsymbol{H}^{-1/2}({\partial\Omega}_{N})$ from \eqref{ch2BVPD} and \eqref{ch2BVPN}.
Let us now represent 
\begin{equation}\label{ch2gTrepr}
\gamma^{+}\boldsymbol{v}=\boldsymbol{\Phi}_{0} + \boldsymbol{\varphi},\quad 
\boldsymbol{T}^{+}(p,\bs{v}) = \boldsymbol{\Psi}_{0} +\boldsymbol{\psi} \text{ on } {\partial\Omega}, 
\end{equation}
where  $\boldsymbol{\varphi}\in\widetilde{\boldsymbol{H}}^{1/2}({\partial\Omega}_{N})$ and $\boldsymbol{\psi}\in\widetilde{\boldsymbol{H}}^{-1/2}({\partial\Omega}_{D})$ are unknown boundary functions. 
\subsection{ BDIE system (M11$_*$)}
Let us now take equations \eqref{ch2GP} and \eqref{ch2GV} in the domain $\Omega$ and restrictions of equations \eqref{ch2GG} and \eqref{ch2GT} to the boundary parts ${\partial\Omega}_{D}$ and ${\partial\Omega}_{N}$, respectively. Substituting there representations \eqref{ch2gTrepr} and considering further the unknown boundary functions $\boldsymbol{\varphi}$ and $\boldsymbol{\psi}$ as formally independent of (segregated from) the unknown domain functions $p$ and $\boldsymbol{v}$, we obtain the following system (M11$_*$)  of four boundary-domain integral equations for four unknowns, 
$(p,\bs{v})\in \bs{H}^{1,0}(\Omega,\boldsymbol{\mathcal{A}})$,  $\boldsymbol{\varphi}\in\widetilde{\boldsymbol{H}}^{1/2}({\partial\Omega}_{N})$ and $\boldsymbol{\psi}\in\widetilde{\boldsymbol{H}}^{-1/2}({\partial\Omega}_{D})$:
\begin{subequations}
\label{ch2M11}
\begin{align}
p+\mathcal{R}^{\bullet} \boldsymbol{v} - {\Pi^s}\boldsymbol{\psi} +{\Pi^d}\boldsymbol{\varphi}&=
F_0,\quad\text{in } \Omega,\label{ch2M11p}\\
\boldsymbol{v}+\boldsymbol{\mathcal{R}v}-\boldsymbol{V\psi}+\boldsymbol{W\varphi}&=
\boldsymbol{F},\quad\text{in } \Omega,\label{ch2M11v}\\
r_{{\partial\Omega}_{D}}\boldsymbol{\gamma}^{+}\boldsymbol{\mathcal{R}v}-r_{{\partial\Omega}_{D}}\boldsymbol{\mathcal{V}\psi}
+r_{{\partial\Omega}_{D}}\boldsymbol{\mathcal{W}\varphi}&=
r_{{\partial\Omega}_{D}}\boldsymbol{\gamma}^{+}\boldsymbol{F}-\boldsymbol{\varphi}_{0},\quad\text{on } {\partial\Omega}_{D}, \label{ch2M11G}\\
r_{{\partial\Omega}_{N}}\boldsymbol{T}^{+}( \mathcal{R}^{\bullet},  \boldsymbol{\mathcal{R}})\boldsymbol{v}
-r_{{\partial\Omega}_{N}}\boldsymbol{\mathcal{W'}\psi}+r_{{\partial\Omega}_{N}}\boldsymbol{\mathcal{L}}^{+}\boldsymbol{\varphi}&=
r_{{\partial\Omega}_{N}}\boldsymbol{T}^{+}( F_{0}, \boldsymbol{F})-\boldsymbol{\psi}_{0}, \quad\text{on } {\partial\Omega}_{N},\label{ch2M11T}
\end{align}
\end{subequations}
where
\begin{align}
F_0=\mathring{\mathcal{Q}}\boldsymbol{f}+\dfrac{4}{3}g\mu+{\Pi^s}\boldsymbol{\Psi}_{0}-{\Pi^d}\boldsymbol{\Phi}_{0},\quad
\boldsymbol{F}=\boldsymbol{\boldsymbol{\mathcal{U}f}}{-}\boldsymbol{\mathcal{Q}}g
+\boldsymbol{V\Psi}_{0}-\boldsymbol{W\Phi}_{0}. \label{ch2F0F}
\end{align}
By  Theorems \ref{ch2thmUR:theo} and \ref{ch2T3}, $(F_{0},\boldsymbol{F})\in \bs{H}^{1,0}(\Omega,\boldsymbol{\mathcal{A}})$ and hence $\bs{T}(F_{0},\boldsymbol{F})$ is well defined.

{
Let us denote the right hand side of BDIE system \eqref{ch2M11} as
\begin{align}\label{ch2F11def}
\mathcal{F}^{11}_{*}&=
[F_{0}, \boldsymbol{F}, r_{{\partial\Omega}_{D}}\gamma^{+}\boldsymbol{F}-\boldsymbol{\varphi}_{0}, r_{{\partial\Omega}_{N}}\boldsymbol{T}^{+}( F_{0} , \bs{F})-\boldsymbol{\psi}_{0}]. 
\end{align}
Then Theorems \ref{ch2thmUR:theo} and \ref{ch2T3} imply the inclusion
${\mathcal{F}^{11}_*}
\in \bs{H}^{1,0}(\Omega,\boldsymbol{\mathcal{A}})\times\boldsymbol{H}^{1/2}({\partial\Omega}_{D})\times 
 \boldsymbol{H}^{-1/2}({\partial\Omega}_{N}).
$ 

\begin{rem} \label{ch2remM11}
{ Let $\mathcal{F}_{*}^{11}$ be defined by \eqref{ch2F0F}, \eqref{ch2F11def}.
Then} $\mathcal F_*^{11}=\bs{0}$ if and only if 
$(\boldsymbol{f}, g, \boldsymbol{\Phi}_{0},\boldsymbol{\Psi}_{0})=\boldsymbol{0}$.

Indeed, from \eqref{ch2F0F} and \eqref{ch2F11def} we immediately obtain that $(\bs{f}, g,\boldsymbol{\Phi}_{0},\boldsymbol{\Psi}_{0})=0$ implies $\mathcal{F}^{11}_{*}=0$. 
Let us now prove that if $\mathcal{F}^{11}_{*}=0$ then $(\bs{f}, g,\boldsymbol{\Phi}_{0},\boldsymbol{\Psi}_{0})=0$. 
Lemma \ref{ch2L1} with $F_{0}=0$ for $p$ and $\bs{F}=0$ for $\bs{v}$ applied to equations \eqref{ch2F0F} implies that $\bs{f}=0$, $g=0$ and 
\begin{align}\label{PP0}
{\Pi^s}\bs{\Psi}_{0}- {\Pi^d}\bs{\Phi}_{0}&=0,\quad
\bs{V}\bs{\Psi}_{0}- \bs{W}\bs{\Phi}_{0}=0\quad \mbox{in } \Omega.
\end{align}
In addition, since $F_{0}=0$ and $\bs{F}=0$, we get from \eqref{ch2F11def} that
\begin{align*}
r_{{\partial\Omega}_{D}}\boldsymbol{\Phi}_{0}=\boldsymbol{\varphi}_{0}= \bs{0},\quad
r_{{\partial\Omega}_{N}}\boldsymbol{\Psi}_{0}=\boldsymbol{\psi}_{0}=\bs{0}.
\end{align*}
Consequently, $\bs{\Psi}_{0}\in \widetilde{H}^{-1/2}({\partial\Omega}_{D})$, 
$\bs{\Phi}_{0}\in \widetilde{H}^{1/2}({\partial\Omega}_{N})$ and by \eqref{PP0} and Lemma \ref{ch2lemma2} we  obtain that $\boldsymbol{\Psi}_{0}=0$ and $\boldsymbol{\Phi}_{0}=0$ on ${\partial\Omega}$. 
\end{rem}
\subsection{ BDIE system (M22$_*$)}
Let us  take equations \eqref{ch2GP} and \eqref{ch2GV} in the domain $\Omega$ and restrictions of equations \eqref{ch2GG} and \eqref{ch2GT} to the boundary parts ${\partial\Omega}_{N}$ and ${\partial\Omega}_{D}$ respectively. Substituting there representations \eqref{ch2gTrepr} and considering again the unknown boundary functions $\boldsymbol{\varphi}$ and $\boldsymbol{\psi}$ as formally independent of (segregated from) the unknown domain functions $p$ and $\boldsymbol{v}$, we obtain the following system (M22$_*$) of four BDIEs { for $(  p,\bs{v})\in \bs{H}^{1,0}(\Omega,\boldsymbol{\mathcal{A}})$,  $\boldsymbol{\varphi}\in\widetilde{\boldsymbol{H}}^{1/2}({\partial\Omega}_{N})$ and $\boldsymbol{\psi}\in\widetilde{\boldsymbol{H}}^{-1/2}({\partial\Omega}_{D})$,}
\begin{subequations}
\label{ch2M22}
\begin{align}
\label{ch2M22p}
p+\mathcal{R}^{\bullet} \boldsymbol{v} - {\Pi^s}\boldsymbol{\psi} +{\Pi^d}\boldsymbol{\varphi}&=
F_0\quad in \ \Omega,\\
\label{ch2M22v}
\boldsymbol{v}+\boldsymbol{\mathcal{R}v}-\boldsymbol{V\psi}+\boldsymbol{W\varphi}&=
\boldsymbol{F}\quad in \ \Omega,\\
\label{ch2M22D}
\hspace{-1em}\dfrac{1}{2}\boldsymbol{\psi}+r_{{\partial\Omega}_{D}}\boldsymbol{T}^{+}( \mathcal{R}^{\bullet},  \boldsymbol{\mathcal{R}})\boldsymbol{v}
-r_{{\partial\Omega}_{D}}\boldsymbol{\mathcal{W'}\psi}+r_{{\partial\Omega}_{D}}\boldsymbol{\mathcal{L}}^{+}\boldsymbol{\varphi}&=r_{{\partial\Omega}_{D}}\boldsymbol{T}^{+}( F_{0}, \boldsymbol{F})-r_{{\partial\Omega}_{D}}\boldsymbol{\boldsymbol{\Psi}}_{0}
\quad on \ {\partial\Omega}_{D},\\
\label{ch2M22N}
\dfrac{1}{2}\boldsymbol{\boldsymbol{\varphi}}+r_{{\partial\Omega}_{N}}\boldsymbol{\gamma}^{+}\boldsymbol{\mathcal{R}v}-r_{{\partial\Omega}_{N}}\boldsymbol{\mathcal{V}\psi}
+r_{{\partial\Omega}_{N}}\boldsymbol{\mathcal{W}\varphi}
&=r_{{\partial\Omega}_{N}}\boldsymbol{\gamma}^{+}\boldsymbol{F}-r_{{\partial\Omega}_{N}}\boldsymbol{\boldsymbol{\Phi}}_{0}\quad on \ {\partial\Omega}_{N},
\end{align}
\end{subequations}
{ where the terms $F_{0}$ and $\boldsymbol{F}$  in the right hand side are given by \eqref{ch2F0F}.

Let us denote the right hand side of BDIE system \eqref{ch2M22} as
\begin{align}\label{7.6}
\mathcal{F}_{*}^{22}
&=[F_{0},\boldsymbol{F}, 
r_{{\partial\Omega}_{D}}\boldsymbol{T}^{+}( F_{0}, \boldsymbol{F})
 -r_{{\partial\Omega}_{D}}\boldsymbol{\boldsymbol{\Psi}}_{0}, 
 r_{{\partial\Omega}_{N}}\boldsymbol{\gamma}^{+}\boldsymbol{F}
  -r_{{\partial\Omega}_{N}}\boldsymbol{\boldsymbol{\Phi}}_{0} ].
\end{align}
Then Theorems \ref{ch2thmUR:theo} and \ref{ch2T3} imply the inclusion
$
\mathcal{F}_{*}^{22}\in \bs{H}^{1,0}(\Omega,\boldsymbol{\mathcal{A}})\times \boldsymbol{H}^{-1/2}({\partial\Omega}_{D})\times \boldsymbol{H}^{1/2}({\partial\Omega}_{N}).
$ 
}

Note that the BDIE system \eqref{ch2M22p}-\eqref{ch2M22N} can be split into the BDIE system (M22) of 3 vector equations, \eqref{ch2M22v}-\eqref{ch2M22N}, for 3 vector unknowns, $\boldsymbol{v}$, $\boldsymbol{\psi}$ and $\boldsymbol{\varphi}$, and the separate equation \eqref{ch2M22p} that can be used, after solving the system, to obtain the pressure, $p$. However, since the couple $(p, \boldsymbol{v})$ shares the space $\bs{H}^{1,0}(\Omega,\boldsymbol{\mathcal{A}})$, equations \eqref{ch2M22v}, \eqref{ch2M22D} and \eqref{ch2M22N} are not completely separate from equation \eqref{ch2M22p}.

\begin{rem}\label{ch2remM22} 
{ Let $\mathcal{F}_{*}^{22}$ be given by \eqref{ch2F0F}, \eqref{7.6}. 
Then} $\mathcal F_*^{22}=\bs{0}$ if and only if $(\bs{f},g,\boldsymbol{\Phi}_0,\boldsymbol{\Psi}_0)=0$.
Indeed, it is evident that $(\bs{f}, g,\boldsymbol{\Phi}_{0},\boldsymbol{\Psi}_{0})=0$ implies $\mathcal{F}_{*}^{22}=0$. 
Let now $\mathcal{F}^{22}_{*}=0.$ 
Lemma \ref{ch2L1} with $F_{0}=0$ for $p$ and $\bs{F}=0$ for $\bs{v}$ applied to equations \eqref{ch2F0F} implies that $\bs{f}=0$, $g=0$ and 
\begin{align}\label{6.8}
{\Pi^s}\bs{\Psi}_{0}- {\Pi^d}\bs{\Phi}_{0}&=0,\quad
\bs{V}\bs{\Psi}_{0}- \bs{W}\bs{\Phi}_{0}=0\quad \mbox{in } \Omega.
\end{align}
In addition, since $F_{0}=0$ and $\bs{F}=0$, we get from \eqref{7.6} that
\begin{align*}
 & r_{{\partial\Omega}_{D}}\boldsymbol{\Psi}_{0}=0,\quad
 r_{{\partial\Omega}_{N}}\boldsymbol{\Phi}_{0}=0.
\end{align*}
Consequently, $\bs{\Psi}_{0}\in \widetilde{H}^{-1/2}({\partial\Omega}_{N})$ and $\bs{\Phi}_{0}\in \widetilde{H}^{1/2}({\partial\Omega}_{D})$ . Therefore by \eqref{6.8} and Lemma \ref{ch2lemma2} we obtain that $\boldsymbol{\Psi}_{0}=0$ and $\boldsymbol{\Phi}_{0}=0$ on ${\partial\Omega}$. 
\end{rem}

\subsection{Equivalence of BDIE Systems and BVP}\label{6.3}
\begin{theorem}[Equivalence Theorem]\label{ch2thEQm11}
Let $\boldsymbol{f}\in \boldsymbol{L}_{2}(\Omega)$, $g\in L_{2}(\Omega)$ and let $\boldsymbol{\Phi}_{0}\in \boldsymbol{H}^{-1/2}({\partial\Omega})$ and  $\boldsymbol{\Psi}_{0}\in \boldsymbol{H}^{-1/2}({\partial\Omega})$ be some fixed extensions of $\boldsymbol{\varphi}_{0}\in \boldsymbol{H}^{1/2}({\partial\Omega}_{D})$ and $\boldsymbol{\psi}_{0}\in \boldsymbol{H}^{-1/2}({\partial\Omega}_{N})$ respectively. 
\begin{enumerate}
\item[(i)] If a couple $(p,\bs{v})\in L_{2}(\Omega)\times \bs{H}^{1}(\Omega)$ solves the mixed BVP \eqref{ch2BVPM}, then the set
$(p, \boldsymbol{v}, \boldsymbol{\psi}, \boldsymbol{\varphi})$
where
\begin{equation}\label{ch2eqcond}
\boldsymbol{\varphi}=\gamma^{+}\boldsymbol{v}-\boldsymbol{\Phi}_{0}, \quad
\boldsymbol{\psi}=\boldsymbol{T}^{+}(p,\bs{v})-\boldsymbol\Psi{}_{0}\quad \mbox{on } \partial\Omega,
\end{equation}
belongs to
$\bs{H}^{1,0}(\Omega;\boldsymbol{\mathcal{A}}) 
\times\widetilde{\boldsymbol{H}}^{-1/2}({\partial\Omega}_{D})
\times\widetilde{\boldsymbol{H}}^{1/2}({\partial\Omega}_{N})$ 
and
solves $BDIE$ systems \eqref{ch2M11} and \eqref{ch2M22}. 
\item[(ii)] If a set
$(p, \boldsymbol{v}, \boldsymbol{\psi}, \boldsymbol\varphi{} )\in 
L_{2}(\Omega)\times \bs{H}^{1}(\Omega) \times\widetilde{\boldsymbol{H}}^{-1/2}({\partial\Omega}_{D})\times\widetilde{\boldsymbol{H}}^{1/2}({\partial\Omega}_{N})$ 
solves one of BDIE systems, \eqref{ch2M11} or  \eqref{ch2M22}, then it solves the other BDIE system, the couple  $(p,\bs{v})$ belongs to $\bs{H}^{1,0}(\Omega;\boldsymbol{\mathcal{A}})$ and  solves mixed BVP \eqref{ch2BVPM}, while $\boldsymbol{\psi}, \boldsymbol{\varphi}$ satisfy \eqref{ch2eqcond}.  
\item[(iii)] Both BDIE systems, \eqref{ch2M11} and \eqref{ch2M22}, are uniquely solvable for
$(p, \boldsymbol{v}, \boldsymbol{\psi}, \boldsymbol\varphi{} )\in L_{2}(\Omega)\times \bs{H}^{1}(\Omega) \times\widetilde{\boldsymbol{H}}^{-1/2}({\partial\Omega}_{D})\times\widetilde{\boldsymbol{H}}^{1/2}({\partial\Omega}_{N})$. 
\end{enumerate}
\end{theorem}
\begin{proof}
(i) Let $(p,\bs{v})\in L_{2}(\Omega)\times \bs{H}^{1}(\Omega)$ be a solution of BVP \eqref{ch2BVPM}. 
Since $\boldsymbol{f}\in \boldsymbol{L}_{2}(\Omega)$ then $(  p,\bs{v})\in \bs{H}^{1,0}(\Omega;\boldsymbol{\mathcal{A}})$. 
Let us define the functions $\boldsymbol{\varphi}$ and $\boldsymbol{\psi}$ by \eqref{ch2eqcond}. By the BVP boundary conditions,  $\gamma^{+}\boldsymbol{v}=\boldsymbol{\varphi}_{0}=\boldsymbol{\Phi}_{0}$ on ${\partial\Omega}_D$ and $\boldsymbol{T}^{+}(p,\bs{v})=\boldsymbol\psi{}_{0}=\boldsymbol\Psi{}_{0}$  on ${\partial\Omega}_N$.
Then \eqref{ch2eqcond} implies that $(\boldsymbol{\psi},\boldsymbol{\varphi}) \in \widetilde{\boldsymbol{H}}^{-1/2}({\partial\Omega}_{D})\times\widetilde{\boldsymbol{H}}^{1/2}({\partial\Omega}_{N})$. Taking into account the third Green identities \eqref{ch2GP}-\eqref{ch2GT}, we immediately obtain that 
$(p, \boldsymbol{v},\boldsymbol{\varphi}, \boldsymbol{\psi})$ solves BDIE systems \eqref{ch2M11} and and \eqref{ch2M22}.

(ii-11) Let 
$(p, \boldsymbol{v}, \boldsymbol{\psi}, \boldsymbol\varphi)\in 
L_{2}(\Omega)\times \bs{H}^{1}(\Omega) \times\widetilde{\boldsymbol{H}}^{-1/2}({\partial\Omega}_{D})\times\widetilde{\boldsymbol{H}}^{1/2}({\partial\Omega}_{N})$  
solve  BDIE system \eqref{ch2M11}. 
Then equations \eqref{ch2M11p}, \eqref{ch2M11v} and Theorems \ref{ch2thmUR:theo}, \ref{ch2T3} imply that 
$(p, \boldsymbol{v}, \boldsymbol{\psi}, \boldsymbol\varphi)\in 
\bs{H}^{1,0}(\Omega;\boldsymbol{\mathcal{A}}) \times\widetilde{\boldsymbol{H}}^{-1/2}({\partial\Omega}_{D})
\times\widetilde{\boldsymbol{H}}^{1/2}({\partial\Omega}_{N})$  and the canonical conormal derivative $\boldsymbol{T}^+(p,\bs{v})$ is well defined. 
If we take the trace of \eqref{ch2M11v} restricted to ${\partial\Omega}_{D}$, use the jump relations for the trace of $\bs{V}$ and $\bs{W}$, see Theorem \ref{ch2jumps}, and subtract it from \eqref{ch2M11G}, we arrive at
$r_{{\partial\Omega}_{D}}\gamma^{+}\boldsymbol{v} -\dfrac{1}{2}r_{{\partial\Omega}_{D}}\boldsymbol{\varphi}=\boldsymbol{\varphi}_{0}$ on ${\partial\Omega}_{D}.$
Since $\boldsymbol{\varphi}$ vanishes on ${\partial\Omega}_{D}$,  the Dirichlet condition \eqref{ch2BVPD} is satisfied. 

Repeating the same procedure but now taking the traction of \eqref{ch2M11p} and \eqref{ch2M11v}, restricted to ${\partial\Omega}_{N}$, using the jump relations for the traction of $({\Pi^d},\bs{W})$ and subtracting it from \eqref{ch2M11T}, we arrive at
$r_{{\partial\Omega}_{N}}\boldsymbol{T}(p,\bs{v}) -\dfrac{1}{2}r_{{\partial\Omega}_{N}}\boldsymbol{\psi}=\boldsymbol{\psi}_{0}
$ on ${\partial\Omega}_{N}$. Since $\boldsymbol{\psi}$ vanishes on ${\partial\Omega}_{N}$, the Neumann condition \eqref{ch2BVPN} is satisfied. 

Because $\boldsymbol{\varphi}_0=\boldsymbol{\Phi}_0$, on ${\partial\Omega}_D$; and $\boldsymbol{\psi}_0=\boldsymbol{\Psi}_0$, on ${\partial\Omega}_N$, we also obtain, 
\begin{align}
\bs{\Psi}^{*}&:=\boldsymbol{\psi} + \boldsymbol{\Psi}_{0} - \boldsymbol{T}^{+}(p,\bs{v})\in \widetilde{\boldsymbol{H}}^{-1/2}({\partial\Omega}_{D}),\label{Psi*inc}&
\bs{\Phi}^{*}&=\boldsymbol{\varphi} + \boldsymbol{\Phi}_{0} - \gamma^{+}\boldsymbol{v}\in \widetilde{\boldsymbol{H}}^{1/2}({\partial\Omega}_{N}).
\end{align}
By relations \eqref{ch2M11p} and \eqref{ch2M11v} the hypotheses of Lemma \ref{ch2L1} are satisfied with $\boldsymbol{\Psi}=\boldsymbol{\psi} + \boldsymbol{\Psi}_{0}$ and $\boldsymbol{\Phi} = \boldsymbol{\varphi} + \boldsymbol{\Phi}_{0}$ . 
As a result, we obtain that the couple $(p,\bs{v})$ satisfies \eqref{ch2BVP1} and \eqref{ch2BVPD}  and, moreover,
\begin{align}\label{ch2VW*}\quad
{\Pi^s}(\boldsymbol{\Psi}^{*}) -{\Pi^d}(\boldsymbol{\Phi}^{*}) = {0},\quad
\boldsymbol{V}(\boldsymbol{\Psi}^{*}) -\boldsymbol{W}(\boldsymbol{\Phi}^{*}) = \boldsymbol{0}\quad \text{in } \Omega
\end{align}
Due to inclusions \eqref{Psi*inc} and relations \eqref{ch2VW*},  Lemma \ref{ch2lemma2} for $S_{1}={\partial\Omega}_{D}$, and $S_{2}={\partial\Omega}_{N}$ implies
$ \boldsymbol{\Psi}^{*}=\boldsymbol{\Phi}^{*}=\boldsymbol{0}$ on ${\partial\Omega}$ and thus
conditions \eqref{ch2eqcond}.

Hence, by item (i) the set $(p, \boldsymbol{v}, \boldsymbol{\psi}, \boldsymbol\varphi)$ solves also BDIE system \eqref{ch2M22}.

(ii-22) Let now
$(p, \boldsymbol{v}, \boldsymbol{\psi}, \boldsymbol\varphi)\in 
L_{2}(\Omega)\times \bs{H}^{1}(\Omega) \times\widetilde{\boldsymbol{H}}^{-1/2}({\partial\Omega}_{D})\times\widetilde{\boldsymbol{H}}^{1/2}({\partial\Omega}_{N})$  
solve  BDIE system \eqref{ch2M22}.
Then equations \eqref{ch2M22p}, \eqref{ch2M22v} and Theorems \ref{ch2thmUR:theo}, \ref{ch2T3} imply that 
$(p, \boldsymbol{v}, \boldsymbol{\psi}, \boldsymbol\varphi)\in 
\bs{H}^{1,0}(\Omega;\boldsymbol{\mathcal{A}}) \times\widetilde{\boldsymbol{H}}^{-1/2}({\partial\Omega}_{D})
\times\widetilde{\boldsymbol{H}}^{1/2}({\partial\Omega}_{N})$ and the canonical conormal derivative $\boldsymbol{T}^+(p,\bs{v})$ is well defined. 
Applying Lemma \ref{ch2L1} with $\boldsymbol{\boldsymbol{\Psi}}=\boldsymbol{\psi} + \boldsymbol{\boldsymbol{\Psi}}_{0}$ and $\boldsymbol{\boldsymbol{\Phi}} = \boldsymbol{\varphi} + \boldsymbol{\boldsymbol{\Phi}}_{0}$ to BDIEs \eqref{ch2M22p}-\eqref{ch2M22v}, we deduce that the couple $(p, \bs{v})$ solves PDE system \eqref{ch2BVP1}-\eqref{ch2BVPdiv} and
\begin{align}\label{ch2M22VW*}
\Pi^s(\boldsymbol{\boldsymbol{\Psi}}^{*}) -\Pi^d(\boldsymbol{\boldsymbol{\Phi}}^{*}) = {0},\quad
\quad \boldsymbol{V}(\boldsymbol{\boldsymbol{\Psi}}^{*}) -\boldsymbol{W}(\boldsymbol{\boldsymbol{\Phi}}^{*}) = \boldsymbol{0},\quad \text{in } \Omega,
\end{align}
where
\begin{align}\label{ch2M22L2cond}
\boldsymbol{\Psi}^{*}&:=\boldsymbol{\psi} + \boldsymbol{\Psi}_{0} - \boldsymbol{T}^{+}(p,\bs{v}),&
\boldsymbol{\Phi}^{*}&:=\boldsymbol{\varphi} + \boldsymbol{\Phi}_{0} - \gamma^{+}\boldsymbol{v},\quad\quad on \quad {\partial\Omega}.
\end{align}

Taking the traction of \eqref{ch2M22p} and \eqref{ch2M22v} restricted to ${\partial\Omega}_{D}$ and subtracting it from \eqref{ch2M22D} we get
\begin{equation}\label{ch2M22eqrsd}
r_{{\partial\Omega}_{D}}\boldsymbol{T}^+(p,\bs{v})- r_{{\partial\Omega}_{D}}\boldsymbol{\boldsymbol{\Psi}}_{0}=\boldsymbol{\psi}
,\quad on\quad {\partial\Omega}_{D}.
\end{equation}
Taking the trace of \eqref{ch2M22v} restricted to ${\partial\Omega}_{N}$ and subtracting it from \eqref{ch2M22N} we get
\begin{equation}\label{ch2M22eqrsn}
r_{{\partial\Omega}_{N}}\gamma^{+}\boldsymbol{v} -r_{{\partial\Omega}_{N}}\boldsymbol{\boldsymbol{\Phi}}_{0}=\boldsymbol{\varphi},\quad on\quad {\partial\Omega}_{N}.
\end{equation}

Due to \eqref{ch2M22eqrsd} and \eqref{ch2M22eqrsn}, we have $\boldsymbol{\Psi}^{*}\in \widetilde{\boldsymbol{H}}^{-1/2}({\partial\Omega}_{D})$ and $\boldsymbol{\Phi}^{*}\in \widetilde{\boldsymbol{H}}^{1/2}({\partial\Omega}_{N})$. Now, we can apply Lemma \ref{ch2lemma2} with $S_{1}={\partial\Omega}_{D}$ and $S_{2}={\partial\Omega}_{N}$, to obtain $\boldsymbol{\Psi}^{*}=\boldsymbol{\Phi}^{*}=\boldsymbol{0}$ on ${\partial\Omega}$, which by \eqref{ch2M22L2cond} imply relations \eqref{ch2eqcond}.  
Since $r_{\partial\Omega_{D}}\bs{\Phi}_{0}=\bs{\phi}_{0}$ and $r_{\partial\Omega_{N}}\bs{\Psi}_{0}=\bs{\psi}_{0}$, relations  \eqref{ch2eqcond} imply the BVP boundary conditions \eqref{ch2BVPD} and \eqref{ch2BVPD}.

(iii) Finally, the unique solvability of the BDIE systems \eqref{ch2M11} and \eqref{ch2M22} in item (iii) follows from the unique solvability of the BVP, see Theorem \ref{ch2BVPUS}, and items (i) and (ii). 
\end{proof}

\subsection{Boundary Integral equations}

When $\mu\equiv 1$, the operator $\boldsymbol{\mathcal{A}}$ becomes $\mathring{\boldsymbol{\mathcal{A}}}$ and $\boldsymbol{\mathcal{R}}=\boldsymbol{\mathcal{R}}^{\bullet}\equiv 0$. 
Consequently, the boundary-domain integral equations system \eqref{ch2M11} can be split into a system of two vector boundary integral equations, 
\begin{align}
r_{{\partial\Omega}_{D}}(-\mathring{\boldsymbol{\mathcal V}}\psi
+\mathring{\boldsymbol{\mathcal W}}\varphi)&=
r_{{\partial\Omega}_{D}}\boldsymbol{\gamma}^{+}\boldsymbol{F}-\boldsymbol{\varphi}_{0},\quad\text{on } {\partial\Omega}_{D}, \label{ch1BIE1}\\
r_{{\partial\Omega}_{N}}(-\mathring{\boldsymbol{\mathcal W}}{}'\psi
+\mathring{\boldsymbol{\mathcal L}}\boldsymbol{\varphi})&=
r_{{\partial\Omega}_{N}}\boldsymbol{T}^{+}( F_{0}, \boldsymbol{F})-\boldsymbol{\psi}_{0}, \quad\text{on } {\partial\Omega}_{N},\label{ch1BIE2}
\end{align}
and two integral representations, for $p$ and $\bs{v}$,
\begin{align}
p&=F_{0}+\mathring{\Pi}^s\bs{\psi}-\mathring{\Pi}^d\boldsymbol{\varphi}\mbox{ in }\Omega,\label{ch23rdp}\\
\bs{v}&=\boldsymbol{F}+\boldsymbol{\mathring{V}}\bs{\psi}-\boldsymbol{\mathring{W}}\boldsymbol{\varphi}
\mbox{ in }\Omega,\label{ch23rdv}
\end{align}
where $F_{0}$ and $\bs{F}$ are given by \eqref{ch2F0F}. 

Similarly, the boundary-domain integral equations system \eqref{ch2M22} can be split into a system of two vector boundary integral equations, 
for $\bs{\psi}$ and $\boldsymbol{\varphi}$,
\begin{align}
r_{{\partial\Omega}_{D}}\left(\dfrac{1}{2}\bs{\psi}-\bs{\mathring{\mathcal{W'}}}\bs{\psi} + \bs{\mathring{\mathcal{L}}}\boldsymbol{\varphi}\right)=r_{{\partial\Omega}_{D}}\boldsymbol{T}^{+}(F_{0},\boldsymbol{F})-r_{{\partial\Omega}_{D}}\boldsymbol{\Psi}_{0}\mbox{ on }{\partial\Omega}_{D},\label{ch2BIE1}\\
r_{{\partial\Omega}_{N}}\left(\dfrac{1}{2}\boldsymbol{\varphi}-\mathring{\mathcal{V}}\bs{\psi} + \bs{\mathring{\mathcal{W}}}\boldsymbol{\varphi}\right)=r_{{\partial\Omega}_{N}}\bs{\gamma}^{+}\boldsymbol{F}-r_{{\partial\Omega}_{N}}\boldsymbol{\Phi}_{0}\mbox{ on }{\partial\Omega}_{N}\label{ch2BIE2}
\end{align}
and two integral representations, \eqref{ch23rdp} and \eqref{ch23rdv}, for $p$ and $\bs{v}$.

Equivalence Theorem \ref{ch2thEQm11} for BDIE system \eqref{ch2thEQm11} leads to the following equivalence assertion for the constant coefficient case.
\begin{cor}\label{ch2corBIE} Let $\mu\equiv 1$, $\bs{f}\in \bs{L}^{2}(\Omega)$ and $g\in L^{2}(\Omega)$. Moreover, let $\boldsymbol{\Phi}_{0}\in \boldsymbol{H}^{1/2}({\partial\Omega})$ and $\boldsymbol{\Psi}_{0}\in \boldsymbol{H}^{-1/2}({\partial\Omega})$ be some extensions of $\boldsymbol{\varphi}_{0}\in \boldsymbol{H}^{1/2}({\partial\Omega}_{D})$ and $\bs{\psi}_{0}\in \boldsymbol{H}^{-1/2}({\partial\Omega}_{N})$, respectively.
\begin{enumerate}
\item[(i)] If a couple $(p, \bs{v})\in L_{2}(\Omega)\times \bs{H}^{1}(\Omega)$ solves BVP \eqref{ch2BVPM}, then the solution is unique, the couple $(\bs{\psi}, \boldsymbol{\varphi})\in \widetilde{\boldsymbol{H}}^{-1/2}({\partial\Omega}_{D})\times \widetilde{\boldsymbol{H}}^{1/2}({\partial\Omega}_{N})$  given by
\begin{equation}\label{ch2eqcond2}
\boldsymbol{\varphi}=\gamma^{+}\boldsymbol{v}-\boldsymbol{\Phi}_{0}, \quad
\boldsymbol{\psi}=\boldsymbol{T}^{+}(p,\bs{v})-\boldsymbol\Psi{}_{0}\quad\text{on } {\partial\Omega},
\end{equation}
solves BIE systems \eqref{ch1BIE1}-\eqref{ch1BIE2} and \eqref{ch2BIE1}-\eqref{ch2BIE2}, and the couple $(p,\bs{v})$ satisfies \eqref{ch23rdp}, \eqref{ch23rdv}.
 
\item[(ii)] If a couple $(\bs{\psi}, \boldsymbol{\varphi})\in \widetilde{\boldsymbol{H}}^{-1/2}({\partial\Omega}_{D})\times \widetilde{\boldsymbol{H}}^{1/2}({\partial\Omega}_{N})$ solves one of BIE system, \eqref{ch1BIE1}-\eqref{ch1BIE2} or \eqref{ch2BIE1}-\eqref{ch2BIE2}, then it solves the other BDIE system, the couple $(p,\bs{v})$ given by \eqref{ch23rdp}-\eqref{ch23rdv} belongs to $\bs{H}^{1,0}(\Omega;\boldsymbol{\mathcal{A}})$ and solves BVP \eqref{ch2BVPM}, while $\boldsymbol{\psi}, \boldsymbol{\varphi}$ satisfy relations \eqref{ch2eqcond2}. 

(iii) Both systems \eqref{ch1BIE1}-\eqref{ch1BIE2} and \eqref{ch2BIE1}-\eqref{ch2BIE2} are uniquely solvable for  $(\bs{\psi}, \boldsymbol{\varphi})\in\widetilde{\boldsymbol{H}}^{-1/2}({\partial\Omega}_{D})\times \widetilde{\boldsymbol{H}}^{1/2}({\partial\Omega}_{N})$.
\end{enumerate}
\end{cor}

\section{BDIE Operators Invertibility}
\subsection{Operators $\mathcal{M}_{*}^{11}$ and $\mathcal{M}^{11}$}
BDIE system \eqref{ch2M11} can be written using matrix notation as
\begin{equation}\label{ch2MXF}
\mathcal{M}_{*}^{11}\mathcal{X}=\mathcal{F}_{*}^{11},
\end{equation}
where 
\begin{align}\label{M11*def}
   \mathcal{M}_{*}^{11}=
  \left[ {\begin{array}{cccc}
  I & \mathcal{R}^{\bullet} &- {\Pi^s} & {\Pi^d}\\
  \bs{0}&  \bs{I}+\boldsymbol{\mathcal{R}} & -\boldsymbol{V} & \boldsymbol{W} \\
   \bs{0}& r_{{\partial\Omega}_{D}}\boldsymbol{\gamma^{+}\mathcal{R}} & -r_{{\partial\Omega}_{D}}\boldsymbol{\mathcal{V}} & r_{{\partial\Omega}_{D}}\boldsymbol{\mathcal{W}}\\
  \bs{0} & r_{{\partial\Omega}_{N}}\boldsymbol{T}^{+}( \mathcal{R}^{\bullet},  \boldsymbol{\mathcal{R}})& -r_{{\partial\Omega}_{N}}\boldsymbol{\mathcal{W'}} & r_{{\partial\Omega}_{N}}\boldsymbol{\mathcal{L}}^+
  \end{array} } \right]
\end{align}
and 
$
\mathcal{X}=(p,\boldsymbol{v},\boldsymbol{\psi},\boldsymbol{\varphi}).
$
By Theorems~\ref{ch2thmUR:theo}-\ref{ch2T3cal} the mapping properties of the operators involved in the matrix imply continuity of the operator 
\begin{multline}\label{M*nar}
\mathcal{M}_{*}^{11}: \bs{H}^{1,0}(\Omega,\boldsymbol{\mathcal{A}})\times\widetilde{\boldsymbol{H}}^{-1/2}({\partial\Omega}_{D})\times\widetilde{\boldsymbol{H}}^{1/2}({\partial\Omega}_{N}) \\
\to \bs{H}^{1,0}(\Omega,\boldsymbol{\mathcal{A}})\times \boldsymbol{H}^{1/2}({\partial\Omega}_{D})\times \boldsymbol{H}^{-1/2}({\partial\Omega}_{N}).
\end{multline}

We can also consider the operator $\mathcal{M}_{*}^{11}$, defined by \eqref{M11*def}, in wider spaces,
\begin{multline}
\mathcal{M}_{*}^{11}: L_{2}(\Omega)\times \bs{H}^{1}(\Omega)\times\widetilde{\boldsymbol{H}}^{-1/2}({\partial\Omega}_{D})\times\widetilde{\boldsymbol{H}}^{1/2}({\partial\Omega}_{N})\\
\to L_{2}(\Omega)\times \bs{H}^{1}(\Omega)\times \boldsymbol{H}^{1/2}({\partial\Omega}_{D})\times \boldsymbol{H}^{-1/2}({\partial\Omega}_{N})
\label{ch2M11-inv}
\end{multline}
Theorems~\ref{ch2thmUR:theo}-\ref{ch2T3cal} imply that operator \eqref{ch2M11-inv} is also continuous.

Let us also write BIE system \eqref{ch1BIE1}-\eqref{ch1BIE2}, for $\mu\equiv 1$,  in the matrix form as
\begin{equation}\label{ch2M11CC}
\mathring{\mathcal{M}}^{11}\mathring{\mathcal{X}}=\mathring{\mathcal{F}}^{11},
\end{equation}
where $\mathring{\mathcal{X}}=(\boldsymbol{\psi} , \boldsymbol{\varphi})$,
$
  \mathring{\mathcal{F}}^{11}=
  \left[ 
r_{{\partial\Omega}_{N}}\boldsymbol{T}^{+}( F_{0}, 
\bs{F})-\boldsymbol{\psi}_{0}
\right]
$
and
\begin{equation}\label{ch2M11CC1}
  \mathring{\mathcal{M}}^{11}=
  \left[ {\begin{array}{cc}
   -r_{{\partial\Omega}_{D}}\mathring{\boldsymbol{\mathcal V}} & r_{{\partial\Omega}_{D}}\mathring{\boldsymbol{\mathcal W}}\\ 
   -r_{{\partial\Omega}_{N}}\mathring{\boldsymbol{\mathcal W}}{}' & r_{{\partial\Omega}_{N}}\mathring{\boldsymbol{\mathcal L}}  \\
  \end{array} } \right].
\end{equation} 
The following assertion is implied by \cite[Theorem 3.10]{kohr1}.
\begin{theorem}\label{ch2thinvM11ring}
The operator 
\begin{align}\label{M11ringop}
\mathring{\mathcal M}^{11}:\widetilde{\boldsymbol{H}}^{-1/2}({\partial\Omega}_{D})
\times\widetilde{\boldsymbol{H}}^{1/2}({\partial\Omega}_{N})
\to  \boldsymbol{H}^{1/2}({\partial\Omega}_{D})\times \boldsymbol{H}^{-1/2}({\partial\Omega}_{N})
\end{align} 
is continuous and continuously invertible.
\end{theorem}
Theorem~\ref{ch2thinvM11ring} will be instrumental in proving the following result.
\begin{theorem}\label{ch2thinvM11}
Operators \eqref{M*nar} and \eqref{ch2M11-inv}
are continuously invertible.
\end{theorem}
\begin{proof}
%
(i) Let us start from operator \eqref{ch2M11-inv}. 
To this end let us define the operator
\[
   \widetilde{\mathcal{M}}^{11}=
  \left[ {\begin{array}{cccc}
  I & \mathcal{R}^{\bullet} &-{\Pi^s} & {\Pi^d}\\
  \bs{0} & \bs{I} & -\bs{V} & \bs{W} \\
   \bs{0} & \bs{0} & -r_{{\partial\Omega}_{D}}\bs{\mathcal{V}} & r_{{\partial\Omega}_{D}}\bs{\mathcal{W}}\\
  \bs{0} & \bs{0} & -r_{{\partial\Omega}_{N}}\bs{\mathring{\mathcal{W'}}} & r_{{\partial\Omega}_{N}}\bs{\widehat{\mathcal{L}}}
  \end{array} } \right], 
\]
and consider the new system 
\begin{equation}\label{ch2M11-tilde}
\widetilde{\mathcal{M}}^{11}\widetilde{\mathcal{X}} = \widetilde{\mathcal{F}}^{11}
\end{equation} where 
$\widetilde{\mathcal{X}} =[\tilde{p},\tilde{\bs{v}},\tilde{\bs{\phi}},\tilde{\bs{\psi}}]^{\top}\in 
L_{2}(\Omega)\times\bs{H}^{1}(\Omega)\times\widetilde{\boldsymbol{H}}^{-1/2}({\partial\Omega}_{D})
\times\widetilde{\boldsymbol{H}}^{1/2}({\partial\Omega}_{N})$ and\\ $\widetilde{\mathcal{F}}=[\widetilde{\mathcal{F}}^{11}_{1}, \widetilde{\mathcal{F}}^{11}_{2},\widetilde{\mathcal{F}}^{11}_{3},\widetilde{\mathcal{F}}^{11}_{4}]^{\top}\in L_{2}(\Omega)\times \bs{H}^{1}(\Omega)\times \boldsymbol{H}^{1/2}({\partial\Omega}_{D})
\times \boldsymbol{H}^{-1/2}({\partial\Omega}_{N})$. 

Consider now, the last two equations of the system \eqref{ch2M11-tilde},  
\begin{align}
-r_{{\partial\Omega}_{D}}\bs{\mathcal{V}}\tilde{\bs{\psi}} +r_{{\partial\Omega}_{D}}\bs{\mathcal{W}}\tilde{\bs{\phi}} &= \widetilde{\mathcal{F}}^{11}_{3},\label{ch2M11-1}\\
 -r_{{\partial\Omega}_{N}}\bs{\mathring{\mathcal{W'}}}\tilde{\bs{\psi}}  +r_{{\partial\Omega}_{N}}\bs{\widehat{\mathcal{L}}}\tilde{\bs{\phi}} &=\widetilde{\mathcal{F}}^{11}_{4}.\label{ch2M11-2}
\end{align}
Multiplying equation \eqref{ch2M11-1} by $\mu$ and applying relations \eqref{ch2relationVW} and \eqref{ch2relationL} we obtain \begin{align}
-r_{{\partial\Omega}_{D}}\bs{\mathring{\mathcal{V}}}\tilde{\bs{\psi}} +r_{{\partial\Omega}_{D}}\bs{\mathring{\mathcal{W}}}(\mu\tilde{\bs{\phi}}) &= \mu\widetilde{\mathcal{F}}^{11}_{3},\label{ch2M11-3}\\
 -r_{{\partial\Omega}_{N}}\bs{\mathring{\mathcal{W'}}}\tilde{\bs{\psi}}  +r_{{\partial\Omega}_{N}}\bs{\mathring{\mathcal{L}}}(\mu\tilde{\bs{\phi}}) &=\widetilde{\mathcal{F}}^{11}_{4}.\label{ch2M11-4}
\end{align}
This system is uniquely solvable for $\widetilde{\bs{\phi}}$ and $\widetilde{\bs{\psi}}$, since the matrix operator of the left hand side is invertible, cf. Theorem~\ref{ch2thinvM11ring}. 
Hence $\widetilde{\bs{v}}$ is uniquely determined from the second equation of the system \eqref{ch2M11-tilde} and thus also is $p$ from the first equation. This proves the invertibility of the operator $\widetilde{\mathcal{M}}^{11}$, which implies that $\widetilde{\mathcal{M}}^{11}$ is a Fredholm operator with zero index.

Furthermore, the operator 
\begin{multline*}
\mathcal{M}_{*}^{11}-\widetilde{\mathcal{M}}^{11}: L_{2}(\Omega)\times \bs{H}^{1}(\Omega)\times\widetilde{\boldsymbol{H}}^{-1/2}({\partial\Omega}_{D})\times\widetilde{\boldsymbol{H}}^{1/2}({\partial\Omega}_{N})\\
\to L_{2}(\Omega)\times \bs{H}^{1}(\Omega)\times \boldsymbol{H}^{1/2}({\partial\Omega}_{D})\times \boldsymbol{H}^{-1/2}({\partial\Omega}_{N})
\end{multline*}
is compact due to Theorems \ref{ch2thRcomp}, \ref{ch2T3cal} and \ref{ch2Lcompact}. Thus  operator \eqref{ch2M11-inv} is also a Fredholm operator with zero index. 
By virtue of the Equivalence Theorem \ref{ch2thEQm11} and Remark \ref{ch2remM11}, the homogeneous system (M11) has only the trivial solution, hence operator \eqref{ch2M11-inv} is invertible.  

(ii) Let us now consider operator \eqref{M*nar}.
Let  $\mathcal{X} = (\mathcal{M}^{11}_{*})^{-1}\mathcal{F}_{*}^{11}$ be the solution of system \eqref{ch2MXF} with an arbitrary right hand side $\mathcal{F}^{11}_{*}\in L^{2}(\Omega)\times \bs{H}^{1}(\Omega)\times \boldsymbol{H}^{1/2}({\partial\Omega}_{D})\times \boldsymbol{H}^{-1/2}({\partial\Omega}_{N})$, where 
\begin{multline*}
(\mathcal{M}^{11}_{*})^{-1}: L_{2}(\Omega)\times \bs{H}^{1}(\Omega)\times \boldsymbol{H}^{1/2}({\partial\Omega}_{D})\times \boldsymbol{H}^{-1/2}({\partial\Omega}_{N})\\
\to
L_{2}(\Omega)\times \bs{H}^{1}(\Omega)\times\widetilde{\boldsymbol{H}}^{-1/2}({\partial\Omega}_{D})\times\widetilde{\boldsymbol{H}}^{1/2}({\partial\Omega}_{N})
\end{multline*}
is the inverse of operator \eqref{ch2M11-inv}. 

If, moreover, $
\mathcal{F}^{11}_{*}\in \bs{H}^{1,0}(\Omega,\boldsymbol{\mathcal{A}})\times\boldsymbol{H}^{1/2}({\partial\Omega}_{D})\times 
 \boldsymbol{H}^{-1/2}({\partial\Omega}_{N}),
 $
then  the first two equations of system \eqref{ch2MXF} and the mapping properties of the operators in these equations imply that $\mathcal{X} \in \bs{H}^{1,0}(\Omega;\bs{\mathcal{A}})\times\widetilde{\boldsymbol{H}}^{-1/2}({\partial\Omega}_{D})\times\widetilde{\boldsymbol{H}}^{1/2}({\partial\Omega}_{N})$. 
Consequently, the operator
\begin{multline*}
(\mathcal{M}^{11}_{*})^{-1}: L_{2}(\Omega)\times \bs{H}^{1}(\Omega)\times \boldsymbol{H}^{1/2}({\partial\Omega}_{D})\times \boldsymbol{H}^{-1/2}({\partial\Omega}_{N})\\
\to
L_{2}(\Omega)\times \bs{H}^{1}(\Omega)\times\widetilde{\boldsymbol{H}}^{-1/2}({\partial\Omega}_{D})\times\widetilde{\boldsymbol{H}}^{1/2}({\partial\Omega}_{N})
\end{multline*}
is also continuous and is an inverse of operator \eqref{M*nar}. 
\end{proof}
}

Note that the BDIE system (M11$_*$) given by \eqref{ch2M11} can be split into the BDIE system (M11), of 3 vector equations \eqref{ch2M11v}, \eqref{ch2M11G}, \eqref{ch2M11T} for 3 vector unknowns, $\boldsymbol{v}$, $\boldsymbol{\psi}$ and $\boldsymbol{\varphi}$, and the scalar equation \eqref{ch2M11p} that can be used, after solving the system, to obtain the pressure, $p$. 
Using matrix notation,
\begin{equation}\label{ch2M11y}
\mathcal{M}^{11}(\boldsymbol{v},\boldsymbol{\psi},\boldsymbol{\varphi})^\top=\mathcal{F}^{11},
\end{equation}
where 
\begin{align}\label{M11}
   \mathcal{M}^{11}=
  \left[ {\begin{array}{ccc}
   \bs{I}+\boldsymbol{\mathcal{R}} & -\boldsymbol{V} & \boldsymbol{W} \\
   r_{{\partial\Omega}_{D}}\boldsymbol{\gamma^{+}\mathcal{R}} & -r_{{\partial\Omega}_{D}}\boldsymbol{\mathcal{V}} & r_{{\partial\Omega}_{D}}\boldsymbol{\mathcal{W}}\\
   r_{{\partial\Omega}_{N}}\boldsymbol{T}^{+}( \mathcal{R}^{\bullet},  \boldsymbol{\mathcal{R}})& -r_{{\partial\Omega}_{N}}\boldsymbol{\mathcal{W'}} & r_{{\partial\Omega}_{N}}\boldsymbol{\mathcal{L}}^+
  \end{array} } \right].
\end{align}

\begin{theorem}\label{ch2corinvM11}
The operator 
\begin{align}\label{M11op}
\mathcal{M}^{11}: \boldsymbol{H}^{1}(\Omega)\times\widetilde{\boldsymbol{H}}^{-1/2}({\partial\Omega}_{D})
\times\widetilde{\boldsymbol{H}}^{1/2}({\partial\Omega}_{N})
\to \boldsymbol{H}^{1}(\Omega)\times \boldsymbol{H}^{1/2}({\partial\Omega}_{D})\times \boldsymbol{H}^{-1/2}({\partial\Omega}_{N})
\end{align} 
is continuous and continuously invertible.
\end{theorem}
\begin{proof}
Operator \eqref{M11op} is continuous due to the mapping properties of the integral operators involved in \eqref{M11}. 

By the same arguments as in part (i) of the proof of Theorem~\ref{ch2thinvM11}, we obtain that operator \eqref{M11op} is Fredholm with zero index.
%
Complementing system \eqref{ch2M11y} with an arbitrary right hand side 
$\mathcal{F}^{11}\in \boldsymbol{H}^{1}(\Omega)\times \boldsymbol{H}^{1/2}({\partial\Omega}_{D})\times \boldsymbol{H}^{-1/2}({\partial\Omega}_{N})$ 
by equation \eqref{ch2M11p} with a right hand side $F_0\in L_{2}(\Omega)$, we arrive at system \eqref{ch2MXF} which is solvable in 
$L_{2}(\Omega)\times\bs{H}^{1}(\Omega)\times\widetilde{\boldsymbol{H}}^{-1/2}({\partial\Omega}_{D})
\times\widetilde{\boldsymbol{H}}^{1/2}({\partial\Omega}_{N})$,
due to Theorem~\ref{ch2thinvM11} for operator \eqref{ch2M11-inv}, and thus delivers a solution $(\boldsymbol{v},\boldsymbol{\psi},\boldsymbol{\varphi})\in
\boldsymbol{H}^{1}(\Omega)\times\widetilde{\boldsymbol{H}}^{-1/2}({\partial\Omega}_{D})
\times\widetilde{\boldsymbol{H}}^{1/2}({\partial\Omega}_{N})$ 
of system \eqref{ch2M11y}, which implies surjectivity of operator \eqref{M11op}.
To prove that the operator is also injective, we assume the opposite, which would imply that operator \eqref{ch2M11-inv} is also non-injective thus contradicting its invertibility.
\end{proof}

\begin{corollary}\label{ch2invBVPM} Let $\bs{f}\in \bs{L}^{2}(\Omega)$, $g\in L^{2}(\Omega)$, $\bs{\phi}_{0}\in \bs{H}^{1/2}({\partial\Omega}_{D})$ and $\bs{\psi}_{0}\in \bs{H}^{-1/2}({\partial\Omega}_{N})$ respectively. Then, the BVP \eqref{ch2BVPM} is uniquely solvable in $\bs{H}^{1,0}(\Omega;\bs{\mathcal{A}})$ and the operator 
\begin{align}\label{AM}
\bs{\mathcal{A}}_{M}:\bs{H}^{1,0}(\Omega;\bs{\mathcal{A}})\to \bs{L}^{2}(\Omega)\times L^{2}(\Omega)\times \bs{H}^{1/2}({\partial\Omega}_{D})\times \bs{H}^{-1/2}({\partial\Omega}_{N}) 
\end{align}
is continuously invertible.
\end{corollary} 

\begin{proof}
 By Theorem \ref{ch2thinvM11} for operator \eqref{M*nar}, BDIE system \eqref{ch2M11} is uniquely solvable and by Theorem \ref{ch2thEQm11} it is equivalent to the BVP \eqref{ch2BVPM}, which implies unique solvability of the latter. 
In addition, the inverse to operator \eqref{AM} is defined as 
\[
\mathcal{A}^{-1}_{M}(\bs{f},g,r_{{\partial\Omega}_{D}}\bs{\Phi}_{0}, r_{{\partial\Omega}_{N}}\bs{\Psi}_{0}) 
= [ ((\mathcal{M}^{11}_{*})^{-1}\mathcal{F}^{11}_{*})_{1} ,  ((\mathcal{M}^{11}_{*})^{-1}\mathcal{F}^{11}_{*})_{2}]
\]
and is continuous since operator \eqref{M*nar} is continuously invertible and $\mathcal{F}^{11}_{*}$ is a continuous function of $(\bs{f},g, \bs{\Psi}_{0},\bs{\Phi}_{0})$ due to the mapping properties of the operators involved in \eqref{ch2F0F} and \eqref{ch2F11def}.
\end{proof}

\subsection{Operators $\mathcal{M}_{*}^{22}$ and $\mathcal{M}^{22}$}
{ BDIE system \eqref{ch2M22} can be written in the matrix form as }
\begin{equation}\label{ch2sM22}
\mathcal{M}^{22}_{*}\mathcal{X}=\mathcal{F}_{*}^{22},
\end{equation}
{ where} 
\begin{equation}\label{ch2M22M}
  \mathcal{M}_{*}^{22}=
  \left[ {\begin{array}{cccc}
  I & \mathcal{R}^{\bullet} & -{\Pi^s} & {\Pi^d} \\
   \bs{0} & \boldsymbol{I}+\boldsymbol{\mathcal{R}}& -\boldsymbol{V} & \boldsymbol{W} \\
   \bs{0} & r_{{\partial\Omega}_{D}}\boldsymbol{T}^{+}( \mathcal{R}^{\bullet},  \boldsymbol{\mathcal{R}}) & r_{{\partial\Omega}_{D}}\left(\dfrac{1}{2}\boldsymbol{I}-\boldsymbol{\mathcal{W'}}\right) & r_{{\partial\Omega}_{D}}\boldsymbol{\mathcal{L}}^{+} 
    \\
   \bs{0} & r_{{\partial\Omega}_{N}}\gamma^{+}\boldsymbol{\mathcal{R}}& -r_{{\partial\Omega}_{N}}\boldsymbol{\mathcal{V}} & r_{{\partial\Omega}_{N}}\left(\dfrac{1}{2}\boldsymbol{I} + \boldsymbol{\mathcal{W}}\right)
  \end{array} } \right],
\end{equation}
{ and 
$
\mathcal{X}=(p,\boldsymbol{v},\boldsymbol{\psi},\boldsymbol{\varphi}).
$
By Theorems~\ref{ch2thmUR:theo}-\ref{ch2T3cal},} the mapping properties of the operators involved in
\eqref{ch2M22M} imply continuity of the operator
\begin{multline}\label{M*nar22} 
\mathcal{M}^{22}_{*}:\bs{H}^{1,0}(\Omega,\boldsymbol{\mathcal{A}})\times \widetilde{\boldsymbol{H}}^{-1/2}({\partial\Omega}_{D}) \times \widetilde{\boldsymbol{H}}^{1/2}({\partial\Omega}_{N})\\
\to \bs{H}^{1,0}(\Omega,\boldsymbol{\mathcal{A}})\times \boldsymbol{H}^{-1/2}({\partial\Omega}_{D})\times \boldsymbol{H}^{1/2}({\partial\Omega}_{N}).
\end{multline}

\begin{lemma}\label{ch2L513} Let ${\partial\Omega}=\bar{S}_{1}\cup\bar{S}_{2}$, where $S_{1}$ and $S_{2}$ are two non-intersecting simply connected nonempty submanifolds of ${\partial\Omega}$ with infinitely smooth boundaries. For any vector
\begin{center}
$\mathcal{F}=( F_{0}, \boldsymbol{F}, \boldsymbol{\boldsymbol{\Psi}}, \boldsymbol{\boldsymbol{\Phi}})\in \bs{H}^{1,0}(\Omega;\mathcal{A})\times \boldsymbol{H}^{-1/2}(S_{1})\times \boldsymbol{H}^{1/2}(S_{2})$
\end{center}
there exists a unique four-tuple
\begin{center}
$(g_{*}, \boldsymbol{f}_{*}, \boldsymbol{\boldsymbol{\Psi}}_{*}, \boldsymbol{\boldsymbol{\Phi}}_{*})
=\widetilde{\mathcal{C}}_{S_{1}, S_{2}}\mathcal{F}\in L_{2}(\Omega)\times \boldsymbol{L}_{2}(\Omega)\times \boldsymbol{H}^{-1/2}({\partial\Omega})\times \boldsymbol{H}^{1/2}({\partial\Omega})$
\end{center}
such that
\begin{subequations}
\label{ch2lemma3}
\begin{align}
\mathring{\mathcal{Q}}\boldsymbol{f}_{*} + \dfrac{4}{3}\mu g_{*}+ {\Pi^s}\boldsymbol{\boldsymbol{\Psi}}_{*}-{\Pi^d}\boldsymbol{\boldsymbol{\Phi}}_{*}&= F_{0}\quad \mbox{in} \ \Omega,\label{ch2lemma3p}\\
\boldsymbol{\mathcal{U}f}_{*}{-}\boldsymbol{\mathcal{Q}}g_{*}+\boldsymbol{\boldsymbol{V\Psi}}_{*}
-\boldsymbol{\boldsymbol{W\Phi}}_{*} &= \boldsymbol{F}\quad \mbox{in} \ \Omega,\label{ch2lemma3v}\\
r_{S_{1}}\boldsymbol{\boldsymbol{\Psi}}_{*}&=\boldsymbol{\boldsymbol{\Psi}}\quad \mbox{on} \ S_{1},\label{ch2lemma3d}\\
r_{S_{2}}\boldsymbol{\boldsymbol{\Phi}}_{*}&=\boldsymbol{\boldsymbol{\Phi}}\quad \mbox{in} \ S_{2}.\label{ch2lemma3n}
\end{align}
\end{subequations}
Furthermore, the operator 
\begin{multline}\label{Ctil-cont}
\widetilde{\mathcal{C}}_{S_{1}, S_{2}}:\bs{H}^{1,0}(\Omega;\mathcal{A})\times \boldsymbol{H}^{-1/2}(S_{1})\times \boldsymbol{H}^{1/2}(S_{2})\\
\to L_{2}(\Omega)\times \boldsymbol{L}_{2}(\Omega)\times \boldsymbol{H}^{-1/2}({\partial\Omega})\times \boldsymbol{H}^{1/2}({\partial\Omega})
\end{multline}
is continuous.
\end{lemma}
\begin{proof} { Let $E^s_{S_i}:\boldsymbol{H}^s(S_i)\to\boldsymbol{H}^{s}(\partial\Omega)$, $i=\{1,2\}$, $|s|\le1$, be some linear continuous extension operators from $S_i$ to the whole boundary ${\partial\Omega}$ (cf. \cite[Subsection 4.2]{triebel}), and
let $\boldsymbol{\Psi}^{0}=E^{-1/2}_{S_1}\boldsymbol{\Psi}$ and 
$\boldsymbol{\Phi}^0=E^{1/2}_{S_2}\boldsymbol{\Phi}$.}
Consequently, arbitrary extensions of the functions $\boldsymbol{\Psi}$ and $\boldsymbol{\Phi}$ can be represented as
\begin{equation}\label{ch2aa1}
\boldsymbol{\Psi}_{*}=\boldsymbol{\Psi}^{0}+\boldsymbol{\widetilde{\psi}},\quad 
\boldsymbol{\Phi}_{*}=\boldsymbol{\Phi}^{0}+\widetilde{\boldsymbol{\varphi}}\ \mbox{ on }\partial\Omega;\quad
\widetilde{\boldsymbol{\psi}}\in \widetilde{\boldsymbol{H}}^{-1/2}(S_{2}),\quad
\widetilde{\boldsymbol{\varphi}}\in \widetilde{\boldsymbol{H}}^{1/2}(S_{1}).
\end{equation}
The functions $\boldsymbol{\Psi}_{*}$ and $\boldsymbol{\Phi}_{*}$, in form \eqref{ch2aa1} satisfy conditions \eqref{ch2lemma3d} and \eqref{ch2lemma3n}. Consequently, it is only left to show that the functions $g_{*},\boldsymbol{f}_{*}, \boldsymbol{\widetilde{\psi}}$ and $ \widetilde{\boldsymbol{\varphi}}$ can be chosen in a particular way such that equations \eqref{ch2lemma3p}-\eqref{ch2lemma3v} are satisfied.

Applying relations \eqref{ch2relationU}-\eqref{ch2relationP} to equations \eqref{ch2lemma3p}-\eqref{ch2lemma3v}, we obtain
\begin{eqnarray}
\mathring{\mathcal{Q}}f_{*} + \dfrac{4}{3}\mu g_{*}+  \mathring{{\Pi^s}}\left(\boldsymbol{\Psi}_{0}+\boldsymbol{\widetilde{\psi}}\right)-\mathring{{\Pi^d}}\left(\mu\boldsymbol{\Phi}_{0}+\mu\boldsymbol{\varphi}\right) &=& F_{0}\quad \mbox{in} \ \Omega,\label{ch2a2p}\\
\boldsymbol{\mathring{\mathcal{U}}}f_{*} {-} \boldsymbol{\mathring{\mathcal{Q}}}(\mu g_{*}) + \boldsymbol{\mathring{V}}\left(\boldsymbol{\Psi}_{0}+\boldsymbol{\widetilde{\psi}}\right) - \boldsymbol{\mathring{W}}\left(\mu\boldsymbol{\Phi}_{0}+\mu\boldsymbol{\varphi}\right)&=&\mu \boldsymbol{F}\quad \mbox{in} \ \Omega.\label{ch2a2}
\end{eqnarray}

Applying the Stokes operator with constant viscosity $\mu=1$, $\mathring{\boldsymbol{\mathcal{A}}}$, to equations \eqref{ch2a2p}, \eqref{ch2a2}, and the divergence operator to equation \eqref{ch2a2}, we obtain
\begin{align}\label{ch2a3}
\boldsymbol{f}_{*}&=\mathring{\boldsymbol{\mathcal{A}}}(F_{0}, \mu \boldsymbol{F}),\quad
g_{*} = \frac{1}{\mu}\div(\mu \boldsymbol{F})
\end{align}
which shows that the function $\boldsymbol{f}_{*}$ and $g_{*}$ are uniquely determined by $F_{0}$ and $\mu \boldsymbol{F}$ and belong to $\bs{L}^{2}(\Omega)$ and  $L^{2}(\Omega)$, respectively. 

Substituting now \eqref{ch2a3} into equations \eqref{ch2a2p}-\eqref{ch2a2} gives
\begin{align}
 \mathring{{\Pi^s}}\boldsymbol{\widetilde{\psi}} - \mathring{{\Pi^d}}(\mu\widetilde{\boldsymbol{\varphi}})&=J_{0}\mathcal F, 
\quad
\label{ch2a31}
\boldsymbol{\mathring{V}}\boldsymbol{\widetilde{\psi}} - \boldsymbol{\mathring{W}}(\mu\widetilde{\boldsymbol{\varphi}})
= \boldsymbol{J}\mathcal F
\quad
\text{in}\hspace{0.2em} \Omega,
\end{align}
where the continuous operators $J_{0}$ and $\boldsymbol{J}$ are defined as
\begin{align}
J_{0}\mathcal F&:=\left( F_{0}-\dfrac{4}{3}\div(\mu \bs{F}) - \mathring{\mathcal{Q}}\left(\mathring{\boldsymbol{\mathcal{A}}}(F_{0}, \mu \boldsymbol{F})\right) - \mathring{{\Pi^s}}(E^{-1/2}_{S_1}\boldsymbol{\Psi}) 
+ \mathring{{\Pi^d}}(\mu E^{1/2}_{S_2}\boldsymbol{\Phi}) \right),
\label{J0F}\\
\boldsymbol{J}\mathcal F&:= 
\left(\mu \boldsymbol{F}-\boldsymbol{\mathring{\mathcal{U}}}\left(\mathring{\boldsymbol{\mathcal{A}}}(F_{0}, \mu \boldsymbol{F})\right) 
{+} \boldsymbol{\mathring{\mathcal{Q}}}\div(\mu \bs{F})  -\boldsymbol{\mathring{V}}(E^{-1/2}_{S_1}\boldsymbol{\Psi})
+\boldsymbol{\mathring{W}}(\mu E^{1/2}_{S_2}\boldsymbol{\Phi})\right).
\label{JF}
\end{align}

By Theorems \ref{ch2thmUR:theo}, \ref{ch2T3}, 
$({J}_0\mathcal F,\boldsymbol{J}\mathcal F)\in \bs{H}^{1,0}(\Omega;\mathring{\mathcal{A}})$, thus the canonical conormal derivative $\boldsymbol{\mathring{T}}^{+}(J_{0}\mathcal F, \boldsymbol{J}\mathcal F)$ is well defined. 
Then system \eqref{ch2a31} implies
\begin{align}
r_{S_{2}}\boldsymbol{\gamma}^{+}\left( \boldsymbol{\mathring{V}}\boldsymbol{\widetilde{\psi}} - \boldsymbol{\mathring{W}}(\mu\widetilde{\boldsymbol{\varphi}})\right) &= r_{S_{2}}\left(\boldsymbol{\gamma}^{+}\boldsymbol{J}\mathcal F\right),\label{ch2a4}\\
r_{S_{1}}\left[ \boldsymbol{\mathring{T}}^{+}\left( \mathring{{\Pi^s}}(\widetilde{\psi}) - \mathring{{\Pi^d}}(\mu\boldsymbol{\widetilde{\varphi}}),\boldsymbol{\mathring{V}}\widetilde{\psi} - \boldsymbol{\mathring{W}}(\mu\boldsymbol{\widetilde{\varphi}})\right)\right] &= r_{S_{1}}\left(\boldsymbol{\mathring{T}}^{+}(J_{0}\mathcal F, \boldsymbol{J}\mathcal F)\right)\label{ch2a5}.
\end{align}
System \eqref{ch2a4}-\eqref{ch2a5} can be written in the matrix form as
\begin{equation}\label{ch2lemma3sys}
\left[
\begin{array}{cc}
r_{S_{2}}\bs{\mathring{\mathcal{V}}} & r_{S_{2}}\bs{\gamma}^{+}\bs{\mathring{W}} \\  
r_{S_{1}}\bs{\mathring{W}}{}'  & r_{S_{1}}\bs{\mathring{\mathcal{L}}} \\ 
\end{array}\right]\left[ \begin{array}{c}
\boldsymbol{\widetilde{\psi}}\\
\mu\boldsymbol{\widetilde{\varphi}}  
\end{array}\right] = \left[ \begin{array}{c}
r_{S_{2}}\left(\boldsymbol{\gamma}^{+}\boldsymbol{J}\mathcal F\right)\\
r_{S_{1}}\left(\boldsymbol{\mathring{T}}^{+}(J_{0}\mathcal F, \boldsymbol{J}\mathcal F)\right)  
\end{array}\right].
\end{equation}
The matrix operator given by the left-hand side of the equations \eqref{ch2lemma3sys} is an isomorphism between the spaces $\widetilde{\boldsymbol{H}}^{-1/2}(S_{2})\times\widetilde{\boldsymbol{H}}^{1/2}(S_{1})$ and $\boldsymbol{H}^{1/2}(S_{2})\times \boldsymbol{H}^{-1/2}(S_{1})$ (see Theorem~\ref{ch2thinvM11ring}). 
Therefore the solution  of system \eqref{ch2lemma3sys} can be written as $(\widetilde{\boldsymbol{\varphi}},\boldsymbol{\widetilde{\psi}})=\mathring C\mathcal F$, where $\mathring C$ is a continuous operator, which together with \eqref{ch2a3}, \eqref{ch2aa1} and continuity of the extension operator $E^s_{S_i}$ produces a linear continuous operator $\widetilde{\mathcal{C}}_{S_{1}, S_{2}}$ in \eqref{Ctil-cont}. 

Let us prove that $\boldsymbol{\Psi}_{*}$ and $\boldsymbol{\Phi}_{*}$, obtained by substituting in \eqref{ch2aa1}
any solution $(\widetilde{\boldsymbol{\psi}},\widetilde{\boldsymbol{\varphi}})$ of \eqref{ch2lemma3sys}, and $\boldsymbol{f}_{*}$ $g_{*}$, given by \eqref{ch2a3}, satisfy \eqref{ch2lemma3}. 
Equations \eqref{ch2lemma3d} and \eqref{ch2lemma3n} are immediately implied by \eqref{ch2aa1}.
The couple 
$
\left(\mathring{{\Pi^s}}\boldsymbol{\widetilde{\psi}} - \mathring{{\Pi^d}}(\mu\widetilde{\boldsymbol{\varphi}}),
\boldsymbol{\mathring{V}}\boldsymbol{\widetilde{\psi}} - \boldsymbol{\mathring{W}}(\mu\widetilde{\boldsymbol{\varphi}})\right)
$ 
satisfies the incompressible homogeneous PDE Stokes system with $\mu=1$. 
It is easy to check that the same system is also satisfied by the couple $({J}_0\mathcal F,\boldsymbol{J}\mathcal F)$.
By \eqref{ch2a4}-\eqref{ch2a5}, the couples have coinciding mixed boundary conditions and thus they coincide also in the domain $\Omega$ by virtue of uniqueness of solution of the mixed BVP for the Stokes system with $\mu=1$, i.e., equations \eqref{ch2a31} hold and substitution of \eqref{J0F}, \eqref{JF} into their right hand sides leads to \eqref{ch2lemma3p} and \eqref{ch2lemma3v}.

To prove that the operator $\widetilde{\mathcal{C}}_{S_{1}, S_{2}}$ is unique, let us consider system \eqref{ch2lemma3} with zero right-hand side $\mathcal F$. Then \eqref{ch2a3} implies $\boldsymbol{f}_{*}=\bs{0}$, $g_{*} =0$, while \eqref{ch2lemma3d}-\eqref{ch2lemma3n} and \eqref{ch2aa1} give $\boldsymbol{\Psi}_{*}=\boldsymbol{\widetilde{\psi}}$,
$\boldsymbol{\Phi}_{*}=\widetilde{\boldsymbol{\varphi}}$ on $\partial\Omega$, and finally \eqref{ch2lemma3sys} implies 
$\boldsymbol{\widetilde{\psi}}=\bs{0}$, $\widetilde{\boldsymbol{\varphi}}=\bs{0}$.
This means the solution $(g_{*}, \boldsymbol{f}_{*}, \boldsymbol{\boldsymbol{\Psi}}_{*}, \boldsymbol{\boldsymbol{\Phi}}_{*})$ of inhomogeneous system \eqref{ch2lemma3} is unique, along with the operator $\widetilde{\mathcal{C}}_{S_{1}, S_{2}}$.  
\end{proof}

\begin{cor}\label{ch2corinvM22a} 
{ Let ${\partial\Omega}=\bar{S}_{1}\cup\bar{S}_{2}$, where $S_{1}$ and $S_{2}$ are two non-intersecting simply connected nonempty submanifolds of ${\partial\Omega}$ with infinitely smooth boundaries.}
For any four-tuple
\begin{center}
$\mathcal{F}=(F_{0}, \bs{F}, \boldsymbol{\mathcal{F}}_{2}, \boldsymbol{\mathcal{F}}_{3})^{\top}\in \bs{H}^{1,0}(\Omega;\mathcal{A})\times \boldsymbol{H}^{-1/2}(S_{1})\times \boldsymbol{H}^{1/2}(S_{2}),$
\end{center}
there exists a unique four-tuple
\begin{center}
$(g_{*}, \boldsymbol{f}_{*}, \boldsymbol{\Psi}_{*}, \boldsymbol{\Phi}_{*})^{\top}=\mathcal{C}_{S_{1}, S_{2}}\mathcal{F}\in L^{2}(\Omega)\times\boldsymbol{L}_{2}(\Omega)\times \boldsymbol{H}^{-1/2}({\partial\Omega})\times \boldsymbol{H}^{1/2}({\partial\Omega}),$
\end{center}
such that
\begin{align}
\mathring{\mathcal{Q}}\boldsymbol{f}_{*} +\dfrac{4}{3}\mu g_{*}+ {\Pi^s}\boldsymbol{\Psi}_{*} -{\Pi^d}\boldsymbol{\Phi}_{*}&= F_{0},\hspace{0.3em}in\hspace{0.2em}\Omega,\label{ch2cora1}\\
 \boldsymbol{\mathcal{U}}\boldsymbol{f}_{*}{-}\bs{\mathcal{Q}}g_{*}+\boldsymbol{V}\boldsymbol{\Psi}_{*}-\boldsymbol{W}\boldsymbol{\Phi}_{*} &= \boldsymbol{F},\hspace{0.3em}in\hspace{0.2em}\Omega,\label{ch2cora2}\\
r_{S_{1}}(\boldsymbol{T}^{+}({F}_{0},\boldsymbol{\mathcal{F}}_{1})-\boldsymbol{\Psi}_{*})&=\boldsymbol{\mathcal{F}}_{2},\hspace{0.3em}on\hspace{0.2em}S_{1}\label{ch2cora3}\\
r_{S_{2}}(\boldsymbol{\gamma}^{+}\boldsymbol{\mathcal{F}}_{1}- \boldsymbol{\Phi}_{*})&=\boldsymbol{\mathcal{F}}_{3},\hspace{0.3em}on\hspace{0.2em}S_{2}\label{ch2cora4}.
\end{align}
Furthermore, the operator 
\begin{center}
$\mathcal{C}_{S_{1}, S_{2}}:\bs{H}^{1,0}(\Omega;\mathcal{A})\times \boldsymbol{H}^{-1/2}(S_{1})\times \boldsymbol{H}^{1/2}(S_{2})\to L^{2}(\Omega)\times\bs{L}^{2}(\Omega)\times \boldsymbol{H}^{-1/2}({\partial\Omega})\times \boldsymbol{H}^{1/2}({\partial\Omega})$
\end{center}
is continuous.
\end{cor}
\begin{proof}
The Corollary follows from applying Lemma \ref{ch2L513} with $\bs{\Psi}:=r_{S_{1}}\bs{T}^{+}({F}_{0},\bs{\mathcal{F}}_{1})-\bs{\mathcal{F}}_{2}$ and $\bs{\Phi}:=r_{S_{2}}\gamma^{+}\bs{\mathcal{F}}_{1}- \bs{\mathcal{F}}_{3}$.
\end{proof}

\begin{theorem}\label{ch2thM22inv}
Operator \eqref{M*nar22}  is continuously invertible.
\end{theorem}
\begin{proof}
{ Let us consider system \eqref{ch2sM22} with an arbitrary right hand side }
\[
\mathcal{F}_{*}^{22}\in \bs{H}^{1,0}(\Omega ;\mathcal{A})\times \boldsymbol{H}^{-1/2}({\partial\Omega}_{D})\times \boldsymbol{H}^{1/2}({\partial\Omega}_{N}).
\] 
{ By} Corollary \ref{ch2corinvM22a}, the right hand side $\mathcal{F}_{*}^{22}$ can be written in form \eqref{ch2cora1}-\eqref{ch2cora4} with $S_{1}={\partial\Omega}_{D}$ and $S_{2}={\partial\Omega}_{N}$. In addition, $(g_{*}, \boldsymbol{f}_{*}, \boldsymbol{\Psi}_{*}, \boldsymbol{\Phi}_{*})^{\top}=\mathcal{C}_{{\partial\Omega}_{D}, {\partial\Omega}_{N}}\mathcal{F}^{22}$ where the operator 
\begin{multline*}
\mathcal{C}_{{\partial\Omega}_{D}, {\partial\Omega}_{N}}:\bs{H}^{1,0}(\Omega ;\mathcal{A})\times \boldsymbol{H}^{-1/2}({\partial\Omega}_{D})\times \boldsymbol{H}^{1/2}({\partial\Omega}_{N})\\
\to L^{2}(\Omega)\times\boldsymbol{L}_{2}(\Omega)\times \boldsymbol{H}^{-1/2}({\partial\Omega})\times \boldsymbol{H}^{1/2}({\partial\Omega})
\end{multline*}
is continuous.

By Corollary \ref{ch2invBVPM} and Equivalence Theorem \ref{ch2thEQm11}, there exists a solution of the equation $\mathcal{M}^{22}_{*}\mathcal{X}=\mathcal{F}^{22}_{*}$. This solution can be represented as 
\[
\mathcal{X}=(\mathcal{M}^{22}_{*})^{-1}\mathcal{F}^{22}_{*}=[p,\bs{v},\bs{\psi},\bs{\phi}],
\] 
where the operator
 \begin{multline}\label{M22*inv}
(\mathcal{M}^{22}_{*})^{-1}:\bs{H}^{1,0}(\Omega ;\mathcal{A})\times \boldsymbol{H}^{-1/2}({\partial\Omega}_{D})\times \boldsymbol{H}^{1/2}({\partial\Omega}_{N})\\
\to \bs{H}^{1,0}(\Omega;\mathcal{A})\times\widetilde{\boldsymbol{H}}^{-1/2}({\partial\Omega}_{D})\times\widetilde{\boldsymbol{H}}^{1/2}({\partial\Omega}_{N}).
\end{multline}
 is given by 
 \begin{align}
& (p,\bs{v})
=\bs{\mathcal{A}}_{M}^{-1}[g_{*}, \boldsymbol{f}_{*}, r_{{\partial\Omega}_{D}}\boldsymbol{\Psi}_{*}, r_{{\partial\Omega}_{N}}\boldsymbol{\Phi}_{*}]
 =\bs{\mathcal{A}}_{M}^{-1}\mathcal{C}_{{\partial\Omega}_{D}, {\partial\Omega}_{N}}\mathcal{F}^{22}_{*},
 \label{7.34}\\
&\bs{\psi}=\bs{T}^{+}(p,\bs{v})-\bs{\Psi}_{*}
=\bs{T}^{+}(p,\bs{v})-(\mathcal{C}_{{\partial\Omega}_{D}, {\partial\Omega}_{N}}\mathcal{F}^{22}_{*})_3,
 \label{7.35}\\
&\bs{\phi}=\bs{\gamma}^{+}\bs{v}-\bs{\Phi}_{*}
 =\bs{\gamma}^{+}\bs{v}-(\mathcal{C}_{{\partial\Omega}_{D}, {\partial\Omega}_{N}}\mathcal{F}^{22}_{*})_4.
  \label{7.36}
\end{align}
Continuity of the operators in \eqref{7.34}-\ref{7.36} implies continuity of operator \eqref{M22*inv}.

\end{proof}

%
%
%
%

Let us express BIE system \eqref{ch2BIE1}-\eqref{ch2BIE2}, for $\mu\equiv 1$, in the matrix form as
\begin{equation}\label{ch2M22CC}
\mathring{\mathcal{M}}^{22}\mathring{\mathcal{X}}=\mathring{\mathcal{F}}^{22},
\end{equation}
where $\mathring{\mathcal{X}}=(\boldsymbol{\psi} , \boldsymbol{\varphi})$,
$
  \mathring{\mathcal{F}}^{22}=
  \left[ 
   r_{{\partial\Omega}_{D}}\left(\mathring{T}^{+}\boldsymbol{F} -\boldsymbol{\Psi}_{0}\right), 
   r_{{\partial\Omega}_{N}}\left(\gamma^{+}\boldsymbol{F}  -\boldsymbol{\Phi}_{0}\right),
\right]
$
and
\begin{equation}\label{ch2M22CC1}
  \mathring{\mathcal{M}}^{22}=
  \left[ {\begin{array}{cc}
   r_{{\partial\Omega}_{D}}\left(\dfrac{1}{2}\bs{I} - \mathring{\bs{\mathcal{W'}}}\right) & 
   r_{{\partial\Omega}_{D}}\mathring{\bs{\mathcal{L}}} \\ 
   -r_{{\partial\Omega}_{N}}\mathring{\bs{\mathcal{V}}}  & r_{{\partial\Omega}_{N}}\left(\dfrac{1}{2}\bs{I} + \mathring{\bs{\mathcal{W}}}\right)  \\
  \end{array} } \right],
\end{equation} 

The operator 
\begin{align}\label{M022op}
\mathring{\mathcal{M}}^{22}:\widetilde{\boldsymbol{H}}^{-1/2}({\partial\Omega}_{D})\times \widetilde{\boldsymbol{H}}^{1/2}({\partial\Omega}_{N}) \to \boldsymbol{H}^{-1/2}({\partial\Omega}_{D})\times \boldsymbol{H}^{1/2}({\partial\Omega}_{N}),
\end{align}
 is evidently continuous. 

We will further need the following extended system:
\begin{equation}
\widehat{\mathcal{M}}^{22}\mathcal{X}=\widehat{\mathcal{F}}^{22},
\end{equation}
where $\mathcal{X}=(p,\bs{v},\bs{\psi}, \boldsymbol{\varphi})^{\top}$, $ \widehat{\mathcal{F}}^{22} = ( 0, \bs{0}, \mathring{\bs{\mathcal{F}}}^{22}_{2},\mathring{\bs{\mathcal{F}}}^{22}_{3})^{\top}$ and 
\begin{equation}\label{ch2BHM22}
  \widehat{\mathcal{M}}^{22}=
  \left[ {\begin{array}{cccc}
  I\quad & 0\quad  &  -\mathring{{\Pi^s}} & \mathring{{\Pi^d}} \\
  0\quad  & \bs{I}\quad  & -\bs{\mathring{V}} & \bs{\mathring{W}} \\
   0\quad  & 0\quad & r_{{\partial\Omega}_{D}}\left(\dfrac{1}{2}\bs{I} - \mathring{\bs{\mathcal{W'}}}\right) & 
    r_{{\partial\Omega}_{D}}\mathring{\bs{\mathcal{L}}}  \\
    0\quad  & 0\quad  &  -r_{{\partial\Omega}_{N}}\mathring{\bs{\mathcal{V}}}  & r_{{\partial\Omega}_{N}}\left(\dfrac{1}{2}\bs{I} + \mathring{\bs{\mathcal{W}}}\right)
  \end{array} } \right].
\end{equation}
In virtue of Theorem \ref{ch2thM22inv} with $\mu =1$, the operator  
\begin{multline}\label{M*nar22hat} 
\widehat{\mathcal{M}}^{22}:\bs{H}^{1,0}(\Omega,\boldsymbol{\mathcal{A}})\times \widetilde{\boldsymbol{H}}^{-1/2}({\partial\Omega}_{D}) \times \widetilde{\boldsymbol{H}}^{1/2}({\partial\Omega}_{N})\\
\to \bs{H}^{1,0}(\Omega,\boldsymbol{\mathcal{A}})\times \boldsymbol{H}^{-1/2}({\partial\Omega}_{D})\times \boldsymbol{H}^{1/2}({\partial\Omega}_{N}).
\end{multline}
has a { continuous inverse.

\begin{theorem}\label{ch2thinvM22ring} Boundary integral operator \eqref{M022op} 
 is continuously invertible.
\end{theorem}
}
\begin{proof}
A solution of the system \eqref{ch2M22CC} with an arbitrary right hand side 
\begin{equation}
\mathring{\mathcal{F}}^{22} = [\widehat{\bs{\mathcal{F}}}^{22}_{2},\widehat{\bs{\mathcal{F}}}^{22}_{3}]^{\top}\in \boldsymbol{H}^{-1/2}({\partial\Omega}_{D})\times \boldsymbol{H}^{1/2}({\partial\Omega}_{N})
\end{equation}
is given by the pair $(\bs{\psi}, \boldsymbol{\varphi})$ which satisfies the following extended system:
\begin{equation}
\widehat{\mathcal{M}}^{22}\mathcal{X}=\widehat{\mathcal{F}}^{22},
\end{equation}
where $\mathcal{X}=(p,\bs{v},\bs{\psi}, \boldsymbol{\varphi})^{\top}$, $ \widehat{\mathcal{F}}^{22} = ( 0, \bs{0}, \mathring{\bs{\mathcal{F}}}^{22}_{2},\mathring{\bs{\mathcal{F}}}^{22}_{3})^{\top}$ and the operator $\widehat{\mathcal{M}}^{22}$ is defined by \eqref{ch2BHM22}.
Since operator \eqref{M*nar22hat}  has a { continuous inverse, this implies that operator \eqref{M022op} is surjective. Corollary \ref{ch2corBIE} implies that operator \eqref{M022op} is also injective and thus an isomorphism.}

\end{proof}

Let us now consider the operator $\mathcal{M}_{*}^{22}$, defined by \eqref{ch2M22M}, in wider spaces,
\begin{multline}\label{M*nar22w} 
\mathcal{M}_{*}^{22}:L^{2}(\Omega)\times\boldsymbol{H}^{1}(\Omega)\times \widetilde{\boldsymbol{H}}^{-1/2}({\partial\Omega}_{D})\times \widetilde{\boldsymbol{H}}^{1/2}({\partial\Omega}_{N})\\
\to L^{2}(\Omega)\times\boldsymbol{H}^{1}(\Omega)\times \boldsymbol{H}^{-1/2}({\partial\Omega}_{D})\times \boldsymbol{H}^{1/2}({\partial\Omega}_{N}).
\end{multline}
Theorems~\ref{ch2thmUR:theo}-\ref{ch2T3cal} imply that operator \eqref{M*nar22w} is continuous and we can now prove its invertibility.
\begin{theorem}\label{ch2thM22star} 
Operator \eqref{M*nar22w} is continuously invertible.
\end{theorem}

\begin{proof}
Let us consider the operator 
\begin{equation}\label{ch2tildeM22}
  \widetilde{\mathcal{M}}^{22}=
  \left[ {\begin{array}{cccc}
   I\quad & 0\quad  &  -{\Pi^s} & {\Pi^d} \\
   0\quad & \bs{I}& -\boldsymbol{V} & \boldsymbol{W} \\
   0 & 0& r_{{\partial\Omega}_{D}}\left(\dfrac{1}{2}\bs{I} - \bs{\mathring{\mathcal{W'}}}\right) & 
   r_{{\partial\Omega}_{D}}\bs{\widehat{\mathcal{L}}} \\
   0 & 0& -r_{{\partial\Omega}_{N}}\bs{\mathcal{V}}  & r_{{\partial\Omega}_{N}}\left(\dfrac{1}{2}\bs{I} + \bs{\mathcal{W}}\right)
  \end{array} } \right],
\end{equation}
{ which is a compact perturbation of the operator \eqref{M*nar22w} due to Theorems \ref{ch2thRcomp}, \ref{ch2T3cal} and Corollary \ref{ch2Lcompact}.} Using relations \eqref{ch2relationTV} and \eqref{ch2relationVW}, we can express the operator $\widetilde{\mathcal{M}}^{22}$ in  the form
\begin{equation}
\widetilde{\mathcal{M}}^{22}= \mbox{diag}\left(1, \dfrac{1}{\mu}\bs{I} , \bs{I}, \dfrac{1}{\mu}\bs{I} \right)\widehat{\mathcal{M}}^{22}\mbox{diag}(1, \mu\bs{I} , \bs{I}, \mu\bs{I} ) 
\end{equation}
where $\mbox{diag} (a, b\bs{I},c\bs{I}, d\bs{I})$ represents a 10 by 10 diagonal matrix
\begin{equation}
  \mbox{diag}(a, b\bs{I},c\bs{I}, d\bs{I})=
  \left[ {\begin{array}{cccc}
   a & 0 & 0 & 0 \\
   0& \bs{b} & 0 & 0\\
   0& 0 & \bs{c} & 0\\
   0& 0 & 0 & \bs{d}
  \end{array} } \right].
\end{equation}

The operator $\widehat{\mathcal{M}}^{22}$ defined by \eqref{ch2BHM22} can be understood as a triangular block matrix operator with the three following diagonal operators
\begin{align}
I&:\,\,L^{2}(\Omega^{+})\to L^{2}(\Omega^{+}),\label{7.54}\\
I&:\,\, \boldsymbol{H}^{1}(\Omega^{+})\to \boldsymbol{H}^{1}(\Omega^{+}),\label{7.55}\\
\mathring{\mathcal{M}}^{22}&:\,\, \widetilde{\boldsymbol{H}}^{-1/2}({\partial\Omega}_{D})\times \widetilde{\boldsymbol{H}}^{1/2}({\partial\Omega}_{N})\to \boldsymbol{H}^{-1/2}({\partial\Omega}_{D})\times \boldsymbol{H}^{1/2}({\partial\Omega}_{N}).\label{7.56}
\end{align}

By Theorem \ref{ch2thinvM22ring}, { operator \eqref{7.56}} is invertible. 
Consequently, 
\begin{multline}\label{Mhatnar22w} 
\widehat{\mathcal{M}}^{22}:
L^{2}(\Omega)\times\boldsymbol{H}^{1}(\Omega)\times \widetilde{\boldsymbol{H}}^{-1/2}({\partial\Omega}_{D}) \times \widetilde{\boldsymbol{H}}^{1/2}({\partial\Omega}_{N})\\
\to L^{2}(\Omega)\times\boldsymbol{H}^{1}(\Omega)\times \boldsymbol{H}^{-1/2}({\partial\Omega}_{D})\times \boldsymbol{H}^{1/2}({\partial\Omega}_{N}).
\end{multline} 
is an invertible operator as well. As $\mu$ is strictly positive, the diagonal matrices are invertible and the operator  
\begin{multline}\label{Mtilnar22w} 
\widetilde{\mathcal{M}}^{22}:
L^{2}(\Omega)\times\boldsymbol{H}^{1}(\Omega)\times \widetilde{\boldsymbol{H}}^{-1/2}({\partial\Omega}_{D}) \times \widetilde{\boldsymbol{H}}^{1/2}({\partial\Omega}_{N})\\
\to L^{2}(\Omega)\times\boldsymbol{H}^{1}(\Omega)\times \boldsymbol{H}^{-1/2}({\partial\Omega}_{D})\times \boldsymbol{H}^{1/2}({\partial\Omega}_{N}).
\end{multline} 
is also invertible. Thus, { operator \eqref{M*nar22w} is a zero-index} Fredholm operator.
{ Invertibility of this operator then follows from its injectivity implied by Theorem \ref{ch2thEQm11}(iii).}
\end{proof}

The last three vector equations of the system (M22) are segregated from $p$. Therefore, we can define the new system given by equations \eqref{ch2M22v}-\eqref{ch2M22N} which can be written { in the matrix form} as
\begin{equation}\label{ch2M22y}
\mathcal{M}^{22}\mathcal{Y} = \mathcal{F}^{22},
\end{equation}
where $\mathcal{Y}$ represents the vector containing the unknowns of the system 
\[\mathcal{Y}=(\bs{v},\bs{\psi},\bs{\phi})\in \bs{H}^{1}(\Omega)\times \widetilde{\bs{H}}^{-1/2}({\partial\Omega}_{D})\times \widetilde{\bs{H}}^{1/2}({\partial\Omega}_{N}),\]
and the matrix operator $\mathcal{M}^{22}$ is given by
\begin{align} \label{M22mat}
 \mathcal{M}^{22}:=
  \left[ {\begin{array}{ccc}
   \boldsymbol{I}+\boldsymbol{\mathcal{R}}& -\boldsymbol{V} & \boldsymbol{W} \\
    r_{{\partial\Omega}_{D}}\boldsymbol{T}^{+}( \mathcal{R}^{\bullet},  \boldsymbol{\mathcal{R}}) & r_{{\partial\Omega}_{D}}\left(\dfrac{1}{2}\boldsymbol{I}-\boldsymbol{\mathcal{W'}}\right) & r_{{\partial\Omega}_{D}}\boldsymbol{\mathcal{L}}^{+} 
    \\
   r_{{\partial\Omega}_{N}}\gamma^{+}\boldsymbol{\mathcal{R}}& -r_{{\partial\Omega}_{N}}\boldsymbol{\mathcal{V}} & r_{{\partial\Omega}_{N}}\left(\dfrac{1}{2}\boldsymbol{I} + \boldsymbol{\mathcal{W}}\right)
  \end{array} } \right].
  \end{align}
  
Following the reasoning similar to the proof of Theorem \ref{ch2corinvM11}, we obtain the following assertion.
 \begin{theorem}The operator 
\[ \mathcal{M}^{22}:\boldsymbol{H}^{1}(\Omega)\times \widetilde{\boldsymbol{H}}^{-1/2}({\partial\Omega}_{D})\times \widetilde{\boldsymbol{H}}^{1/2}({\partial\Omega}_{N}) \to \boldsymbol{H}^{1}(\Omega)\times \boldsymbol{H}^{-1/2}({\partial\Omega}_{D})\times \boldsymbol{H}^{1/2}({\partial\Omega}_{N}),\]
  is continuous and continuously invertible. 
 \end{theorem}
\subsection*{Acknowledgment}
This research was supported by the grants EP/H020497/1, EP/M013545/1, and 1636273 from
the EPSRC, UK, and also by the PhD Scholarship from Brunel University London for the second author.  
\bibliographystyle{plain}
\renewcommand{\bibname}{\Large References}
\begin{small}

\paragraph{S.E. Mikhailov$^1$, C.F. Portillo$^2$,\vspace{5pt}\\}
\begin{tabular}{ll}
$^1$ & {Department of Mathematics}  \\
&{Brunel University London} \\ 
&{Kingston Lane} \\ 
&{UB8 3PH, Uxbridge, Middlesex, UK.}\\
$^2$ & {Department of Mathematical Sciences}  \\
&{Oxford Brookes University} \\ 
&{Wheatley Campus } \\ 
&{OX33 1HX, Wheatley, Oxfordshire, UK.} \\  
\end{tabular}
\end{small}

\end{document}